\documentclass[colorlinks=true, linkcolor=blue, citecolor=blue, urlcolor=blue]{article}
\usepackage{orcidlink}
\usepackage{booktabs}

\usepackage{doi}
\usepackage[a4paper, top=1.2in, bottom=1.5in, left=1.3in, right=1.3in]{geometry}
\usepackage{boxedminipage}

\PassOptionsToPackage{colorlinks=true, linkcolor=blue, citecolor=blue, urlcolor=blue}{hyperref}

\usepackage{epsfig,amssymb,amsmath,version,amssymb,version,graphicx,fancybox,mathrsfs,bm}
\usepackage{amsthm}

\usepackage{bm}

\usepackage{subfigure}
\usepackage{subcaption}  
\usepackage{cleveref}
\usepackage{tikz}
\usetikzlibrary{arrows.meta, positioning}

\usepackage[numbers,sort&compress]{natbib}

\theoremstyle{plain}
\newtheorem{theorem}{Theorem}[section]
\newtheorem{lemma}[theorem]{Lemma}
\newtheorem{corollary}[theorem]{Corollary}
\newtheorem{proposition}[theorem]{Proposition}

\theoremstyle{definition}
\newtheorem{definition}[theorem]{Definition}
\newtheorem{example}[theorem]{Example}

\theoremstyle{remark}
\newtheorem{remark}{Remark}

\usepackage{algorithm}
\usepackage{algpseudocode}

\makeatletter
\newenvironment{breakablealgorithm}
  {
   \begin{center}
     \refstepcounter{algorithm}
     \hrule height.8pt depth0pt \kern2pt
     \renewcommand{\caption}[2][\relax]{
       {\raggedright\textbf{\ALG@name~\thealgorithm} ##2\par}
       \ifx\relax##1\relax 
         \addcontentsline{loa}{algorithm}{\protect\numberline{\thealgorithm}##2}
       \else 
         \addcontentsline{loa}{algorithm}{\protect\numberline{\thealgorithm}##1}
       \fi
       \kern2pt\hrule\kern2pt
     }
  }{
     \kern2pt\hrule\relax
   \end{center}
  }
\makeatother
\usepackage{tikz}
\usetikzlibrary{positioning}
\usetikzlibrary{chains}
\usepackage{tikz-cd}
\usepackage{tikz}
\usepackage{pgfplots}
\pgfplotsset{compat=1.18}
\usetikzlibrary{calc,intersections}

\usetikzlibrary{shapes.geometric, arrows, positioning}

\newcommand{\dmax}{\Delta_{\max}}
\newcommand{\xb}{\boldsymbol{x}}
\newcommand{\yb}{\boldsymbol{y}}
\newcommand{\xbb}{\boldsymbol{x}}
\newcommand{\cB}{\mathcal{B}}
\newcommand{\cQ}{\mathcal{Q}}
\newcommand{\cX}{\mathcal{X}}
\newcommand{\R}{\mathbb{R}}  
\newcommand{\y}{\boldsymbol{y}}
\newcommand{\g}{\boldsymbol{g}}
\newcommand{\dd}{\boldsymbol{d}}

\newcommand{\W}{\boldsymbol{W}}
\newcommand{\Hb}{\boldsymbol{H}}

\allowdisplaybreaks[4] 

\usepackage{amsmath}
\usepackage{amssymb}

\usepackage{amsfonts}
\usepackage{contour}
\newtheorem{assumption}{Assumption}

\usepackage{orcidlink}
 \allowdisplaybreaks[4] 

 \usepackage{makecell}

\title{ReMU: Regional Minimal Updating for Model-Based Derivative-Free Optimization}
\author{
    Pengcheng Xie\,\orcidlink{0000-0001-5973-1535}\thanks{Corresponding author: P. Xie, \texttt{pxie@lbl.gov}} \quad
    Stefan M. Wild\,\orcidlink{0000-0002-6099-2772}
}

\date{
    Applied Mathematics and Computational Research Division,\\
    Lawrence Berkeley National Laboratory, Berkeley, CA 94720, USA
}

\begin{document}

\maketitle

\begin{abstract}
Derivative-free optimization (DFO) problems are optimization problems where derivative information is unavailable or extremely difficult to obtain. Model-based DFO solvers have been applied extensively in scientific computing. Powell's {\ttfamily NEWUOA} (2004) \cite{MJDP06} and Wild's {\ttfamily POUNDerS} (2014) \cite{SWCHAP14} explore the numerical power of the minimal norm Hessian (MNH) model for DFO and contributed to the open discussion on building better models with fewer data to achieve faster numerical convergence. Another decade later, we propose the regional minimal updating (ReMU) models, and extend the previous models into a broader class. This paper shows motivation behind ReMU models, computational details, theoretical and numerical results on particular extreme points and the barycenter of ReMU's weight coefficient region, and the associated KKT matrix error and distance. Novel metrics, such as the truncated Newton step error, are proposed to numerically understand the new models' properties. A new algorithmic strategy, based on iteratively adjusting the ReMU model type, is also proposed, and shows numerical advantages by combining and switching between the barycentric model and the classic least Frobenius norm model in an online fashion.  
\end{abstract}


\section{Introduction}
\label{Introduction}

In many real-world settings, optimization objective functions can be expensive to evaluate and do not have readily available derivatives. These problems arise when objective functions are expressed not as an algebraic function, but by querying a ``black box'' (or ``zeroth-order oracle''). Black boxes and oracles include data from the realization of a chemical process or output from a computer simulation and are typically addressed by derivative-free optimization algorithms  \cite{Conn2009a,AudetHare2017,alarie2021two,LMW2019AN}. In this paper, we address such unconstrained problems of the form
\begin{equation}
\label{obj-unconstraint}
\min_{\xbb\in\R^n}\ f(\xbb),
\end{equation}
where the objective function $f$ is a deterministic black box, which one has prior belief possesses some degree of smoothness. 

Derivative-free algorithms for such problems are reviewed in \cite{AudetHare2017, Kelleybook, Conn2009a, LMW2019AN} and include line-search, direct-search, and model-based methods. 
The model-based methods discussed in this paper typically use a trust-region framework for selecting new iteration points \cite{MJDP06,Liuzzi2019,Rinaldi2024,Xie2023derivative}. One example of such models is polynomials, including linear interpolation \cite{Spendley1962}, quadratic interpolation  \cite{Winf69,  Winf73}, underdetermined quadratic interpolation \cite{Zhang2014,Powell2003a, SW08, Ragonneau2024, Xie2024DFO, Xie2023derivative, Xie2024new2D,CRV2008}, and regression models \cite{Conn2006b}. Other forms of models include radial basis function interpolation \cite{SWRRCS07,Bjorkman2000a,SMWCAS13} or interpolation by Gaussian processes, a typical approach in Bayesian optimization  \cite{Bayes01}. 

Central to most model-based approaches is a model $m:\R^n\rightarrow \R$ that interpolates the objective function in \eqref{obj-unconstraint} on an interpolation set $\cX$. We formalize this condition as follows:
\begin{equation}
\label{eq:interp}
m(\xb) = f(\xb), \qquad \forall \, \xb \in \cX 
\tag{Interp$(m,f,\cX)$}.
\end{equation}

This paper explores the use of a new family of models --- the regional minimal updating (ReMU) quadratic model --- in model-based derivative-free trust-region methods.  Building off the work of \cite{xie2023}, ReMU models will make use of norms based on zeroth-, first-, and second-order information in a local region as defined next.

\begin{definition}\label{def42}
Let \(u\) be a function over \(\Omega \subseteq {\R}^n\). If \(u\) is twice continuously differentiable on \(\Omega\),  
then we define  
\[
\begin{aligned}
   \vert u \vert_{H^{0}(\Omega)} & := \left(\int_{\Omega} \vert u(\boldsymbol{x})\vert^2 \, {\rm d}\boldsymbol{x}  \right)^{\frac{1}{2}}, \\ 
   \vert u \vert_{H^{1}(\Omega)} & := \left(\int_{\Omega} \| \nabla u(\boldsymbol{x}) \|_2^2 \, {\rm d}\boldsymbol{x} \right)^{\frac{1}{2}}, \\ 
   \vert u \vert_{H^{2}(\Omega)} & := \left(\int_{\Omega} \| \nabla^2 u(\boldsymbol{x}) \|_F^2 \, {\rm d}\boldsymbol{x} \right)^{\frac{1}{2}}.
\end{aligned}
\]
These norms and weights $C_1, C_2, C_3$ allow us to define the weighted norm foundational to ReMU models:  
\begin{equation}
\label{eq:weighted}
\vert u \vert_{CH(\Omega)} :=  \left(\sum_{i=1}^{3} C_i | u |_{H^{i-1}(\Omega)}^2\right)^{\frac{1}{2}}.
\end{equation}
  \end{definition}
In Definition~\ref{def42} we have used the Frobenius norm defined by $\|\boldsymbol{C}\|_F = ({\sum_{i=1}^m \sum_{j=1}^n\left|c_{i j}\right|^2})^{{1}/{2}}$ for a matrix \(\boldsymbol{C}\in\R^{m\times n}\) with elements \(c_{ij}, 1\le i\le m, 1\le j\le n\). 
With these norms and the interpolation condition \ref{eq:interp} in place, we define the models that are the subject of this paper as follows. 

\begin{definition}[Regional minimal updating quadratic model (ReMU model)]\label{ReMUdef}
Given coefficients $C_1,C_2,C_3\ge 0$ with $C_1+C_2+C_3= 1$ and the region \(\Omega \subset \mathbb{R}^n\), the ReMU quadratic model is  
the quadratic model solving 
  \begin{equation}
  \label{eq:weightproblem}
\mathop{\min}\limits_{m\in\cQ}\left\{ \sum_{i=1}^{3} C_{i}\vert m-m_{k-1}\vert^2_{H^{i-1}{(\Omega)}}  \, : \, m \mbox{ satisfies } \hyperref[eq:interp]{{\rm Interp}(m,f,\cX_k)} \right\},
\end{equation}
where $\cQ$ is the space of quadratics (i.e., polynomials with order not larger than two), $\mathcal{X}_k:=\{\y_1,\dots,\y_{|\cX_k|}\}$ is the interpolation set at the $k$th iteration, and $m_{k-1}$ denotes the quadratic model at the $(k-1)$th iteration. 
\label{def:ReMU}
\end{definition}

It is natural to ask whether this definition is well posed. A starting point is to ask whether a feasible solution to \eqref{eq:weightproblem} exists (i.e., whether a quadratic model exists satisfying \hyperref[eq:interp]{Interp$(m,f,\cX_k)$)}, which induces conditions on the interpolation set $\cX_k$ and/or the function $f$. When the function $f$ is quadratic this existence, as well as uniqueness, can be directly obtained based on the analysis in \cite{xie2023}. When $f$ is a more general function and $n>1$, things are decidedly more complicated and induce geometric conditions on the interpolation set $\cX_k$ \cite{Wendland}. When $f$ is arbitrary, simple counterexamples include the cases where the interpolation set is too large ($|\cX_k|>\frac{1}{2}(n+1)(n+2)$, corresponding to more than the degrees of freedom of a quadratic) or the interpolation is more than affine ($|\cX_k|> n+1$) but lies in a proper subspace of $\R^n$.
 
In this paper, we will work with interpolation sets for which Definition~\ref{def:ReMU} is well posed. A motivation for considering ReMU models is that they have the following projection property.

\begin{theorem}
Given a quadratic objective function \(f\), if \(m_k\) is the solution of \eqref{eq:weightproblem}, then it satisfies
\begin{equation}
\label{proj} 
\sum_{i=1}^{3} C_{i}\left|m_{k}-f\right|^{2}_{H^{i-1}\left(\Omega\right)}
\leq \sum_{i=1}^{3} C_{i}\left|m_{k-1}-f\right|^{2}_{H^{i-1}\left(\Omega\right)}.  
\end{equation}   
\end{theorem}

\begin{proof}
For $\xi \in {\R}$, we define \(m_{\xi} := m_{k}+\xi \left(m_{k}-f\right)\) and observe that \(m_{\xi}\) is a quadratic that satisfies the interpolation conditions \hyperref[eq:interp]{Interp$(m_{\xi},f,\cX_k)$}. Hence  
\(
\varphi(\xi):=\sum_{i=1}^{3}C_{i}\left|m_{\xi}-m_{k-1}\right|_{H^{i-1}(\Omega)}^{2}
\) 
is minimized at \(\xi=0\). We also have that
\begin{equation}
\label{eq:above}
\begin{aligned}
\varphi(\xi)
=&\sum_{i=1}^{3}C_{i}\left|m_{k}+\xi \left(m_{k}-f\right)-m_{k-1}\right|_{H^{i-1}(\Omega)}^{2}\\
=&\,\xi^{2} \sum_{i=1}^{3}C_{i}\left| m_{k}-f \right|_{H^{i-1}(\Omega)}^{2}  +
\sum_{i=1}^{3}C_{i}\left| m_{k}-m_{k-1} \right|_{H^{i-1}(\Omega)}^{2} \\
&+2 \xi \bigg(C_1\int_{\Omega}\left[\left(m_{k}-m_{k-1}\right)(\xb)\right] \cdot\left[\left(m_{k}-f\right)(\xb)\right] {\rm d} \xb\\
&+C_2\int_{\Omega}\left[\nabla(m_{k}-m_{k-1})(\xb)\right]^{\top}\left[\nabla(m_{k}-f)(\xb)\right] {\rm d} \xb \\
&+C_3\int_{\Omega} (1,\ldots,1)
\left[\nabla^{2}(m_{k}-m_{k-1})(\xb)\right] \circ \left[\nabla^{2}(m_{k}-f)(\xb)\right](1,\ldots,1)^{\top} {\rm d} \xb\bigg),
\end{aligned}
\end{equation}
where the symbol \(\circ\) denotes Hadamard product. Because $\varphi$ is minimized at \(\xi=0\), ${\partial \varphi}/{\partial \xi}$ must vanish at \(\xi=0\), and hence the expression in the larger parentheses in \eqref{eq:above} must vanish. The theorem follows by observing that 
\[
\begin{aligned}
\varphi(-1)
=&\sum_{i=1}^{3}C_{i}\left|f-m_{k-1}\right|_{H^{i-1}(\Omega)}^{2} \\
= &\sum_{i=1}^{3}C_{i}\left| m_{k}-f \right|_{H^{i-1}(\Omega)}^{2}  +
\sum_{i=1}^{3}C_{i}\left| m_{k}-m_{k-1} \right|_{H^{i-1}(\Omega)}^{2}.
\end{aligned}
\]
\end{proof} 
 
Different convex combinations (defined by the weight coefficients $C_1, C_2, C_3$ give rise to different projection properties, all of which follow from \eqref{proj}. In this paper, we seek to understand ReMU models defined by different weight coefficients.

\begin{sloppypar}
{\bf Organization.} The rest of this paper is organized as follows. In \Cref{Model-based method using ReMU models}, we introduce a model-based method using ReMU models.  
We also provide the formulae for obtaining regional minimal updating models as well as an initial set of numerical results examining one measure of quality -- based on a truncated Newton step -- of the algorithm using ReMU models.  
\Cref{Regional minimal updating} provides more details on the regional minimal updating, focusing on the KKT matrix error and distance, and introduces the geometric points of the coefficient region. \Cref{Model-based methods using the corrected ReMU models} presents a model-based method that allows one to adaptively incorporate corrected ReMU models. This section presents a new weight correction step, a more advanced model-based algorithm, and numerical results demonstrating the impact of using corrected ReMU models. We conclude in \Cref{Conclusions} with a summary and potential directions for future research.
\end{sloppypar}

\section{Model-based method using ReMU models}
\label{Model-based method using ReMU models}

We begin by specifying a model-based algorithm using ReMU models before providing the formulae for regional minimal updating and some basic numerical tests using different ReMU models.

\subsection{A model-based algorithm using ReMU models}
\label{A model-based algorithm using ReMU models}

Our framework of model-based derivative-free trust-region algorithms with the ReMU model is stated in Algorithm~\ref{algo-TR}. Trust-region algorithms (see, e.g., \cite{LMW2019AN,Conn2009a}) operate using a ``trust region''
\begin{equation}
\label{eq:tr}
\cB_{2}({\xb}_{k}, \Delta):=\left\{\xb: \left\|\xb-{\xb}_{k}\right\|_2\le \Delta \right\},
\end{equation}
centered about the point ${\xb}_{k}$ and with a radius $\Delta>0$. Although we have made the norm explicit in \eqref{eq:tr}, in what follows we assumed that the norm \(\|\cdot\|=\|\cdot\|_2\) unless otherwise indicated.

The quadratic models used at the \(k\)th step in the algorithm are the regional minimal updating quadratic models obtained by solving the subproblem 
\eqref{eq:weightproblem} with 
\(\Omega=\cB_{2}(\xb_{k}, r)\)
where $r=\Delta_k$, the \(k\)th trust-region radius\footnote{The parameter \(r\) can be chosen in other ways for different problems, for instance, another choice is $r=\max\{10\Delta_k,\max_{\xb\in\mathcal{X}_k}\Vert \xb-\xb_{k}\Vert_2\}$, depending on the trust-region radius.}. 

\begin{breakablealgorithm} 
\caption{Framework of model-based derivative-free trust-region algorithms\label{algo-TR}} 
\begin{algorithmic}[1]
\State \textbf{Input}: the initial point \(\xb_{0}\) and 
initialize the interpolation set \(\mathcal{X}_0\), the initial quadratic model \(m_{0}\),  and the parameters \(\dmax\geq\Delta_{0}>0, \gamma>1, \varepsilon_c>0, \mu>0, 0\le\eta_1\le\eta_2\le 1\); set counter $k=0$.
\State \textbf{Step 1} (\textbf{Criticality step}):  
If $\left\|\nabla m_{k}(\xb_{k})\right\| \leq \varepsilon_{c}$, then: (i) If $m_k$ is fully linear on $\cB_2(\xb_k,\Delta_k)$ and $\Delta_{k}\le\mu\left\|\nabla m_k(\xb_{k}) \right\|$, exit the algorithm with $\xb_k$; (ii) otherwise, set $\Delta_{k}={1}/{\gamma}\Delta_k$, perform model improving to make $m_k$ fully linear on $\cB_2(\xb_k,\Delta_k)$, and return to {\bf Step 1}. 
\State \textbf{Step 2} (\textbf{Trial step}):  
Obtain \(\dd_k\) by solving the trust-region subproblem 
\[
\min_{\dd} \left\{ m_k({\xb}_{k}+ \dd)
\, : \, \| \dd \|_{2} \leq \Delta_k\right\}. 
\]

\State \textbf{Step 3} (\textbf{Acceptance of the trial point}):  
Compute \(f(\xb_{k}+\dd_{k})\) and  
\[
\rho_{k}:=\frac{{f(\xb_{k})-f(\xb_{k}+\dd_{k})}}{{m_{k}(\xb_{k})-m_{k}(\xb_{k}+\dd_{k})}},
\]  
and update center
\begin{equation}
\label{updatecenter}
\xb_{k+1} = \begin{cases}
\xb_{k}+\dd_{k}, & \mbox{if } \rho_{k} \geq \eta_1, \\ 
\xb_{k}, & \mbox{otherwise}\\
\end{cases}
\end{equation}
and trust-region radius
\begin{equation}\label{updateradius}
\Delta_{k+1} = \begin{cases}
\min\left\{\gamma \Delta_{k} , \dmax \right\}, & \mbox{if } \rho_{k} \ge \eta_2, \\
\frac{1}{\gamma}\Delta_{k}, & \mbox{if } \rho_{k} <\eta_1 \mbox{ and } m_k \mbox{ fully linear on } \cB_2(\xb_k,\Delta_k),\\
\Delta_{k}, & \mbox{otherwise}.\\
\end{cases}
\end{equation}

\State \textbf{Step 4} (\textbf{Model formation and possible improvement step}):  
Update interpolation set via 
\begin{equation}\label{updateset}\mathcal{X}_{k+1}=\mathcal{X}_{k}\cup\{\xb_{k}+\dd_k\}\backslash\{\xb_{t}\},
\end{equation}
where $\xb_t$ is the point in $\cX_{k}$ farthest from $\xb_{k+1}$. Use $\cX_{k+1}$ to build the quadratic interpolation model by solving the ReMU subproblem
\eqref{eq:weightproblem} with 
\(\Omega=\cB_{2}(\xb_{k+1}, \Delta_{k+1})\) 
to obtain $m_{k+1}$. If $m_{k+1}$ is not fully linear on \(\cB_{2}(\xb_{k+1}, \Delta_{k+1})\) and \(\rho_{k}<\eta_1\), improve the quality of model $m_{k+1}$ by a suitable improvement step, updating both \(m_{k+1}\) and \(\mathcal{X}_{k+1}\).  
\State Increment $k$ by one and go to {\bf Step 1}. 
  \end{algorithmic} 
\end{breakablealgorithm}

Algorithm~\ref{algo-TR} employs conditions based on the model's accuracy (i.e., depending on whether the current model is fully linear) to maintain straightforward global convergence \cite{Conn2009a}. However, in the implementation tested here, we are trying to heighten the influence on the algorithm given by the models and drop such geometry tests.  
Such an idea has been inspired by works such as \cite{GFJMJN09} and the trust-region model self-correction mechanism of \cite{Scheinberg10}. 
Our implementation of Algorithm~\ref{algo-TR} can be downloaded online\footnote{\href{https://github.com/pxie98/ReMU}{https://github.com/pxie98/ReMU}}. 
{We underscore that the implementation codes for the interpolation model performance comparison in \Cref{Numerical results of the algorithm with the ReMU models} is a simplified one, for which trust region updates are defined exclusively by the value of \(\rho_{k}\) instead of whether the model is fully linear.} This allows us to more fully examine the role that model type plays.

We also note that incorporating the geometry condition is straightforward. One way to detect whether the model is fully linear is to check the determinant value of the KKT matrix in the KKT conditions of subproblem
\eqref{eq:weightproblem} (discussed at the end of \Cref{Formula of regional minimal updating}). The geometry of the points can also be improved, for example, by looking for a new interpolation point that maximizes the determinant value of the KKT matrix.

\subsection{Formula of regional minimal updating} 
\label{Formula of regional minimal updating}

We now detail the process of forming a ReMU model, especially focusing on the associated KKT matrix.
We denote the $k$th quadratic model by 
\begin{equation*}
m_k({\xb})=\frac{1}{2}({\xb}-{\xb}_{k})^{\top}\Hb_k({\xb}-{\xb}_{k})+{\g}_k^{\top}({\xb}-{\xb}_{k})+c_k
\end{equation*} 
and the difference between successive models by
\begin{equation*}
m_{k}({\xb})-m_{k-1}({\xb})=D_k({\xb}):=\frac{1}{2}({\xb}-{\xb}_{k})^{\top}\hat{\Hb}(\xb-{\xb}_{k})+\hat{{\g}}^{\top}({\xb}-{\xb}_{k})+\hat{c}.
\end{equation*}

The following lemma from \cite{xie2023} shows the basic computation results. 

\begin{lemma}
\label{theorem1}
Given the quadratic function 
$m(\xb)=\frac{1}{2}(\xb-{\xb}_{k})^{\top}\Hb(\xb-{\xb}_{k})+(\xb-{\xb}_{k})^{\top}\g+c$, 
it holds that
\begin{equation}\label{calculationh2}
\begin{aligned}
\|m\|_{H^{2}(\cB_2({\xb}_{k},r))}^{2}  
=\, & \mathcal{V}_{n} r^{n} \Bigg[ \left(\frac{r^{4}}{2(n+4)(n+2)}+\frac{r^{2}}{n+2}+1\right)\|\Hb\|_{F}^{2}+\left(\frac{r^{2}}{n+2}+1\right)\|\g\|_{2}^{2}\\
  & +\frac{{r^{4}}}{4 (n+4)(n+2)} (\operatorname{Tr}(\Hb))^{2} +\frac{ r^{2}}{n+2} c\operatorname{Tr} (\Hb) + c^{2}\Bigg],
\end{aligned}
\end{equation}
where $\operatorname{Tr}(\cdot)$ denotes the trace of a matrix and \(\mathcal{V}_{n}\) denotes the volume of the $n$-dimensional unit ball  \(\cB_2({\xb}_{k},1)\). 
\end{lemma}

We can use Lemma~\ref{theorem1} to obtain the parameters of the regional minimal updating quadratic model by solving the problem
\begin{equation}
\label{lagobj}
\begin{aligned}
\min_{\hat{c},\hat{\g},\hat{\Hb}=\hat{\Hb}^\top}\ &\eta_{1}\|\hat{\Hb}\|_{F}^{2}+\eta_{2} \| \hat{\g}\|_{2}^{2}+\eta_{3}(\operatorname{Tr} (\hat{\Hb}))^{2}+\eta_{4} \operatorname{Tr} (\hat{\Hb})\hat{c}+\eta_5\hat{c}^{2}\\
\text{subject to}\ &\hat{c}+\hat{\g}^{\top}\left(\y_{i}-{\xb}_{k}\right)+\frac{1}{2}\left(\y_{i}-{\xb}_{k}\right)^{\top} \hat{\Hb} \left(\y_{i}-{\xb}_{k}\right)=f(\y_{i})-m_{k-1}(\y_i),\\
&\forall \, \yb_i \in \cX_k, 
\end{aligned}
\end{equation}
where the solution of \eqref{lagobj} is the difference function of the models (i.e., $D_k$). Points $\y_1,\dots,\y_{|\cX_k|}$ denote the interpolation points at the $k$th iteration for simplicity. The parameters \(\eta_1,\eta_2,\eta_3,\eta_4,\eta_5\) satisfy that
\begin{equation}\label{eta} 
\begin{aligned}
\eta_1&=C_1\frac{r^{4}}{2(n+4)(n+2)}+C_2\frac{r^{2}} {n+2}+C_3,\ 
\eta_2=C_1\frac{r^{2}}{n+2}+C_2,\\ 
\eta_3&=C_1\frac{r^{4}}{4(n+4)(n+2)},\ 
\eta_4=C_1\frac{r^{2}}{n+2},\  
\eta_5=C_1,
\end{aligned} 
\end{equation}
where $r=\Delta_k$, the trust-region radius. 
According to the KKT conditions of the subproblem \eqref{lagobj}, the ReMU model is 
\begin{equation}
\label{ReMUmodelexp}
m(\xb)=\frac{1}{2}(\xb-{\xb}_{k})^{\top}\Hb(\xb-{\xb}_{k})+(\xb-{\xb}_{k})^{\top}\g+c,
\end{equation}
where 
\begin{equation*}
\label{Q(x)G}
\begin{aligned}
 \Hb & = \frac{1}{2 \eta_{1}} \left(\frac{1}{2}\sum_{l=1}^{|\cX_k|} \lambda_{l}\left(\y_{l}-{\xb}_{k}\right)\left(\y_{l}-{\xb}_{k}\right)^{\top} -\left(2\eta_{3} T+\eta_{4}c\right) \boldsymbol{I}\right),\\
T&=\frac{1}{2\left(2 n \eta_{3}+2 \eta_{1}\right)} \sum_{l=1}^{|\cX_k|} \lambda_{l}\left\| \y_{l}-{\xb}_{k}\right\|_2^2-\frac{n \eta_{4}}{2 n \eta_{3}+2 \eta_{1}} c, 
\end{aligned}
\end{equation*}
and \(\boldsymbol{\lambda},\hat{c},\hat{\g}\) come from the solution of 
\begin{equation}\label{system1-1}
\underset{\text{KKT matrix }\W}{\underbrace{\overbrace{
\begin{pmatrix}
 \boldsymbol{A} & 
  \boldsymbol{J}
  &\boldsymbol{X}  \\
\boldsymbol{J}^{\top}   & \frac{n \eta_{4}^{2}}{2 n \eta_{3}+2 \eta_{1}}-2\eta_5 & \boldsymbol{0}_n^{\top}  \\
\boldsymbol{X}^{\top} & \boldsymbol{0}_n &  -2\eta_2\boldsymbol{I}
\end{pmatrix}}^{|\cX_k|+1+n}}}
\begin{pmatrix}
\boldsymbol{\lambda} \\
\hat{c} \\
\hat{\g}
\end{pmatrix}=
\begin{pmatrix}
0\\
\vdots \\
0\\
f(\xb_\text{new})-m_{k-1}{(\xb_\text{new})}\\
0\\
\vdots \\
0
\end{pmatrix},
\end{equation}
where \(\xb_{\text{new}}\) denotes the newest interpolation point at the $k$th iteration and it is one of the points in \(\{\y_1,\dots,\y_{|\cX_k|}\}\), $\boldsymbol{0}_n=(0,\dots,0)^{\top}\in{\R}^{n}$, and $\boldsymbol{\lambda}=(\lambda_1,\dots,\lambda_{|\cX_k|})^{\top}$. Additionally, $\boldsymbol{I}\in{\R}^{n\times n}$ is the identity matrix, and \(\boldsymbol{A}\) has the elements
\begin{equation*}
\begin{aligned}
&\boldsymbol{A}_{i j}=\frac{1}{{8}\eta_1}\left(\left({\y}_{i}-{\xb}_{k}\right)^{\top}\left({\y}_{j}-{\xb}_{k}\right)\right)^{2}-\frac{\eta_{3}}{8\eta_{1}\left(n \eta_{3}+\eta_{1}\right)} \left\|\y_{i}-{\xb}_{k}\right\|_2^2\left\|\y_{j}-{\xb}_{k}\right\|_2^2,
\end{aligned}
\end{equation*}
where \(1 \leq i, j \leq |\cX_k|\). It also holds that
\[
\boldsymbol{X}=\left(\y_1-{\xb}_{k}, \y_2-{\xb}_{k},\ldots,\y_{|\cX_k|}-{\xb}_{k}\right)^{\top},
\]  
and
\begin{equation*}
\boldsymbol{J}=\left(
  1-\frac{\eta_{4}}{4 n \eta_{3}+4 \eta_{1}}\left\|\y_{1}-{\xb}_{k}\right\|_2^2,
  \ldots,
 1-\frac{\eta_{4}}{4 n \eta_{3}+4 \eta_{1}}\left\|\y_{|\cX_k|}-{\xb}_{k}\right\|_2^2
 \right)^{\top}.
 \end{equation*}

We denote the matrix on the left-hand side of \eqref{system1-1} as KKT matrix \(\W\). The following remarks show that two classic least norm types of underdetermined quadratic models are actually the special cases within the ReMU class of models, with special weight coefficients in the subproblem for obtaining the interpolation function.

\begin{remark}
If \(C_{1}=C_2=0\) and \(C_{3}=1\), then $\eta_1=1, \eta_2=\eta_3=\eta_4=\eta_5=0$, and the KKT matrix is exactly the KKT matrix of the least Frobenius norm updating quadratic model \cite{MJDP06}. 

If we obtain the \(k\)th model by solving the problem
\begin{equation*}
\begin{aligned}
\underset{m}{\operatorname{\min}}\ &\left\Vert \nabla^{2} m-\nabla^{2} m_{k-1}\right\Vert_{F}^{2}+\sigma\left\Vert \nabla m-\nabla m_{k-1}\right\Vert_{2}^{2} \\ 
\text{subject to}  \ &m(\y_i)=f(\y_i), \ \y_i \in \mathcal{X}_{k},
\end{aligned}
\end{equation*}
where the weight coefficient $\sigma\ge 0$, then $\W$ is changed to the matrix
\begin{equation*}
\W=\begin{pmatrix}
\boldsymbol{A} & \boldsymbol{E} & \boldsymbol{X} \\
\boldsymbol{E}^{\top} & 
0 & \boldsymbol{0} \\
\boldsymbol{X}^{\top} & \boldsymbol{0} &-2{\sigma}\boldsymbol{I}
\end{pmatrix}.
\end{equation*}
\end{remark}

Given continuously differentiable objective functions, the local approximation error of the ReMU model's gradient and function values can be readily upper bounded; in the case where the objective function is bounded, the model Hessians are bounded, and a subset of $n+1$ points are sufficiently affinely independent, ReMU models can be shown to be fully linear \cite{Conn2009a} following standard results for underdetermined quadratic interpolation models. Additional details on certifying a ReMU model as fully linear follow the discussion in  \cite[Theorem~4.2]{xie2023}, and are highly related to bounding the condition number of the KKT matrix that yields the models.

\subsection{Numerical results of the algorithm with the ReMU models}
\label{Numerical results of the algorithm with the ReMU models}

We now present initial numerical investigations of ReMU models with different weight coefficients. 

\subsubsection{A numerical example}
\label{A numerical example} 

\begin{example}
\label{example-1} 
As a first example, we apply Algorithm~\ref{algo-TR} using ReMU models with different weight coefficients to minimize the 2-dimensional {\ttfamily Rosenbrock} function
\begin{equation}
\label{rosen}
f(\y)=(1 - y_1)^2+100(y_{2} - y_1^2)^2,
\end{equation}
which is a smooth nonconvex function, and where $\y=(y_1,y_2)^{\top}$. The minimizer of \eqref{rosen} lies on a narrow curved valley.  

This numerical example uses the following settings. The corresponding results can be regarded as a reasonable comparison of the behavior of derivative-free algorithms with the ReMU models containing different weight coefficients. The initial input point is set as \(\y_{0}=(1.04, 1.1)^{\top}\). In each iteration, we use 5 points to interpolate. The maximum number of function evaluations is 16, and the initial trust-region radius is $\Delta_0=10^{-4}$. This example probes the case when the trust-region radius is small, which is related to the analysis in \Cref{Regional minimal updating}, The tolerances of trust-region radius, function value, and the gradient norm are all set as $10^{-8}$; other parameters in Algorithm~\ref{algo-TR} are set as $\gamma=2, \eta_1={1}/{4}, \eta_2={3}/{4}$, $\mu=0.1$. 
The initial interpolation points are \(\y_{0}\), \(\y_{0}\pm (\Delta_0,0)^{\top}\), and \(\y_{0}\pm (0,\Delta_0)^{\top}\). 

Table~\ref{diffnorm} lists the seven different (semi-)norms for the weight coefficients we examine. 
\begin{table}[htbp]
\centering
\caption{Different (semi-)norms with corresponding coefficients defining different members of the ReMU model class and results on the 2-dimensional Rosenbrock example.
\label{diffnorm}}       
\begin{tabular}{cccc}
   \toprule 
  \makecell{Weights  \\ \(C_1\), \(C_2\), \(C_3\)} & \makecell{Corresponding \\ (semi-)norms} & \makecell{Function \\ value} & \makecell{Obtained \\ solution}\\
   \midrule
\(1,\ 0,\ 0\) &  \(H^0\) norm (\(L^2\) norm) &\(0.0147\) &\((1.0426,  1.0984)^{\top}\)\\
 \(0,\ 1,\ 0\) &  \(H^1\) semi-norm  &\(0.0193\)&\((1.0421,  1.0992)^{\top}\)\\
\(0,\ 0,\ 1\) &  \(H^2\) semi-norm (Hessian Frob.)   &\(0.0078\)&\((1.0455,  1.1008)^{\top}\)\\
\({1}/{3},\ {1}/{3},\ {1}/{3}\) &  \(H^2\) norm     &\(  0.0031\) & \((1.0495,  1.1040)^{\top}\)\\
\({1}/{2},\ {1}/{2},\ 0\) &  \(H^1\) norm   &\(0.0203\)  &\((1.0418,  1.0990)^{\top}\)\\ 
\(0,\ {1}/{2},\ {1}/{2}\)&  \(H^1\) semi-norm + \(H^2\) semi-norm   &\(0.0169\) &\((1.0427,  1.0996)^{\top}\)\\
\({1}/{2},\ 0,\ {1}/{2}\) &  \(H^0\) norm + \(H^2\) semi-norm&\(0.0242\)  &\((1.0412,   1.0991)^{\top}\)\\ 
   \bottomrule
\end{tabular}
\end{table}

Table~\ref{diffnorm} shows the results of this numerical experiment, including the obtained minimizer, and the best function value corresponding to type of ReMU model. We observe in Table~\ref{diffnorm} that the ReMU model with the weight coefficients \(C_1=C_2=C_3={1}/{3}\) achieves a smaller function value than the other model types. 
\end{example}

\subsubsection{Performance and date profiles for test set}
\label{Performance and date profiles for test set}

To show the general numerical behavior of the algorithms based on the ReMU models, we aggregate results using performance \cite{EDD01} and data \cite{JJMSMW09} profiles. 

Our test set is the classic benchmark set from  \cite{JJMSMW09}, which includes 53 unique problems, of dimension between 2 and 12, with 10 different forms (530 problems in total). 
The 10 forms are 
smooth problems (of the form \(f(\xb) = \sum_{i=1}^{p} F_i(\xb)^2\)), potentially nondifferentiable problems (of the form \(f(\xb) = \sum_{i=1}^p | F_i(\xb) |\)), three deterministic noisy problems based on the smooth function \(f\) (of the form \(f(\xb) + \phi(\xb)\), where \(\phi(\xb)\) is a deterministic oscillatory function \cite{JJMSMW09}, \(f(\xb)  (1 + 10^{-3}  \phi(\xb))\), where \(\phi(\xb)\) is a deterministic oscillatory function, \(f(\xb)  (1 + \sigma  \phi(\xb))\), where \(\phi(\xb)\) is a deterministic oscillatory function, and \(\sigma\) controls the noise level), additive stochastic noisy versions (of the form  \(f(\xb) = \sum_{i=1}^p (F_i(\xb) + \sigma z)^2\)) with stochastic (Gaussian and uniform) noise controlled by \(\sigma\),  and relative stochastic noisy problems (of the form  \(f(\xb) = \sum_{i=1}^p (F_i(\xb)  (1 + \sigma z))^2\)) with \(\sigma\) again controlling the size of the (Gaussian or uniform) noise.  Throughout, we use \(\sigma=10^{-2}\) so that the noise variance is \(10^{-4}\).

We define the value
\[f_{\mathrm{acc}}^{N}:=\frac{f(\xb_{N})-f(\xb_{0})}{f(\xb_{\text{best}})-f(\xb_{0})} \in [0,1],
\]
where \(\xb_{N}\) denotes the best point found by the algorithm after \(N\) function evaluations, \(\xb_{0}\) denotes the initial point, and \(\xb_{\text{best}}\) denotes the best-known solution (given in the test process). Given a tolerance \(\tau \in [0,1]\), we say that the solution reaches the accuracy \(\tau\) when \(f_{\mathrm{acc}}^{N} \ge 1-\tau\). For algorithm $a$ on problem $p$, we then consider the number of evaluations required to reach the accuracy,
\[
N_{a,p}=\left\{
\begin{aligned}
& +\infty, \quad \text{if} \ f_{\mathrm{acc}}^{N_{\max}} < 1-\tau,\\
& \min\{N \in \mathbb{N} \ :\ f_{\mathrm{acc}}^{N}\ge 1-\tau \}, \quad \text{otherwise},
\end{aligned}\right.
\] 
with $N_{\max}$ denoting the maximum number of evaluations performed by algorithm $a$. We then can define the performance ratio
\begin{equation*}
r_{a, p}= \frac{N_{a, p}}{\min_{\tilde{a} \in \mathcal{A}} \left\{N_{\tilde{a}, p}\right\}},
\end{equation*}
where we use the convention that $r_{a,p}=\infty$ for all \(a\) when no algorithm reaches accuracy $\tau$ on problem $p$. For the given tolerance \(\tau\) and a certain problem \(p\) in the problem set \(\mathcal{P}\), \(r_{a, p}\) is the ratio of the number of the function evaluations using the solver \(a\) divided by that using the fastest algorithm on the problem \(p\). 

The performance profile is defined by
\[
\rho_{a}(\alpha)=\frac{1}{\vert\mathcal{P}\vert}\left\vert\left\{p \in \mathcal{P}: r_{a, p} \leq \alpha\right\}\right\vert,
\]
where  \(\alpha 
\in [1, +\infty)\), and \(\vert\cdot\vert\) denotes the cardinality. 
A higher value of \(\rho_a(\alpha)\) represents solving more problems successfully.

The data profile is defined by
\[
\delta_{a}(\beta)=\frac{1}{\vert\mathcal{P}\vert}\left\vert\left\{p \in \mathcal{P}: N_{a, p} \leq \beta(n+1)\right\}\right\vert,
\]
where $\beta\ge 0$.
Higher values of \(\delta_a(\beta)\) again represent solving more problems successfully.

 The tolerances of the trust-region radius and the gradient norm are set to $10^{-8}$. The initial trust-region radius is $
\Delta_0=\max \left\{1,\left\|{\xb}_{0}\right\|_{\infty}\right\}$.  Other parameters in Algorithm~\ref{algo-TR} are set as $\gamma=2, \eta_1={1}/{4}, \eta_2={3}/{4}$, $\mu=0.1$. All experiments in this paper are conducted on a macOS platform featuring an Apple M3 Max chip and 64 GB of memory.

\begin{figure}[htbp]
    \centering
         \includegraphics[height=3.5cm,trim=0 0 185 0,clip]{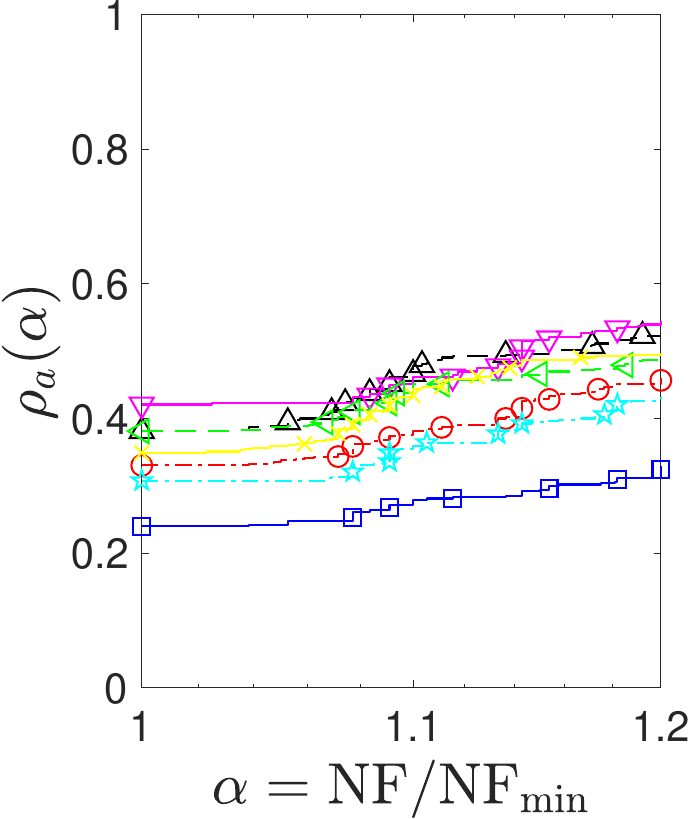} 
\hfill 
         \includegraphics[height=3.5cm,trim=0 0 185 0,clip]{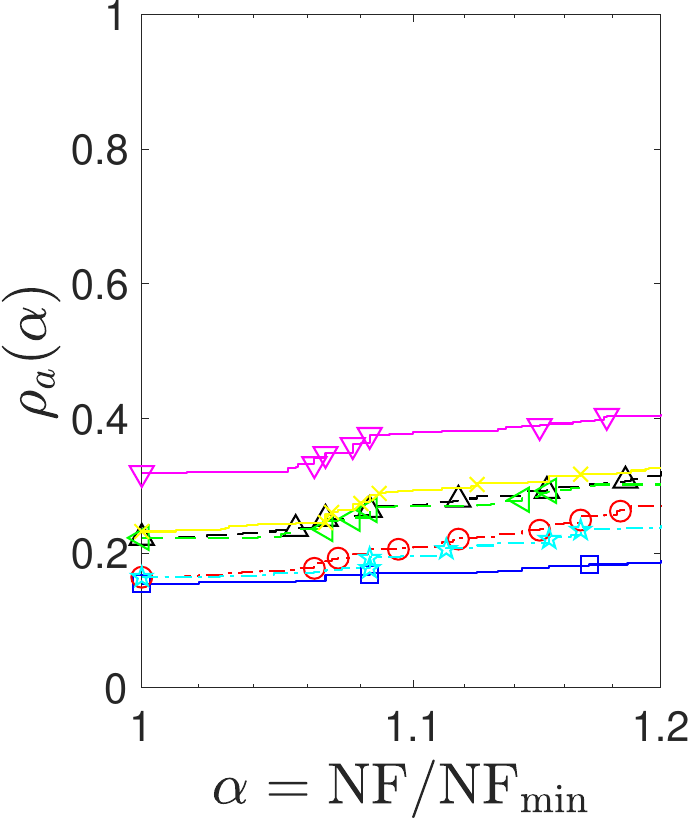} 
    \hfill
         \includegraphics[height=3.5cm,trim=0 0 185 0,clip]{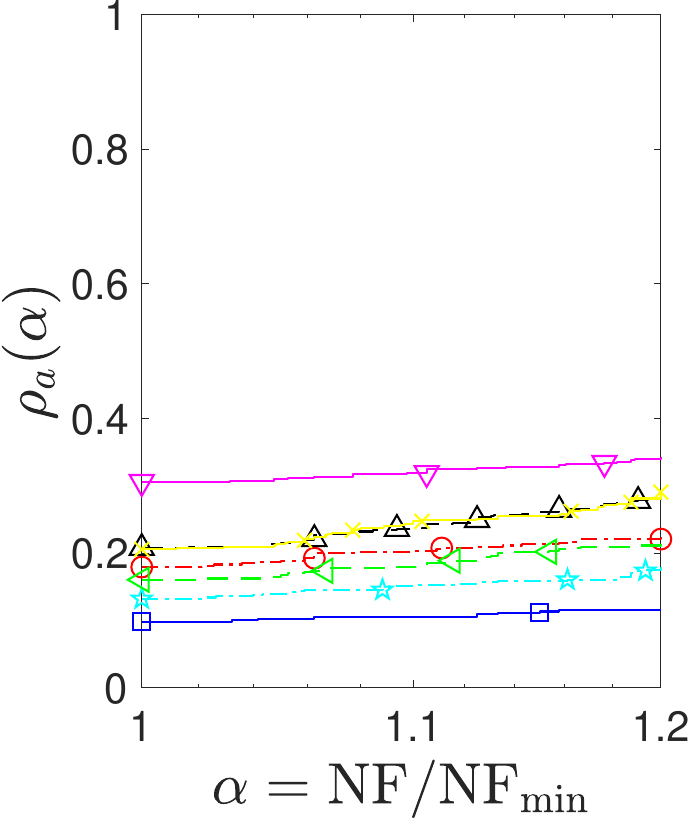}  
         \hfill
         \includegraphics[height=3.5cm,trim=0 0 0 0,clip]{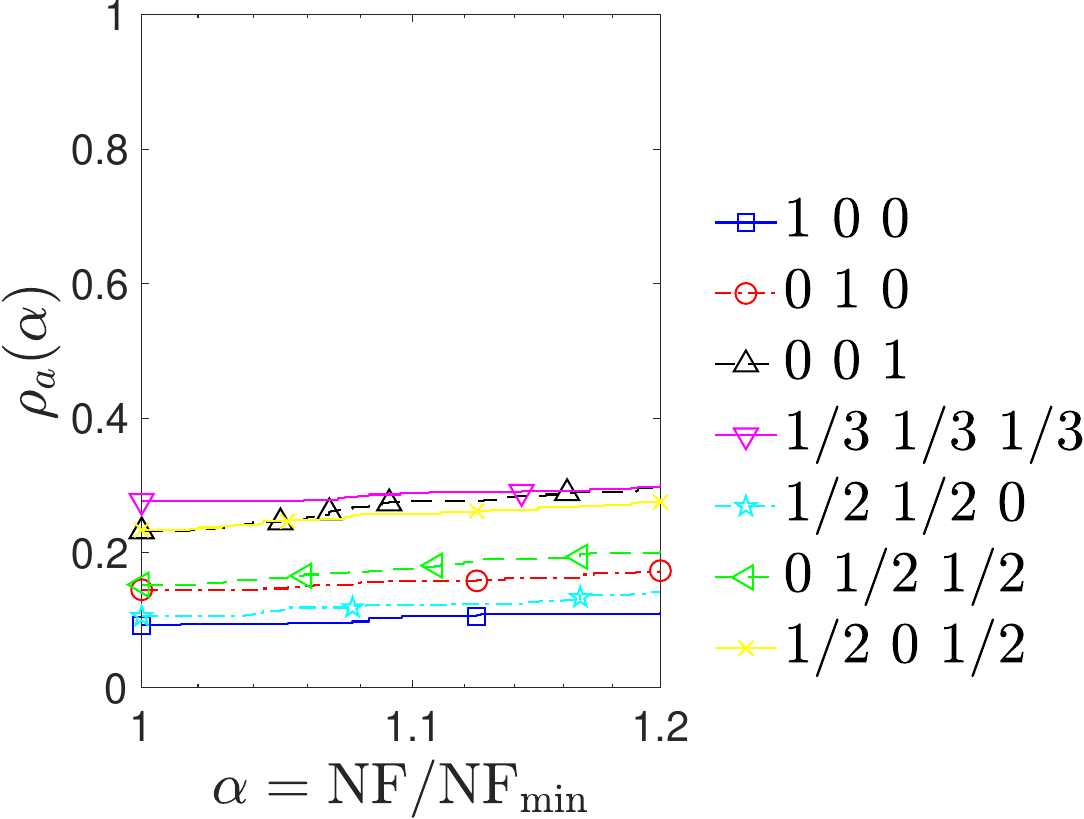}\\
     ~\\
            \centering
         \includegraphics[height=3.5cm,trim=0 0 185 0,clip]{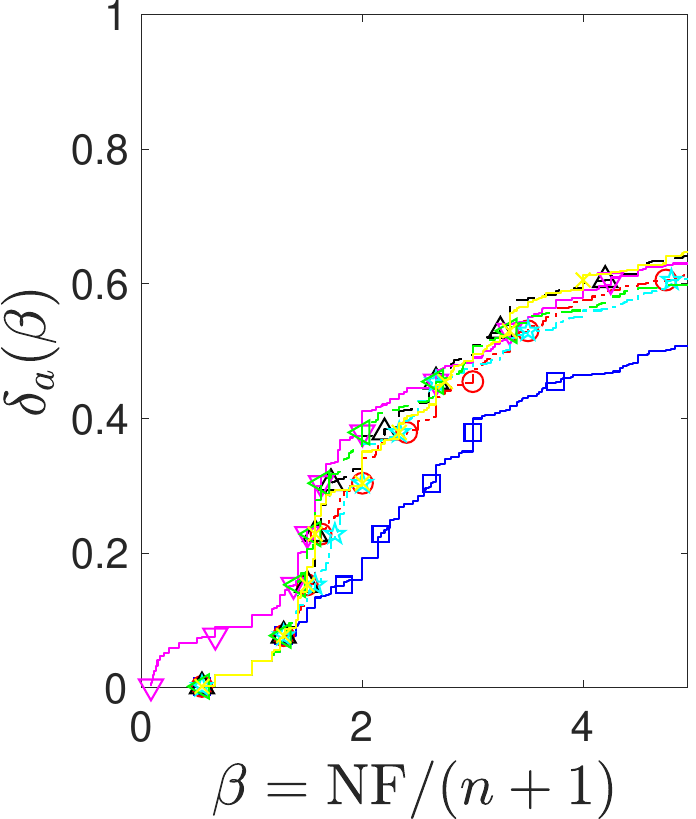} 
    \hfill
         \includegraphics[height=3.5cm,trim=0 0 185 0,clip]{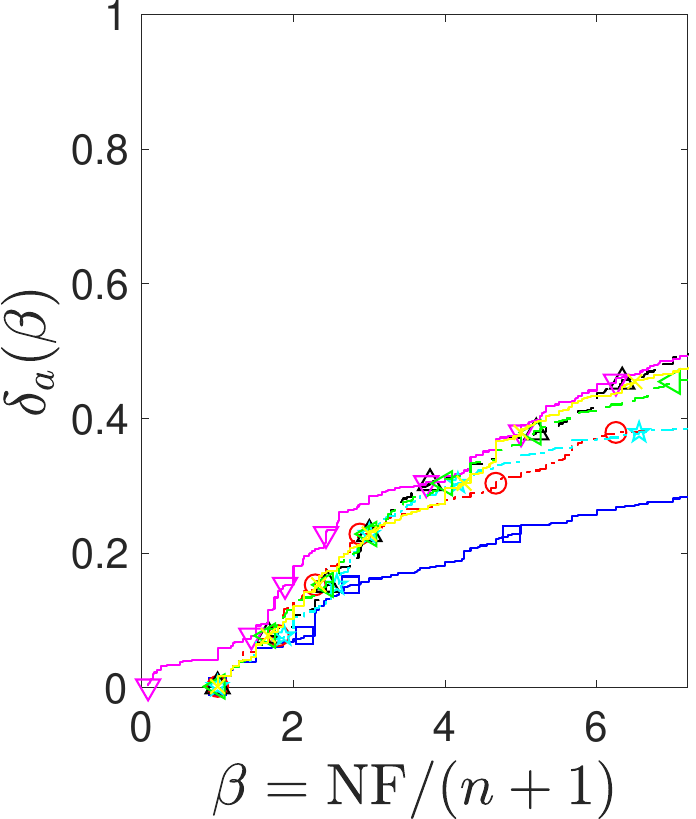} 
    \hfill
         \includegraphics[height=3.5cm,trim=0 0 185 0,clip]{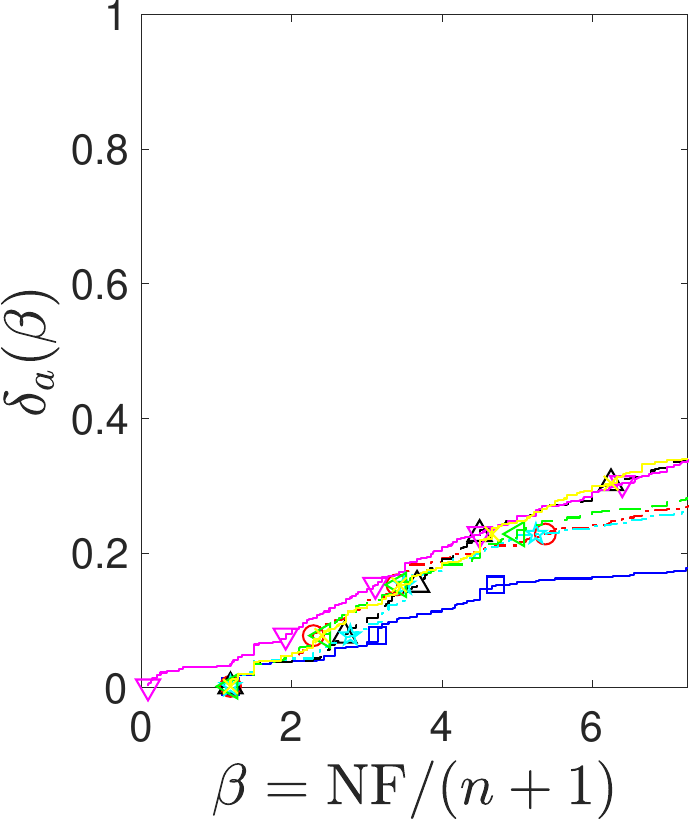} 
           \hfill
         \includegraphics[height=3.5cm,trim=0 0 0 0,clip]{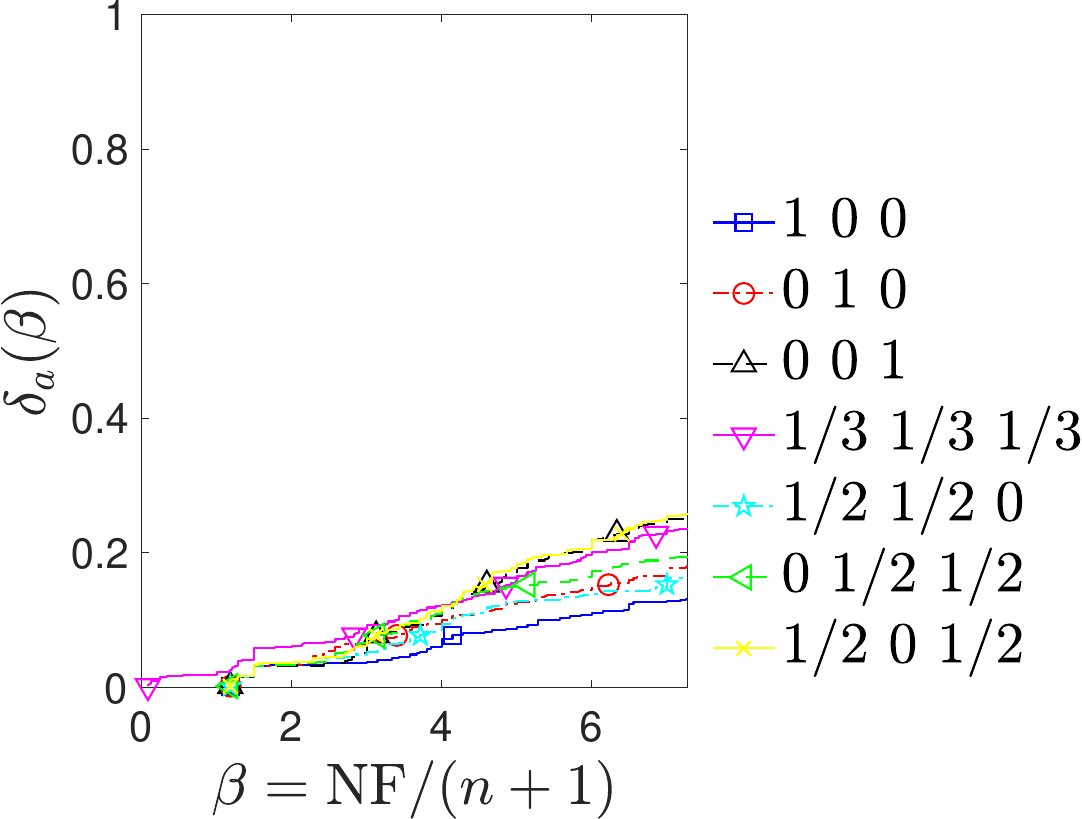} 

         \caption{Performance (top row) and data  (bottom row) profiles with accuracy levels \(\tau=10^{-1},10^{-2},10^{-3},10^{-4}\) (from left to right) for \(2n+1\) interpolation points at each step\label{perf-data-profile-1}}
\end{figure}

\begin{figure}[htbp]
    \centering
         \includegraphics[height=3.5cm,trim=0 0 185 0,clip]{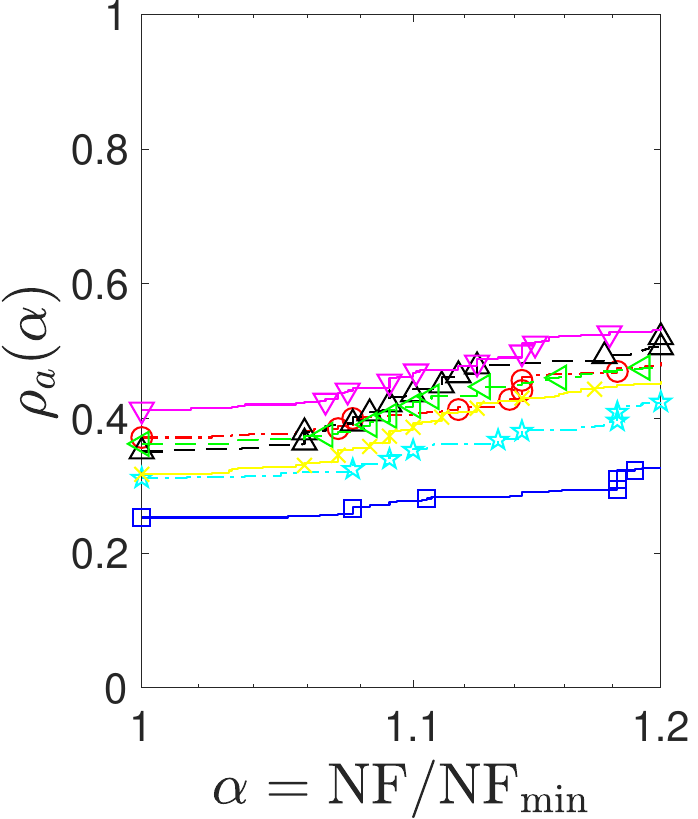} 
    \hfill
         \includegraphics[height=3.5cm,trim=0 0 185 0,clip]{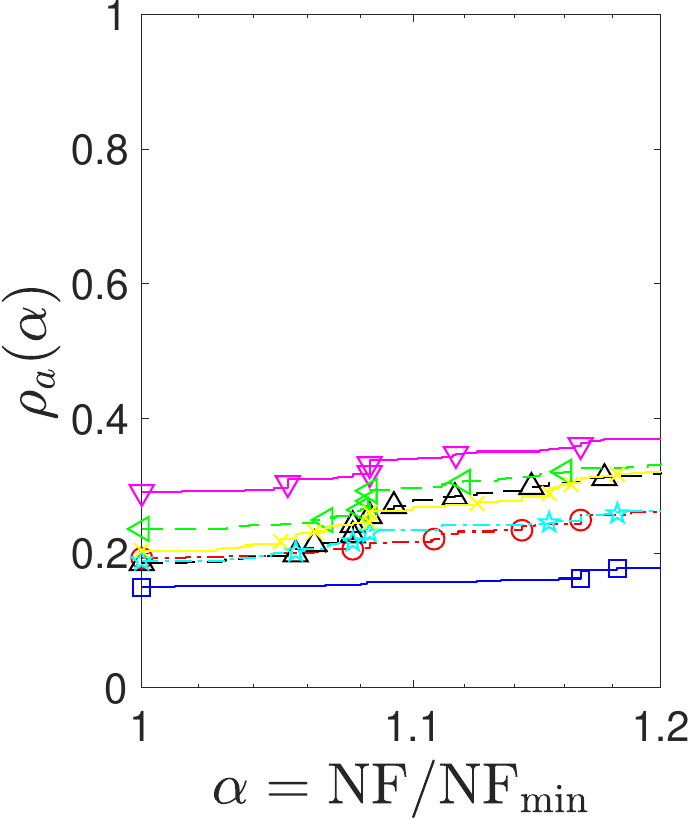} 
    \hfill
         \includegraphics[height=3.5cm,trim=0 0 185 0,clip]{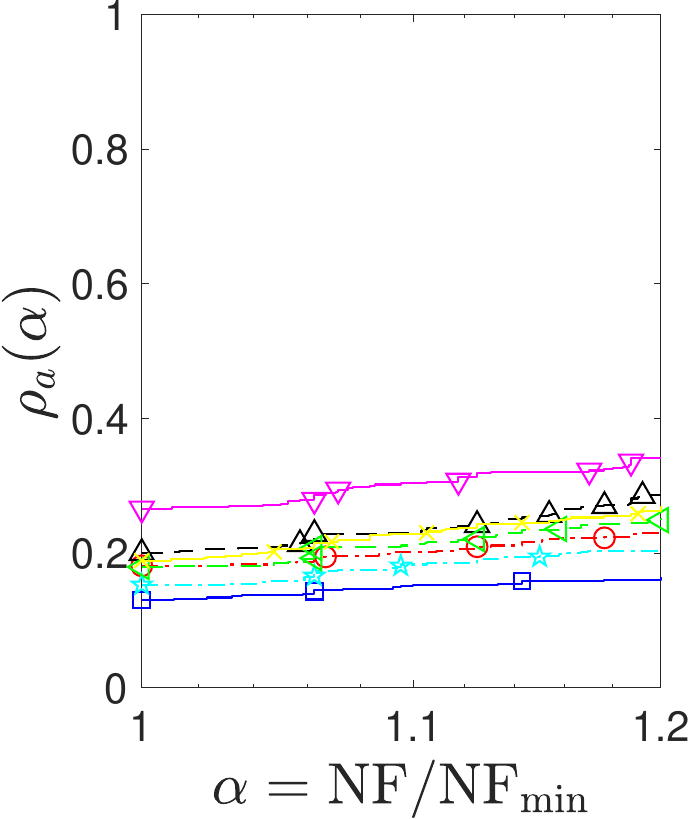}  
         \hfill
         \includegraphics[height=3.5cm,trim=0 0 0 0,clip]{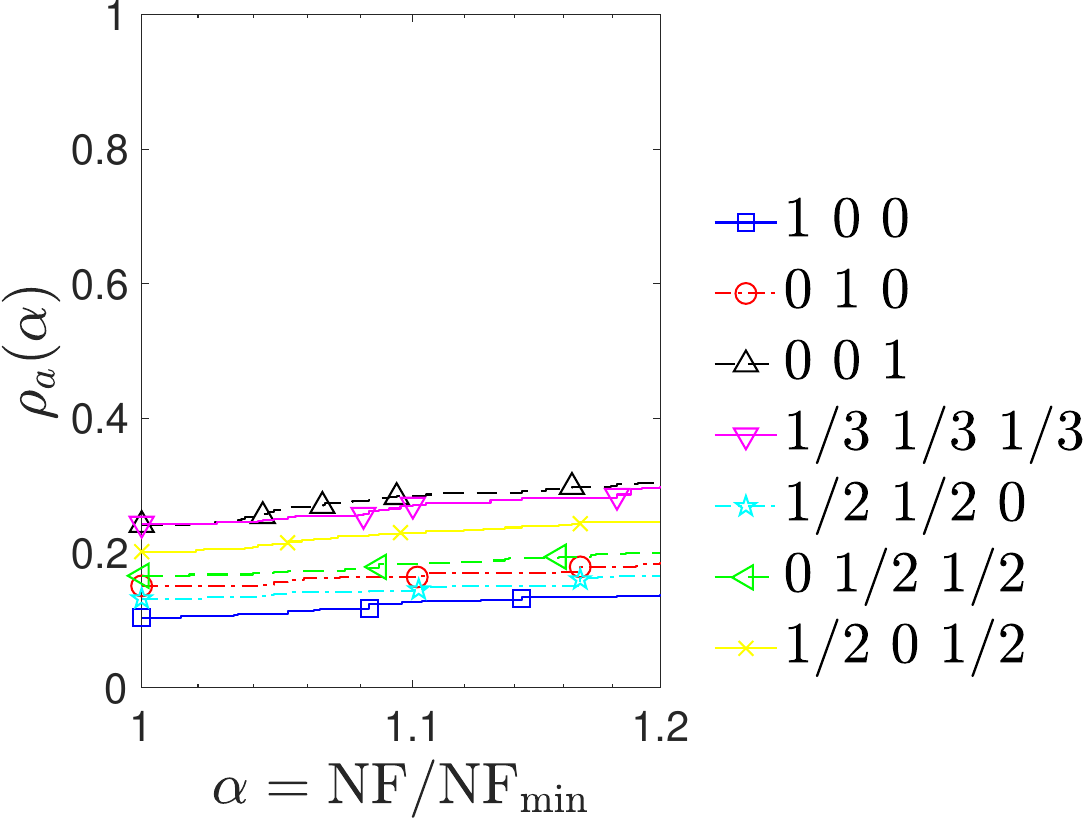}\\
              ~\\
            \centering
         \includegraphics[height=3.5cm,trim=0 0 185 0,clip]{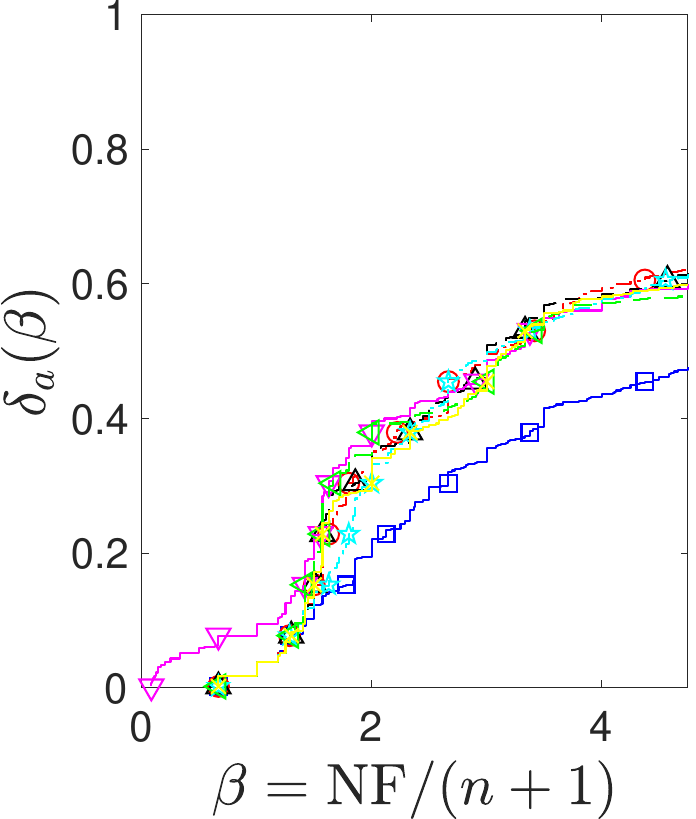} 
    \hfill
         \includegraphics[height=3.5cm,trim=0 0 185 0,clip]{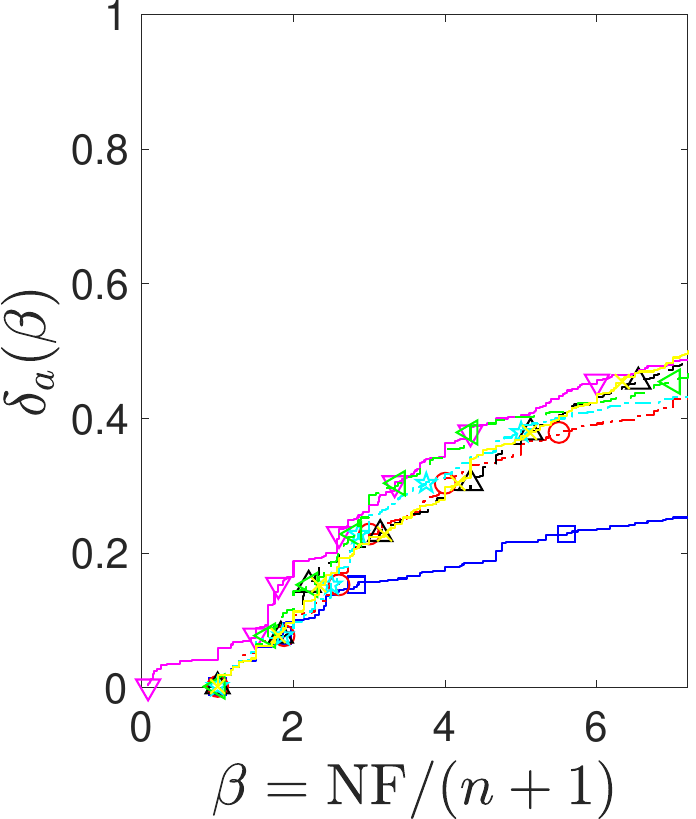} 
    \hfill
         \includegraphics[height=3.5cm,trim=0 0 185 0,clip]{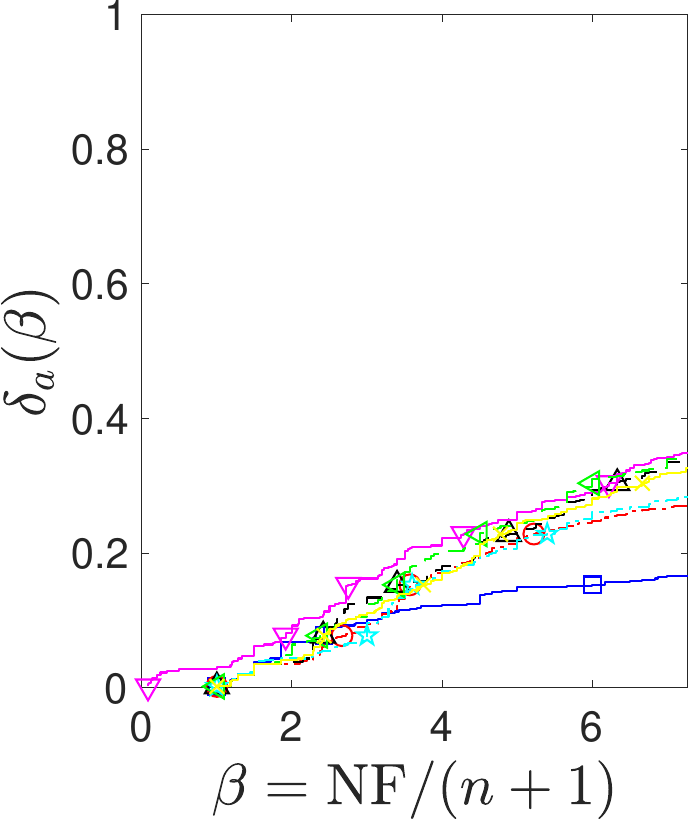} 
           \hfill
         \includegraphics[height=3.5cm,trim=0 0 0 0,clip]{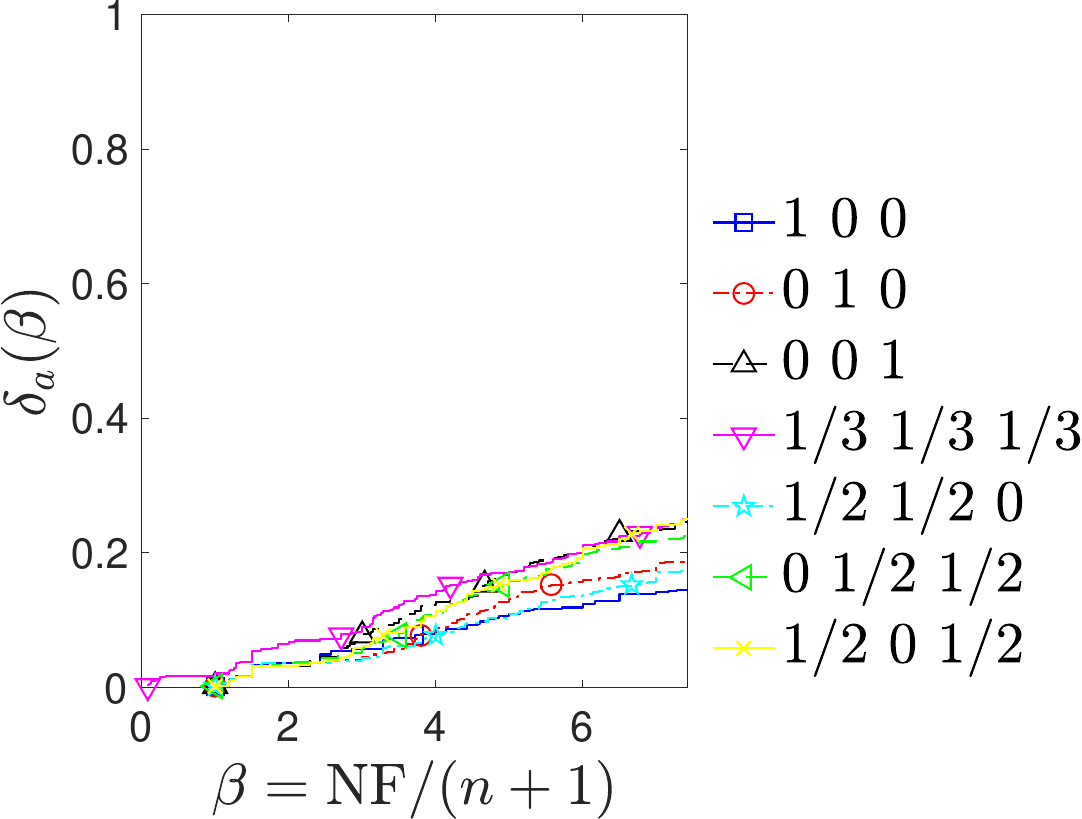} 

         \caption{Performance (top row) and data  (bottom row) profiles with accuracy levels \(\tau=10^{-1},10^{-2},10^{-3},10^{-4}\) (from left to right) for \(n+3\) interpolation points at each step\label{perf-data-profile-2}}
\end{figure}

\figurename~\ref{perf-data-profile-1} and  \figurename~\ref{perf-data-profile-2} show the performance and data profiles obtained with $|\cX_k|=2n+1$ and $|\cX_k|=n+3$, respectively. The three numbers in each legend of the figures refer to the values of \(C_1,C_2,C_3\).
We can observe from \figurename~\ref{perf-data-profile-1} and  \figurename~\ref{perf-data-profile-2} that, among the derivative-free algorithms based on the ReMU models with different weight coefficients listed, the ReMU model with weight coefficients $C_1=C_2=C_3={1}/{3}$ solves more problems than other model types. This result shows the performance variability for different weight coefficients and the advantage of the barycentric ReMU model. 

\subsubsection{Truncated Newton step error comparison}
\label{Truncated Newton step error comparison} 

To further explore properties of the ReMU models with different weight coefficients and the potential advantages of the barycentric model, we examine the different models in a common environment (using the same interpolation points and trust region for all models). The test function used in this experiment is the 11-dimensional ``{\ttfamily Osborne2}'' function from \cite{JJMSMW09} with the ``{\ttfamily absuniform}'' problem type (i.e., additive uniform noise with mean zero and variance \( 10^{-4}\)), providing a robust test environment. We ran 100 iterations of the barycentric model to generate the sequence $(\cX_k,\Delta_k)_k$. This sequence was then used to generate the corresponding models for the seven different ReMU models listed in Table~\ref{diffnorm}, with \(|\cX_k| = n+3\) interpolation points used at each step. 

The way that we will evaluate these models in this common environment is by considering a truncated Newton step. 

\begin{definition}[Truncated Newton step]

Given a twice differentiable function \(h:\R^n\rightarrow \R\), a point \({\bm z}\in \R^n\), and a trust-region radius \(\Delta>0\), the truncated Newton step is 
\[\mathcal{N}(h,{\bm z},\Delta) := - 
\left(\nabla^2 h({\bm z}) \right)^{-1} \nabla h({\bm z})\min\left\{\frac{\Delta}{\|\left(\nabla^2 h({\bm z})\right)^{-1} \nabla h({\bm z})\|_2}, 1\right\},
\]
which is defined to vanish whenever \(\|\left(\nabla^2 h({\bm z})\right)^{-1} \nabla h({\bm z})\|_2=0\). 
    
\end{definition}

The above truncated Newton step (which always satisfies $\|\mathcal{N}(h,{\bm z},\Delta)\|_2\le \Delta$) is a quick indicator of a potentially fruitful step when minimizing a quadratic within a $\Delta$ ball. It allows us another mechanism to compare different ReMU models at each step during the optimizing process.

\begin{definition}[Truncated Newton step error]\label{def:TNSE}

Given two twice differentiable functions \(h_1, h_2:\R^n\rightarrow \R\), a point \({\bm z}\in \R^n\), and a trust-region radius \(\Delta\), the truncated Newton step error is 
\[
{\rm Dist}_{\mathcal{N}}(h_1,h_2,{\bm z},\Delta):= \frac{1}{\Delta}
\|\mathcal{N}(h_1,{\bm z},\Delta) - \mathcal{N}(h_2,{\bm z},\Delta)\|_2. 
\]
    
\end{definition}

\begin{figure}[htbp]
\centering 
 \includegraphics[width=\textwidth,trim=0 0 0 0,clip]{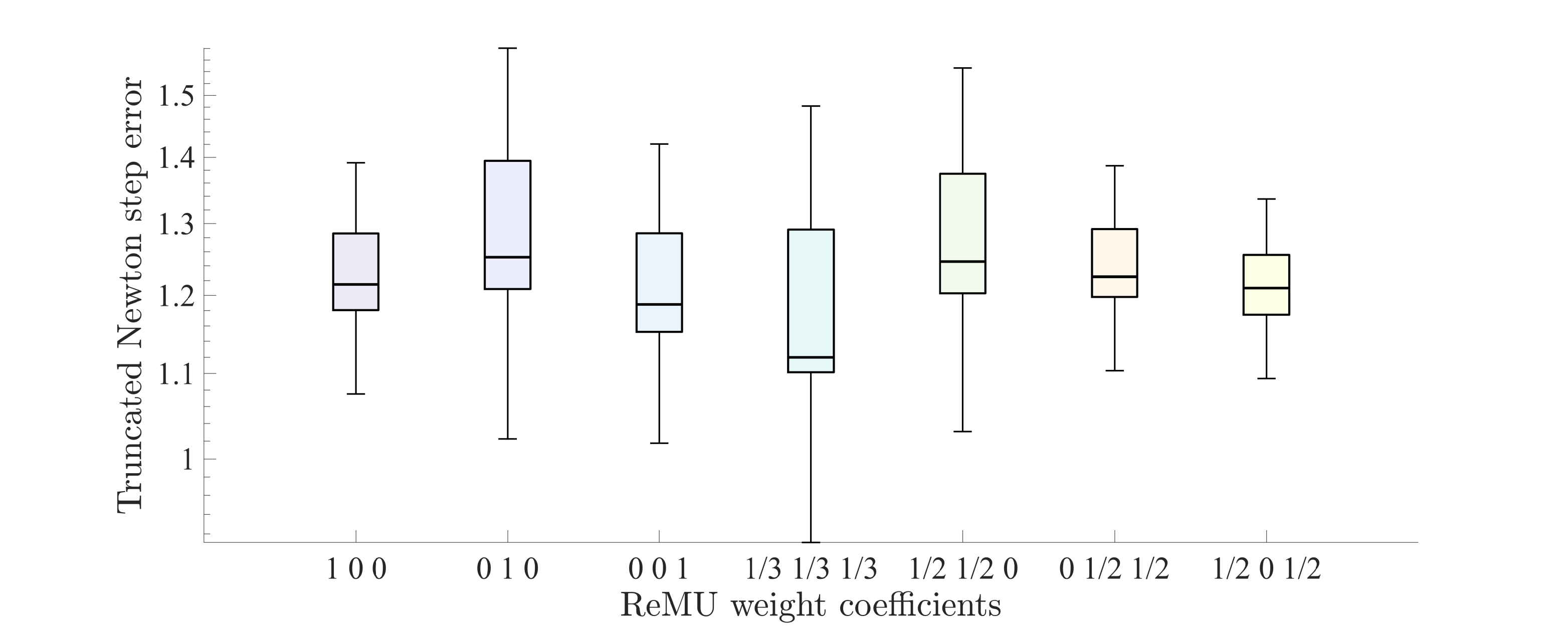} 
\caption{Truncated Newton step error \({\rm Dist}_{\mathcal{N}}(m_k,f,{\xb}_k,\Delta_k)\) for different ReMU weight coefficients. 
The boxes and central line show the 30\%, 50\%, and 70\% quantiles across 100 iterations.
\label{Trunc-delta&newton-error}}
\end{figure}

We note that the truncated Newton step error in Definition~\ref{def:TNSE} is always upper bounded by 2.

For the aforementioned {\ttfamily Osborne2} setup, \figurename~\ref{Trunc-delta&newton-error} shows the truncated Newton step error \({\rm Dist}_{\mathcal{N}}(m_k,f,{\xb}_k,\Delta_k)\) when using the actual objective function $f$ and different ReMU models. 
\figurename~\ref{Trunc-delta&newton-error} shows that all of these model variants result in truncated Newton steps evidently different (i.e., larger than 1) from those that would have been obtained using an exact quadratic Taylor expansion of the objective function $f$. It also shows that among these $(n+3)$-point-interpolating models, the barycentric coefficients \(C_1=C_2=C_3={1}/{3}\) often result in generally lower errors than the other weight combinations. This provides an indication for important early iterations (e.g., before $10(n+1)=120$) as well as bootstrapping to myopic, full quadratic ($(n+1)(n+2)/2=78$) models, in a way that -- to the best of our knowledge -- has not been examined to date. Based on this common testing environment, the truncated Newton step error numerically suggests that the coefficients \(C_1=C_2=C_3={1}/{3}\) can help an algorithm better perform. Similar behavior is also seen when the $n+3$ points and radii are defined by each respective model's trajectory. 

In addition to these results using a common environment, we also ran each of these model variants independently using $|\cX_k|=n+3$. The average (across 100 tests) minimum function value obtained after 100 iterations normalized by the initial function value (i.e., \(f({\xb}_{100}^{\rm best})/f({\xb}_{0})\)) were --  with weight coefficients ordered as listed in \figurename~\ref{Trunc-delta&newton-error} -- 0.6628, 0.9569, 0.4378, 0.4020 (barycentric), 0.6505, 0.4607, and 0.4825, respectively. This is an additional indication of how the barycentric model performs well.   

\section{Regional minimal updating}
\label{Regional minimal updating}

We now give more analysis of the regional minimal updating strategy. 
Notice from the above that the value of $r$ in the regional minimal updating is usually proportionally related to the trust-region radius. Therefore, $r\rightarrow 0$ is exactly the symbol of the case where the trust-region radius tends to zero. The parameters $\eta_1,\eta_2,\eta_3,\eta_4,\eta_5$  can be respectively written in the function form as
\begin{equation*}
\eta_1(C_1,C_2,C_3,n,r),\ \eta_2(C_1,C_2,n,r),\ \eta_3(C_1,n,r),\ \eta_4(C_1,n,r),\ \eta_5(C_1).
\end{equation*}

Some results of parameters of the KKT matrix $\W$ in the limit where the trust-region radius is vanishing are provided in the following proposition.
\begin{proposition} 
\begin{equation}\label{eta-limitation} 
\begin{aligned}
&\lim_{r\rightarrow 0} \eta_1(C_1,C_2,C_3,n,r)=C_3,\
\lim_{r\rightarrow 0} \eta_2(C_1,C_2,n,r)=C_2,\\ 
&\lim_{r\rightarrow 0} \eta_3(C_1,n,r)=0,\ 
\lim_{r\rightarrow 0} \eta_4(C_1,n,r)=0,  
\end{aligned}
\end{equation}
and then 
\begin{equation*} 
\begin{aligned}
&\lim_{r\rightarrow 0} \frac{1}{{8}\eta_1}=\frac{1}{{8}C_3},\ \lim_{r\rightarrow 0} -\frac{\eta_{3}}{8\eta_{1}\left(n \eta_{3}+\eta_{1}\right)}=0,\ \lim_{r\rightarrow 0}-\frac{\eta_{4}}{4 n \eta_{3}+4 \eta_{1}}=0,\\ &\lim_{r\rightarrow 0} -2\eta_2=-2C_2,\ \lim_{r\rightarrow 0} \frac{n \eta_{4}^{2}}{2 n \eta_{3}+2 \eta_{1}}-2\eta_5=-2C_1.
\end{aligned} 
\end{equation*}
\end{proposition} 

\begin{proof}
Direct computation can derive this. 
\end{proof}

\subsection{KKT matrix error and distance}
\label{KKT matrix error and distance}

This section introduces the KKT matrix error and distance when generating the ReMU models. 
The proposed ReMU model is obtained based on calculating the corresponding parameters \(\boldsymbol{\lambda}, \hat{c}, \hat{\g}\) for given weight coefficients. Since \(\boldsymbol{\lambda}, \hat{c}, \hat{\g}\) are determined by the KKT matrix \(\W\) in the KKT equations \eqref{system1-1} directly, it is natural and reasonable to use the KKT matrix distance, defined in Definition~\ref{KKTdist}, to denote the distance between two models. As we can see in the KKT system \eqref{system1-1}, the right-hand side vector \((0, \ldots, 0, f(\xb_{\text{new}})-m_{k-1}(\xb_{\text{new}}), 0,\dots,0)^{\top}\) is independent of the weight coefficients, and thus the only difference as the coefficients change will be in the KKT matrix \(\W\) on the left-hand side of \eqref{system1-1}. The KKT matrix $\W$ directly determines the parameters of our quadratic model and is the key to identifying different ReMU models.

An additional mathematical fact is that the properties of the quadratic ReMU model, such as the projection theory of the ReMU model, should hold for the corresponding (semi-)norms. Different (semi-)norms refer to different KKT matrices, and what we want to do is to identify any tradeoffs coming from the weight coefficients. We also aim to find a kind of central KKT matrix by examining the barycenter of the weight coefficient region. 

\begin{remark}
As shown in Definition \ref{ReMUdef}, without loss of generality, we assume that $C_1+C_2+C_3=1$, and this does not influence our discussion on the weight coefficients. The same assumption also holds for the notations $C^*_1+C^*_2+C^*_3=1$ when representing other coefficient values. 
\end{remark}

In the following analysis, the assumption below holds.
\begin{assumption}
\label{assum-W-W*}
\(\W\) and \({\W^*}\) are respectively the corresponding KKT matrices of obtaining the model determined by the corresponding weight coefficients of the objective function in \eqref{lagobj}, we say \(C_{1}, C_{2}\) and \(C_{1}^{*}, C_{2}^{*}\), according to \eqref{system1-1}.
\end{assumption}
We have the following theorem about the distance of two KKT matrices.
\begin{theorem}
Suppose that $\W$ and $\W^*$ satisfy Assumption~\ref{assum-W-W*}, then \(\left\|\W-\W^{*}\right\|_{F}^{2}\) is
\begin{equation}
\label{W-W*F}
\begin{aligned}
&\left\|\W-\W^{*}\right\|_{F}^{2}
=\left\{\sum_{i=1}^{|\cX_k|} \sum_{j=1}^{|\cX_k|}\left[\left(\y_{i}-{\xb}_{k}\right)^{\top}\left(\y_{j}-{\xb}_{k}\right)\right]^{4}\right\}\left(\frac{1}{8\eta_{1}}-\frac{1}{8\eta_{1}^{*}}\right)^{2} \\
&+\left(\sum_{i=1}^{|\cX_k|}\sum_{j=1}^{|\cX_k|}\left\|\y_{i}-{\xb}_{k}\right\|_2^{4}\left\|\y_{j}-{\xb}_{k}\right\|_2^4\right)\left(\frac{\eta_{3}}{8\eta_{1}\left(n \eta_{3}+\eta_{1}\right)}-\frac{\eta^*_{3}}{8\eta^*_{1}\left(n \eta^*_{3}+\eta^*_{1}\right)} \right)^2 \\
&+\left(\sum_{i=1}^{|\cX_k|}\left\|\y_{i}-{\xb}_{k}\right\|_2^{4}\right)\left[-\frac{\eta_{4}}{4 n \eta_{3}+4 \eta_{1}}-(-\frac{\eta^*_{4}}{4 n \eta^*_{3}+4 \eta^*_{1}})\right]^2 \\
&+n[2(\eta^*_2-\eta_2)]^2+\left[\frac{n \eta_{4}^{2}}{2 n \eta_{3}+2 \eta_{1}}-2\eta_5-(\frac{n (\eta^*_{4})^{2}}{2 n \eta^*_{3}+2 \eta^*_{1}}-2\eta^*_5)\right]^2, 
\end{aligned}
\end{equation}
where \(\eta_1,\eta_2, \eta_3, \eta_4, \eta_5\) are defined as \eqref{eta}, and \(\eta_1^*, \eta_2^*, \eta_3^*, \eta_4^*, \eta_5^*\) are defined as
\begin{equation*}
\begin{aligned}
\eta_{1}^{*}&=C_{1}^{*} \frac{r^{4}}{2(n+4)(n+2)}+C_{2}^{*} \frac{r^{2}}{n+2}+C_{3}^{*}, \\ 
\eta^*_2&=C^*_1\frac{r^{2}}{n+2}+C^*_2,\\
\eta_{3}^{*}&=C_{1}^{*} \frac{r^{4}}{4(n+4)(n+2)}, \\
\eta_{4}^{*}&=C_{1}^{*} \frac{r^{2}}{n+2}, \\ 
\eta_{5}^{*}&=C_{1}^{*}.
\end{aligned}
\end{equation*}
\end{theorem}

\begin{proof}

We calculate \(\boldsymbol{A}-\boldsymbol{A}^{*}, \boldsymbol{J}-\boldsymbol{J}^{*}\), and \(\boldsymbol{X}-\boldsymbol{X}^{*}\) as follows:
\begin{align*}
&\boldsymbol{A}_{i j}-\boldsymbol{A}_{ij}^{*}=\left[\left(\y_{i}-{\xb}_{k}\right)^{\top}\left(\y_{j}-{\xb}_{k}\right)\right]^{2} \left(\frac{1}{8\eta_{1}}-\frac{1}{8\eta_1^{*}}\right) \\
&\quad \quad \quad \quad \quad \,  -\left\|\y_{i}-{\xb}_{k}\right\|_2^2\left\|\y_{j}-{\xb}_{k}\right\|_2^2\left(\frac{\eta_{3}}{8\eta_{1}\left(n \eta_{3}+\eta_{1}\right)}-\frac{\eta^*_{3}}{8\eta^*_{1}\left(n \eta^*_{3}+\eta^*_{1}\right)} \right),\\ 
&\boldsymbol{X}-\boldsymbol{X}^{*}=\mathbf{0},\\
&\boldsymbol{J}-\boldsymbol{J}^{*}=\left(-\frac{\eta_{4}}{4 n \eta_{3}+4 \eta_{1}}-(-\frac{\eta^*_{4}}{4 n \eta^*_{3}+4 \eta^*_{1}})\right)\left(
\left\|\y_{1}-{\xb}_{k}\right\|_2^2,
\dots,
\left\|\y_{|\cX_k|}-{\xb}_{k}\right\|_2^2
\right)^{\top},\\ 
&(-2\eta_2\boldsymbol{I})-(-2\eta_2^*\boldsymbol{I})=2(\eta^*_2-\eta_2)\boldsymbol{I},
\end{align*}
which derives \eqref{W-W*F}, after calculating the square summation. 
\end{proof}

Furthermore, we can obtain the following corollary about \(\Vert \W-\W^*\Vert_{F}^2\).

\begin{corollary}
Suppose that $\W$ and $\W^*$ satisfy Assumption~\ref{assum-W-W*}, then
\(\left\|\W-\W^{*}\right\|_{F}^{2}\) is a function of \(C_{1},C_{2},C_{1}^{*}, C_{2}^{*},n,r\), and in details, it holds that
\begin{equation}\label{formula-distance}\small 
\begin{aligned} 
&\left\|\W-\W^{*}\right\|_{F}^{2}
=:\mathcal{D}({{C_1},{C_2},{C_1^*},{C_2^*},n,r})\\ &=\mathcal{R}_1(\y_1,\dots,\y_{|\cX_k|}) \mathcal{P}_1({{C_1},{C_2},{C_1^*},{C_2^*},n,r}) +\mathcal{R}_2(\y_1,\dots,\y_{|\cX_k|}) \mathcal{P}_2({{C_1},{C_2},{C_1^*},{C_2^*},n,r})\\
&+\mathcal{R}_3(\y_1,\dots,\y_{|\cX_k|}) \mathcal{P}_3({{C_1},{C_2},{C_1^*},{C_2^*},n,r})+\mathcal{P}_4({{C_1},{C_2},{C_1^*},{C_2^*},n,r}).
\end{aligned}
\end{equation}
The terms
\begin{equation*} 
\begin{aligned}
\mathcal{R}_1(\y_1,\dots,\y_{|\cX_k|})&=\sum_{i=1}^{|\cX_k|} \sum_{j=1}^{|\cX_k|}\left[\left(\y_{i}-{\xb}_{k}\right)^{\top}\left(\y_{j}-{\xb}_{k}\right)\right]^{4},\\\mathcal{R}_2(\y_1,\dots,\y_{|\cX_k|})&=\sum_{i=1}^{|\cX_k|}\sum_{j=1}^{|\cX_k|}\left\|\y_{i}-{\xb}_{k}\right\|_2^{4}\left\|\y_{j}-{\xb}_{k}\right\|_2^4,\\ 
\mathcal{R}_3(\y_1,\dots,\y_{|\cX_k|})&=\sum_{i=1}^{|\cX_k|}\left\|\y_{i}-{\xb}_{k}\right\|_2^{4} 
\end{aligned} 
\end{equation*}
only depend on the given interpolation points $\y_1,\dots,\y_{|\cX_k|}$, given a base point ${\xb}_{k}$ at the current iteration, and $\mathcal{P}_1,\mathcal{P}_2,\mathcal{P}_3$ are functions of $C_1,C_2,C_1^*,C_2^*,n,r$, with explicit expressions presented in the next section.

\end{corollary}

\begin{proof}
Substituting $\eta_1, \eta_2, \eta_3, \eta_4, \eta_5$ and $\eta_1^*, \eta_2^*, \eta_3^*, \eta_4^*, \eta_5^*$ with $C_1,C_2$ and $C_1^*,C_2^*$ can derive \eqref{formula-distance}. 
\end{proof}

In order to further discuss the central KKT matrix, we first give the definitions of the KKT matrix distance and the KKT matrix error.

\begin{definition}[KKT matrix distance]\label{KKTdist}
We define the KKT matrix distance between two KKT matrices $\W$ and $\W^*$ as
 \(\|\W-{\W^*}\|_{F}\).
\end{definition}

\begin{theorem}
The KKT matrix distance in Definition~\ref{KKTdist} is a well-defined distance on the set of the KKT matrices.
\end{theorem}

\begin{proof}
We have the following facts.
\begin{itemize}
\item[-] The KKT matrix distance is nonnegative, and 
\(
\left\|\W-\W^{*}\right\|_{F}=0
\)
if and only if \(\W=\W^{*}\).

\item[-] 
The symmetric property 
\(
\left\Vert \W-\W^{*}\right\Vert_ F=\Vert \W^{*}-\W\Vert_F
\) 
holds.

\item[-] Triangle inequality 
\begin{equation*}
\begin{aligned}
\left\Vert \W-\bar{\W}^{*}\right\Vert_{F} &=\left\Vert\left(\W-\W^{*}\right)+\left(\W^{*}-\bar{\W}^{*}\right)\right\Vert_{F}\\
&\leq\left\|\W-\W^{*}\right\|_{F}+\left\|\W^{*}-\bar{\W}^{*}\right\|_{F}
\end{aligned}
\end{equation*}
holds, according to properties of the Frobenius norm.
\end{itemize}
Therefore, we conclude that the KKT matrix distance is a well-defined distance. 
\end{proof}

\begin{definition}[KKT matrix error]\label{KKTerror}
\begin{sloppypar}
Suppose that $\W$ and $\W^*$ satisfy Assumption~\ref{assum-W-W*}. We define the KKT matrix error between two pairs of weight coefficients $(C_1,C_2)$ and $(C_1^*,C_2^*)$ as
$\sqrt{\mathcal{D}({{C_1},{C_2},{C_1^*},{C_2^*},n,r})}$, where ${\mathcal{D}({{C_1},{C_2},{C_1^*},{C_2^*},n,r})}$ is defined in \eqref{formula-distance}.
\end{sloppypar}
\end{definition}

\subsection{Geometric points of the coefficient region}
\label{Geometric points of the coefficient region}

We now examine the geometry of the coefficient region of the ReMU models. 
We aim to search for a weight coefficient pair $C_1,C_2$ (and corresponding $C_3=1-C_1-C_2$) using the KKT matrix error, where $(C_1,C_2)^{\top}$ is in the region 
$\mathcal{C}$. \figurename~\ref{Coefficient set figure} shows the coefficient region $\mathcal{C}$. To avoid the denominator $\eta_1$ appearing in the KKT matrix of \eqref{system1-1} being too small when $r\rightarrow 0$, we assume that there is a lower bound for $C_3$ of $\varepsilon\in (0,1)$. The small parameter \(\varepsilon\) here exists for the convenience of the theoretical analysis and to work in the lower-dimensional coefficient space. 

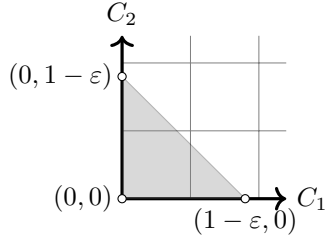
\begin{figure}[htbp]
    \centering
\begin{tikzpicture}[scale=1.8]
\draw[style=help lines,step=0.5cm] (0,0) grid (1.2,1.2);

\draw[->,very thick] (0,0) -- (1.2,0) node[right] {$C_1$};
\draw[->,very thick] (0,0) -- (0,1.2) node[above] {$C_2$};

\filldraw[fill= black!50,opacity=0.3] (0.9,0) -- (0,0.9) -- (0,0); 
\filldraw[fill= white,opacity=1] (0.9,0) circle (0.03);
\filldraw[fill= white,opacity=1] (0,0.9) circle (0.03);
\filldraw[fill= white,opacity=1] (0,0) circle (0.03); 
\node[below] at (0.9,0){$(1-\varepsilon,0)$};
\node[left] at (0,0.9){$(0,1-\varepsilon)$};
\node[left] at (0,0){$(0,0)$};
\end{tikzpicture}
    \caption{Coefficient region $\mathcal{C}$\label{Coefficient set figure}}
\end{figure}

To illustrate the average KKT matrix distance, we first make the following definitions.
\begin{definition}[Average squared KKT matrix error]
\label{ave-KKT-dist}
Considering \eqref{formula-distance}, given a pair of coefficients $(C_1, C_2)$, we define the average squared KKT matrix error as
\begin{equation*}
\begin{aligned}
    \text{Error}_{\text{ave}}(C_1,C_2,n,r,\varepsilon):=&\frac{{\int_{0}^{1-\varepsilon-C_1^*}\int_{0}^{1-\varepsilon}\mathcal{D}(C_1,C_2,C_1^*,C_2^*,n,r){\rm d}C_1^*{\rm d}C_2^*}}{\int_{0}^{1-\varepsilon-C_1}\int_{0}^{1-\varepsilon}{\rm d}C_1^*{\rm d}C_2^*}\\
    =& \frac{2{\int_{0}^{1-\varepsilon-C_1^*}\int_{0}^{1-\varepsilon}\mathcal{D}(C_1,C_2,C_1^*,C_2^*,n,r){\rm d}C_1^*{\rm d}C_2^*}}{(1 - \varepsilon)^2}.
\end{aligned}
\end{equation*}
\end{definition}

\begin{definition}[Barycenter of the coefficient region]
The barycenter of weight coefficient region $\mathcal{C}$ of the ReMU models is the solution of
\begin{equation}
\label{prob-*}
\min_{(C_1,C_2)^{\top}\in \mathcal{C}} \text{Error}_{\text{ave}}(C_1,C_2,n,r,\varepsilon).
\end{equation}
\end{definition}
The barycenter of weight coefficient region $\mathcal{C}$ has the smallest average squared KKT matrix error. It exactly refers to the point in the region \(\mathcal{C}\) that provides the least moment of inertia measured by the KKT matrix error instead of the Euclidean distance.

We give the analytic result of the barycenter of the weight coefficient region of the ReMU models with a vanishing trust-region radius in this section.

\begin{theorem}
Given $C_3\ge \varepsilon$, if $r\rightarrow 0$, we obtain that
\begin{equation*}
\begin{aligned}
    &\lim_{r\rightarrow 0}\text{\rm Error}_{\text{\rm ave}}(C_1,C_2,n,r,\varepsilon)\\
    = &\,\mathcal{R}_1(\y_1,\dots,\y_{|\cX_k|})\text{\rm Error}_{\text{\rm ave}}^{(1)}(C_1,C_2,\varepsilon)+\text{\rm Error}_{\text{\rm ave}}^{(2)}(C_1,C_2,n,\varepsilon),
\end{aligned}
\end{equation*}
where \({\text{\rm Error}_{\text{\rm ave}}^{(1)}(C_1,C_2,\varepsilon)}\) is 
\begin{equation*}
\begin{aligned}
\frac{\frac{({\varepsilon}-1) \left({\varepsilon} (4 {C_1}+4 {C_2}-5)-2 ({C_1}+{C_2}-1)^2+{\varepsilon}^2\right)}{2 {\varepsilon}}+({C_1}+{C_2}-3) ({C_1}+{C_2}-1) \log ({\varepsilon})}{32 (1-{\varepsilon})^2 ({C_1}+{C_2}-1)^2},
\end{aligned}
\end{equation*}
and \({\text{\rm Error}_{\text{\rm ave}}^{(2)}(C_1,C_2,n,\varepsilon)}\) is
\begin{equation*}
\begin{aligned}
\frac{1}{6} \left(24 {C_1}^2+16 {C_1} ({\varepsilon}-1)+6 {C_2}^2 n+{\varepsilon} (4 {C_2} n-2 n-8)\right.\left.-4 {C_2} n+{\varepsilon}^2 (n+4)+n+4\right).
\end{aligned}
\end{equation*}
\end{theorem}

\begin{proof}
It holds that
\begin{align*}
&\mathcal{P}_1({{C_1},{C_2},{C_1^*},{C_2^*},n,r})\\
=&\,\frac{1}{64} \left\{\left[{{C_1} \left(\frac{r^4}{2 \left(n^2+6 n+8\right)}-1\right)+{C_2} \left(\frac{r^2}{n+2}-1\right)+1}\right]^{-1}\right.\\
&\left.-\left[{{C^*_1} \left(\frac{r^4}{2 \left(n^2+6 n+8\right)}-1\right)+{C^*_2} \left(\frac{r^2}{n+2}-1\right)+1}\right]^{-1}\right\}^2,\\
&\mathcal{P}_2({{C_1},{C_2},{C_1^*},{C_2^*},n,r})\\
=&\,\frac{r^8}{1024 (n+2)^2 (n+4)^2} 
\left\{{{C^*_1}}
\left[
\bigg(
\frac{{C^*_1} 
\left(-4 n+r^4-16
\right)}{4 (n+4)}
+{C^*_2} \left(\frac{r^2}{n+2}-1\right)
\right.
\right.
 \\&
 \left.
 \left.
 +1\bigg) 
 \left({C^*_1} 
 \left(\frac{r^4}{2 \left(n^2+6 n+8\right)}-1\right)+\frac{{C^*_2} r^2}{n+2}-{C^*_2}+1
 \right)
 \right]^{-1}
 \right.\\
&\left.
-{{C_1}}\left[
\left(\frac{{C_1} \left(-4 n+r^4-16\right)}{4 (n+4)}+{C_2} \left(\frac{r^2}{n+2}-1\right)+1\right)\right.\right.\\
& \left.\left. \left({C_1} \left(\frac{r^4}{2 \left(n^2+6 n+8\right)}-1\right)+\frac{{C_2} r^2}{n+2}-{C_2}+1\right)
\right]^{-1}
\right\}^2,\\
&\mathcal{P}_3({{C_1},{C_2},{C_1^*},{C_2^*},n,r})\\
=&\,\frac{1}{16} r^8 \bigg\{{{C^*_1}}\left[{C^*_1} \left(4 n^2-n \left(r^4-24\right)-2 r^4+32\right)\right.\\
&\left.+4 (n+4) \left({C^*_2} \left(n-r^2+2\right)-n-2\right)\right]^{-1} -{{C_1}}\left[{C_1} \left(4 n^2-n \left(r^4-24\right)-2 r^4+32\right)\right.\\
&\left.+4 (n+4) \left({C_2} \left(n-r^2+2\right)-n-2\right)\right]^{-1}\bigg\}^2,\\
&\mathcal{P}_4({{C_1},{C_2},{C_1^*},{C_2^*},n,r})\\
=&\,4 \left(-
\left({{C_1}^2 n (n+4) r^4}\right)\left[(n+2) \left({C_1} \left(4 n^2-n \left(r^4-24\right)-2 r^4+32\right)\right.\right.\right.\\
&\left.\left.\left.+4 (n+4) \left({C_2} \left(n-r^2+2\right)-n-2\right)\right)\right]^{-1}
-{C_1}\right.\\
&\left.+\left({{C^*_1}^2 n (n+4) r^4}\right)\left[(n+2) \left({C^*_1} \left(4 n^2-n \left(r^4-24\right)-2 r^4+32\right)\right.\right.\right.\\
&\hspace{-0.4cm}\left.\left.\left.+4 (n+4) \left({C^*_2} \left(n-r^2+2\right)-n-2\right)\right)\right]^{-1}+{C^*_1}\right)^2 +n \left(\frac{r^2 ({C_1}-{C^*_1})}{n+2}+{C_2}-{C^*_2}\right)^2.
\end{align*} 
The limiting behavior of each of $\mathcal{P}_1,\mathcal{P}_2,\mathcal{P}_3,\mathcal{P}_4$ is as follows.
\begin{equation*} 
\begin{aligned}
\lim_{r\rightarrow 0} \mathcal{P}_1({{C_1},{C_2},{C_1^*},{C_2^*},n,r})=&\,\frac{({C_1}+{C_2}-{C^*_1}-{C^*_2})^2}{64 ({C_1}+{C_2}-1)^2 ({C^*_1}+{C^*_2}-1)^2}\\
:=&\,\mathcal{P}_1^{(\text{lim})}({{C_1},{C_2},{C_1^*},{C_2^*}}),\\
\lim_{r\rightarrow 0} \mathcal{P}_2({{C_1},{C_2},{C_1^*},{C_2^*},n,r})=&\,0,\\ 
\lim_{r\rightarrow 0} \mathcal{P}_3({{C_1},{C_2},{C_1^*},{C_2^*},n,r})=&\,0,\\ 
\lim_{r\rightarrow 0} \mathcal{P}_4({{C_1},{C_2},{C_1^*},{C_2^*},n,r})=&\,4 {C_1}^2-8 {C_1} {C^*_1}+n ({C_2}-{C^*_2})^2+4 {C^*_1}^2\\
:=&\,\mathcal{P}_4^{(\text{lim})}({{C_1},{C_2},{C_1^*},{C_2^*},n}).
\end{aligned}
\end{equation*}
Therefore, we can obtain that
\begin{equation*}
\begin{aligned}
&\lim_{r\rightarrow 0}\mathcal{D}(C_1,C_2,C^*_1,C^*_2,n,r)\\
=&\, \mathcal{R}_1(\y_1,\dots,\y_{|\cX_k|}) \mathcal{P}_1^{(\lim)}({{C_1},{C_2},{C_1^*},{C_2^*}})+\mathcal{P}_4^{(\lim)}({{C_1},{C_2},{C_1^*},{C_2^*},n}).
\end{aligned}
\end{equation*}
Thus, it holds that
\begin{equation*}
\begin{aligned}
&{\text{\rm Error}_{\text{\rm ave}}^{(1)}(C_1,C_2,\varepsilon)}=\frac{2{\int_{0}^{1-\varepsilon-C_1^*}\int_{0}^{1-\varepsilon}\mathcal{P}_{1}^{(\lim)}(C_1,C_2,C_1^*,C_2^*){\rm d}C_1^*{\rm d}C_2^*}}{(1 - \varepsilon)^2}\\
&=\frac{\frac{(1-{\varepsilon}) \left({\varepsilon} (4 {C_1}+4 {C_2}-5)-2 ({C_1}+{C_2}-1)^2+{\varepsilon}^2\right)}{2 {\varepsilon}}+({C_1}+{C_2}-3) ({C_1}+{C_2}-1) \log ({\varepsilon})}{32 (1-{\varepsilon})^2 ({C_1}+{C_2}-1)^2},
\end{aligned}
\end{equation*}
and 
\begin{equation*}
\begin{aligned}
&{\text{\rm Error}_{\text{\rm ave}}^{(2)}(C_1,C_2,n,\varepsilon)}=\frac{2{\int_{0}^{1-\varepsilon-C_1^*}\int_{0}^{1-\varepsilon}\mathcal{P}_{4}^{(\lim)}(C_1,C_2,C_1^*,C_2^*,n){\rm d}C_1^*{\rm d}C_2^*}}{(1 - \varepsilon)^2}\\
=&\, \frac{1}{6} \left(24 {C_1}^2+16 {C_1} ({\varepsilon}-1)+6 {C_2}^2 n+{\varepsilon} (4 {C_2} n-2 n-8)\right.\left.-4 {C_2} n+{\varepsilon}^2 (n+4)+n+4\right).
\end{aligned}
\end{equation*}
We obtain the conclusion of this theorem. 
\end{proof}

We can then obtain the following result.
\begin{theorem}
\label{central-result}
If $\mathcal{R}_1(\y_1,\dots,\y_{|\cX_k|})\rightarrow 0$, then \(C_1={(1 - \varepsilon)}/{3},C_2={(1 - \varepsilon)}/{3}\) is the pair of weight coefficients defining the barycenter of the weight coefficient region.  
\end{theorem}
\begin{proof}
It holds that
\begin{equation*}
\lim_{\mathcal{R}_1(\y_1,\dots,\y_{|\cX_k|})\rightarrow 0}\text{Error}_{\text{ave}}(C_1,C_2,n,r,\varepsilon)=\text{Error}_{\text{ave}}^{(2)}(C_1,C_2,n,r,\varepsilon),
\end{equation*}
and
\begin{equation*}
\left(\frac{1 - \varepsilon}{3},\frac{1 - \varepsilon}{3}\right)^{\top}=\arg\min_{(C_1,C_2)^{\top}\in \mathcal{C}}  \text{Error}_{\text{ave}}^{(2)}(C_1,C_2,n,\varepsilon).
\end{equation*}
Therefore, the conclusion holds. 
\end{proof}

\begin{remark}
An interesting fact is that \(\left({(1 - \varepsilon)}/{3}, {(1 - \varepsilon)}/{3}\right)^{\top}\) is exactly also the barycenter of \(\mathcal{C}\), measured by Euclidean distance, since it holds that
\begin{equation*}
\begin{aligned}
\left(\frac{1-\varepsilon}{3}, \frac{1-\varepsilon}{3}\right)^{\top}
=\arg \min _{\left(C_{1}, C_{2}\right)^{\top} \in \mathcal{C}} \frac{\int_{0}^{1-\varepsilon-C_{1}^*} \int_{0}^{1-\varepsilon}\left[\left(C_{1}-C_{1}^{*}\right)^{2}+\left(C_{2}-C_{2}^{*}\right)^{2}\right] {\rm d} C_{1}^{*} {\rm d} C_{2}^{*}}{\int_{0}^{1-\varepsilon-C_{1}} \int_{0}^{1-\varepsilon}  {\rm d} C_{1}^{*} {\rm d} C_{2}^{*}}.
\end{aligned}
\end{equation*}
\end{remark}

An ideal case is that the lower bound of $C_3$, $\varepsilon$, converges to $0$, and the naturally balanced weight coefficients $C_1=C_2=C_3={1}/{3}$ are the limiting result of Theorem~\ref{central-result}.

\section{Model-based methods using the corrected ReMU models}
\label{Model-based methods using the corrected ReMU models}

This section proposes a way to numerically and iteratively correct the ReMU models and the corresponding framework, which can improve a ReMU-model-based derivative-free trust-region method's performance.

\subsection{Model-based algorithm using corrected ReMU models}
\label{Model-based algorithm using corrected ReMU models}

Algorithm~\ref{Trust-region framework with iteratively corrected-weighted ReMU} shows the trust-region framework with iteratively corrected-weighted regional minimal updating models. 
Notice that in Algorithm~\ref{Trust-region framework with iteratively corrected-weighted ReMU}, there will always be an accompanying model, and we will not generate the new queried point using this model in the current iteration. The actual reduction to predicted reduction ratio \(\rho_k\) of such a backup model will be calculated based on the queried iteration point provided by the trial model. Then, Algorithm~\ref{Trust-region framework with iteratively corrected-weighted ReMU} will adjust the model (by changing the corresponding weight coefficients) in the next iteration. To simplify the presentation, Algorithm~\ref{Trust-region framework with iteratively corrected-weighted ReMU} does not include more details about the criticality step or internal algorithm parameters, which can be found in Algorithm~\ref{algo-TR}.

\begin{breakablealgorithm}
    \caption{Trust-region framework with iteratively corrected weighted ReMU \label{Trust-region framework with iteratively corrected-weighted ReMU}}
    \begin{algorithmic}[1]
    \State \textbf{Input}: the initial point \(\xb_{0}\). Let \(k=1\).
\While{not terminate}
        \State Input (or obtain from the previous iteration) the corrected weight coefficients \( c_{1, c}^{(k)}, c_{2, c}^{(k)}, c_{3, c}^{(k)} \) satisfying \( c_{1, c}^{(k)} + c_{2, c}^{(k)} + c_{3, c}^{(k)} = 1 \).
        \State Let \(c_{i, {\text{trial}}}^{(k)}=c_{i, c}^{(k)}\), \quad for \( i=1,2,3\). 
        \State {\bf Step 1 (ReMU model construction step):} Obtain the trial quadratic model \( m_k^{\text{trial}}\) by solving the KKT conditions of the trial ReMU's subproblem
          \[
        \begin{aligned}
        \min_m \ &\sum_{i=1}^3 c^{(k)}_{i,{\rm trial}} | m - m_{k-1} |^2_{H^{i-1}(\Omega)}\\
         {\rm subject\ to\ } & m(\y_i) = f(\y_i), \; \forall\, \y_i \in \mathcal{X}_k .
        \end{aligned}
        \]

        \State {\bf Step 2 (Trial step):} Solve the {trial trust-region subproblem} and obtain \( \xb_{\text{new}}^{\text{trial}}\) by 
          \[
        \xb_{\text{new}}^{\text{trial}} = \xb_k + \arg \min_{\|\dd\| \leq \Delta_k} m^{\text{trial}}_k(\xb_k + \dd).
        \]
        
        \State {\bf Step 3 (ReMU model correcting step):} Obtain the { accompanying weight coefficients} \( c_{1, {\rm acc}}^{(k)}, c_{2, {\rm acc}}^{(k)}, c_{3, {\rm acc}}^{(k)} \) by letting  
        \[
         \left(c_{1, {\rm acc}}^{(k)}, c_{2, {\rm acc}}^{(k)}, c_{3, {\rm acc}}^{(k)}\right) \in \mathcal{C} \backslash \left\{\left(c_{1, {\text{trial}}}^{(k)}, c_{2, {\text{trial}}}^{(k)}, c_{3, {\text{trial}}}^{(k)}\right)\right\},  
        \] 
where 
         \begin{equation}\label{CReMU}
            \mathcal{C} = \left\{ (0, 0, 1), \left(\frac{1}{3}, \frac{1}{3}, \frac{1}{3}\right)\right\}. 
         \end{equation}

        Obtain the {accompanying quadratic model} \(m_k^{\text{acc}} \) by solving the KKT conditions of the accompanying ReMU's subproblem
         \[
        \begin{aligned}
        \min_m \ &\sum_{i=1}^3 c^{(k)}_{i,{\rm acc}} | m - m_{k-1} |^2_{H^{i-1}(\Omega)}\\
         {\rm subject\ to\ } & m(\y_i) = f(\y_i), \; \forall\, \y_i \in \mathcal{X}_k .
        \end{aligned}
        \] 
        Qualify the trial ReMU model and  the {accompanying ReMU model} using 
        \[
        \begin{aligned}
        \rho_k^{\text{trial}} &= \frac{f(\xb_{\text{new}}^{\text{trial}}) - f(\xb_k)}{m_k^{\text{trial}}(\xb_{\text{new}}^{\text{trial}}) - m_k^{\text{trial}}(\xb_k)},\\
                \rho_k^{\text{acc}} &= \frac{f(\xb_{\text{new}}^{\text{trial}}) - f(\xb_k)}{m_k^{\text{acc}}(\xb_{\text{new}}^{\text{trial}}) - m_k^{\text{acc}}(\xb_k)}.
                \end{aligned}
        \] 
        Update \( c_{1, c}^{(k)}, c_{2, c}^{(k)}, c_{3, c}^{(k)}\) based on the ratio comparison:
        $$
c_{1, c}^{(k)}, c_{2, c}^{(k)}, c_{3, c}^{(k)}=
\left\{\begin{aligned} 
&c_{1, {\rm trial}}^{(k)}, c_{2,  {\rm trial}}^{(k)}, c_{3,  {\rm trial}}^{(k)},\ {\rm if }\ \left|\rho_k^{\text {trial}}-1\right| \leq \left|\rho_k^{\rm acc}-1\right|, \\
& c_{1, {\rm acc}}^{(k)}, c_{2, {\rm acc}}^{(k)}, c_{3, {\rm acc}}^{(k)}, \text{\ otherwise. } 
\end{aligned}\right.
$$
        \State {\bf Step 4 (Update step):} Update the iteration \(k\), and update the iteration point, the trust-region center and radius, and the interpolation sets according to \eqref{updatecenter}, \eqref{updateradius}, and \eqref{updateset}, respectively, in Algorithm \ref{algo-TR}. 
\EndWhile
    \end{algorithmic}
\end{breakablealgorithm}

The main principle behind Algorithm~\ref{Trust-region framework with iteratively corrected-weighted ReMU} is that we correct, select, and use the ReMU model that best predicted the value of $f(\xb_{\text{new}}^{\text{trial}})$. In its most general form, this corresponds to finding the coefficients that solve the weight correction subproblem
\begin{equation}
\min_{{\bm c}\in \mathcal{C}} |\rho_k({\bm c})-1|. 
\end{equation}
This mechanism is more flexible than the one with a fixed set of weight coefficients during the whole optimization process. 
Rather than using the entire coefficient region defined in \Cref{Regional minimal updating}, in Algorithm~\ref{Trust-region framework with iteratively corrected-weighted ReMU} we limit ourselves to a discrete set of coefficients of those that performed best in our earlier numerical experiments and define \(\mathcal{C}\) by \eqref{CReMU}. Even with this small set, we are able to take advantage of the strengths of two different ReMU models with this relatively lightweight adjustment. By choosing the ReMU model that had the best prediction on the current iteration, we are using the most up-to-date information on ReMU models' ability to approximate $f$ in decision space areas of interest.

We note that in the implementation tested below, the coefficient correction step is called only in the case where there is an iteration point (given by solving the trust-region subproblem in the current iteration) that has not previously been evaluated.

\subsection{Numerical results}
\label{Numerical results}

We now present numerical experiments to compare the performance of the corrected ReMU models within the {\ttfamily POUNDerS} algorithm \cite{SWCHAP14} (an established model-based DFO solver) framework with the traditional least Frobenius norm model and the Barycentric ReMU model proposed above. 

We adopted the {\ttfamily POUNDerS} algorithmic framework, including certifying whether an interpolation set would result in a fully linear model and performing a model-improvement evaluation when needed. The only difference in our tested variants was the model type: once an interpolation set was determined, we formed the specified model type and employed it in the trust-region subproblem. We did this for {\ttfamily POUNDerS}' default least Frobenius change model, the barycentric ReMU model, and the iteratively corrected ReMU model following Algorithm~\ref{Trust-region framework with iteratively corrected-weighted ReMU}. 

All algorithmic parameters used are the default settings of {\ttfamily POUNDerS} (with a gradient norm tolerance of \(10^{-12}\) to ensure that the algorithms exhaust their budget of $50(n+1)$ evaluations). All other problem configurations are consistent with those detailed in \Cref{Performance and date profiles for test set}. This is similarly true for the performance and data profiles used to compare and observe the general numerical behavior across different models. 

\begin{figure}[htbp]
    \centering
        \includegraphics[width=0.45\textwidth]{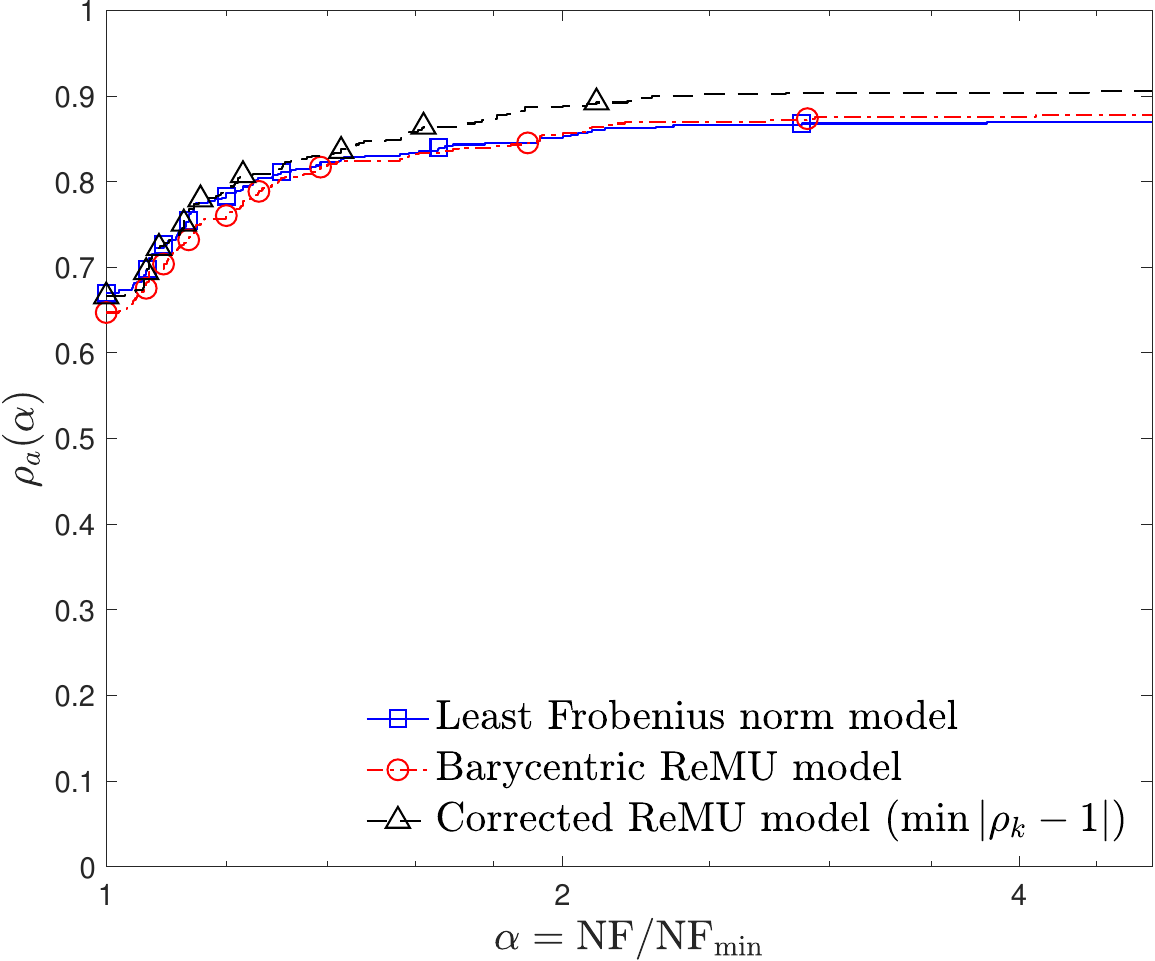} 
     \includegraphics[width=0.45\textwidth]{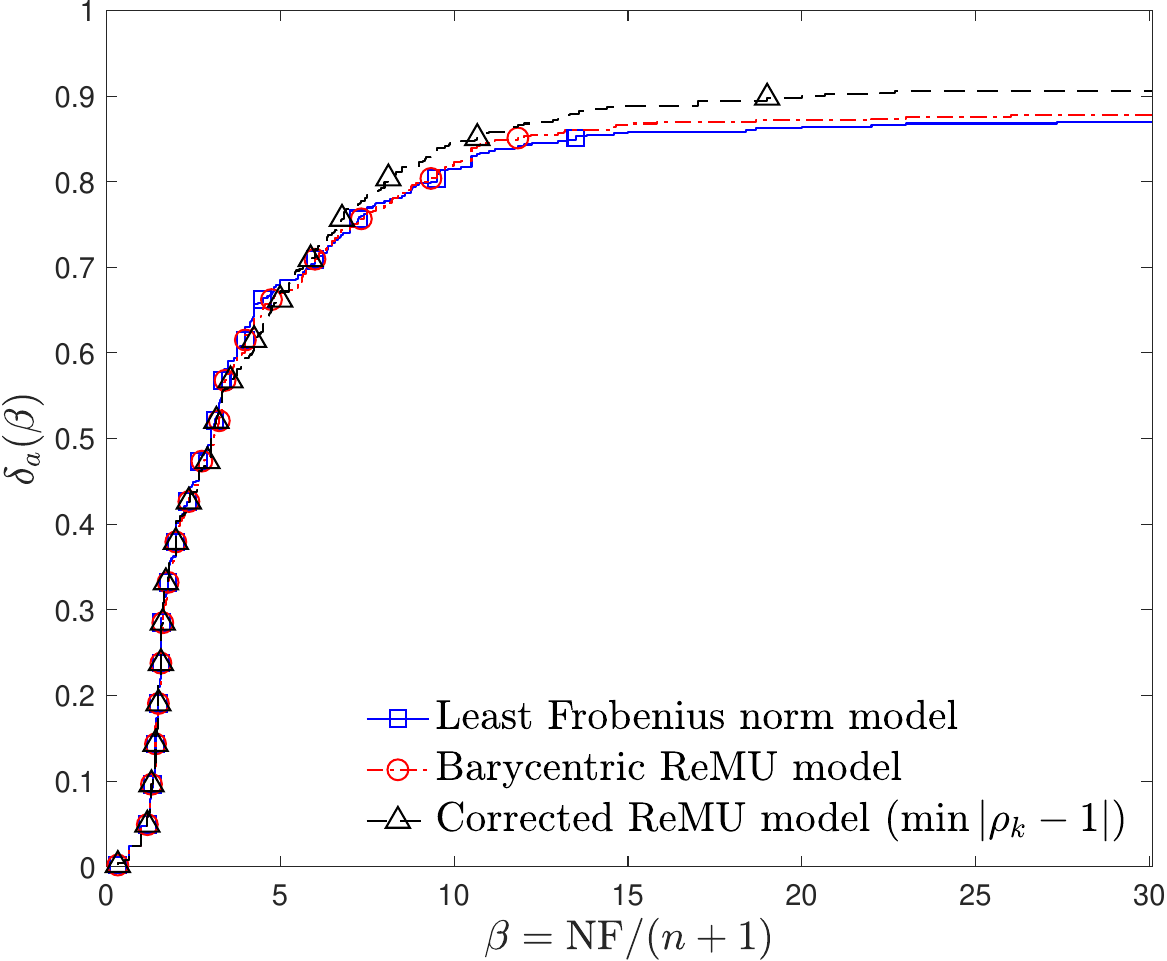} \\
         \includegraphics[width=0.45\textwidth]{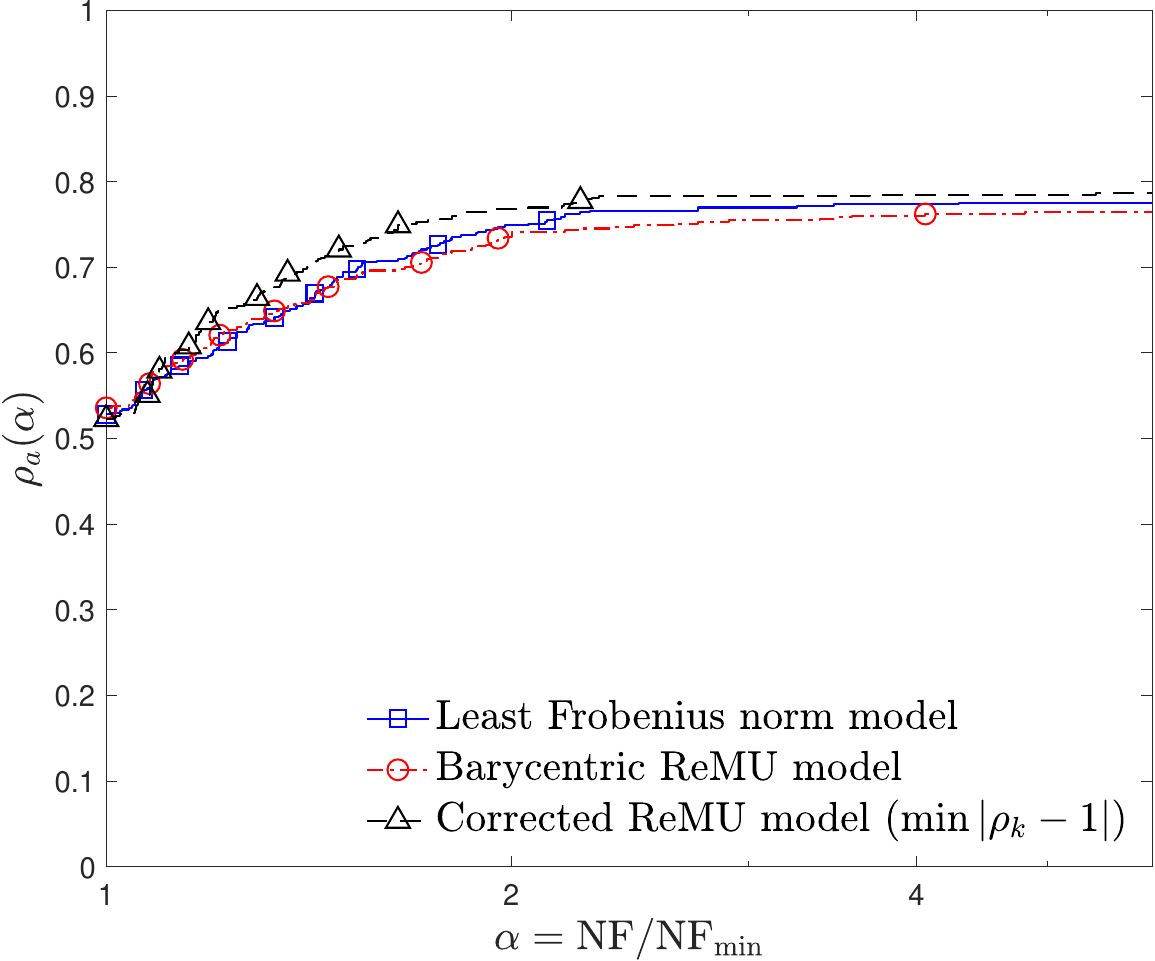} 
     \includegraphics[width=0.45\textwidth]{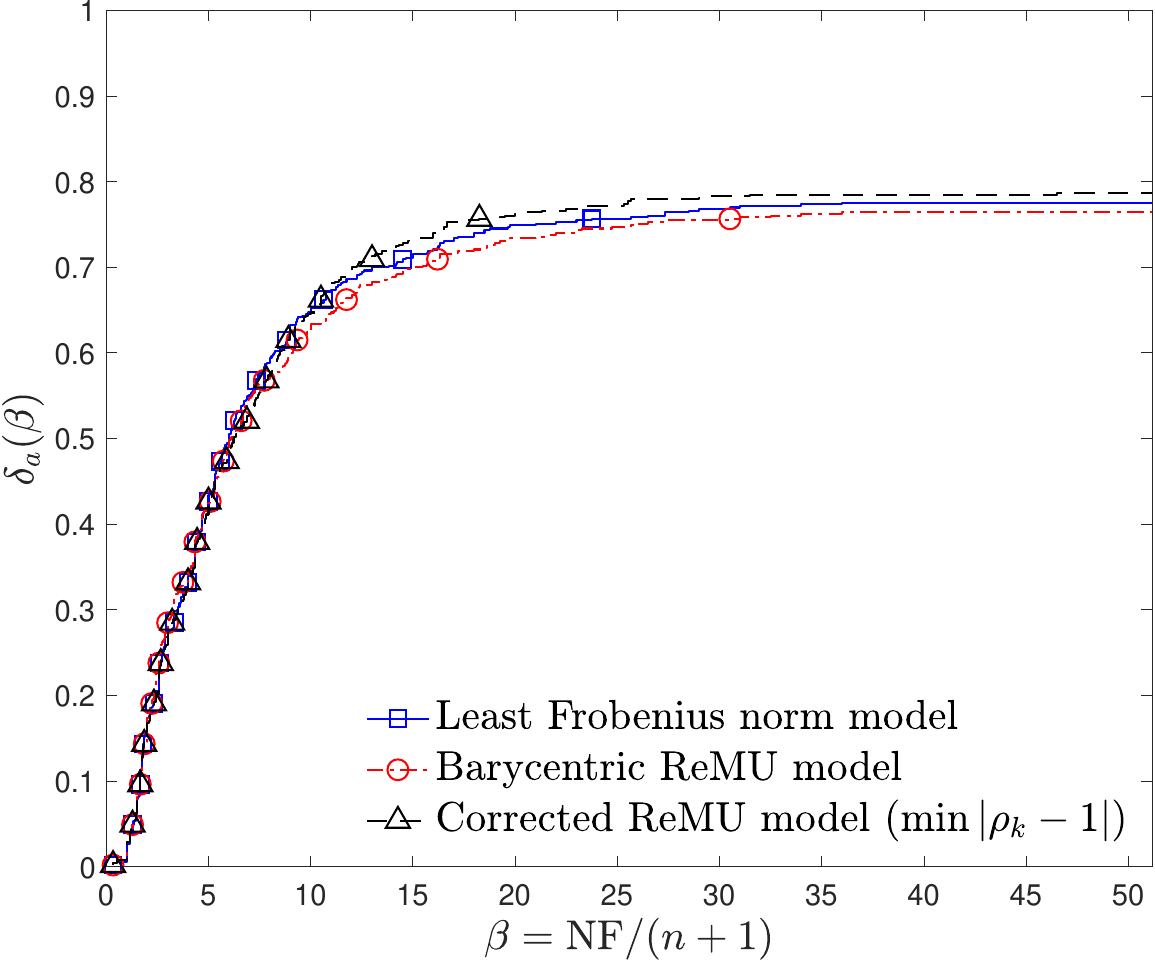} \\
         \includegraphics[width=0.45\textwidth]{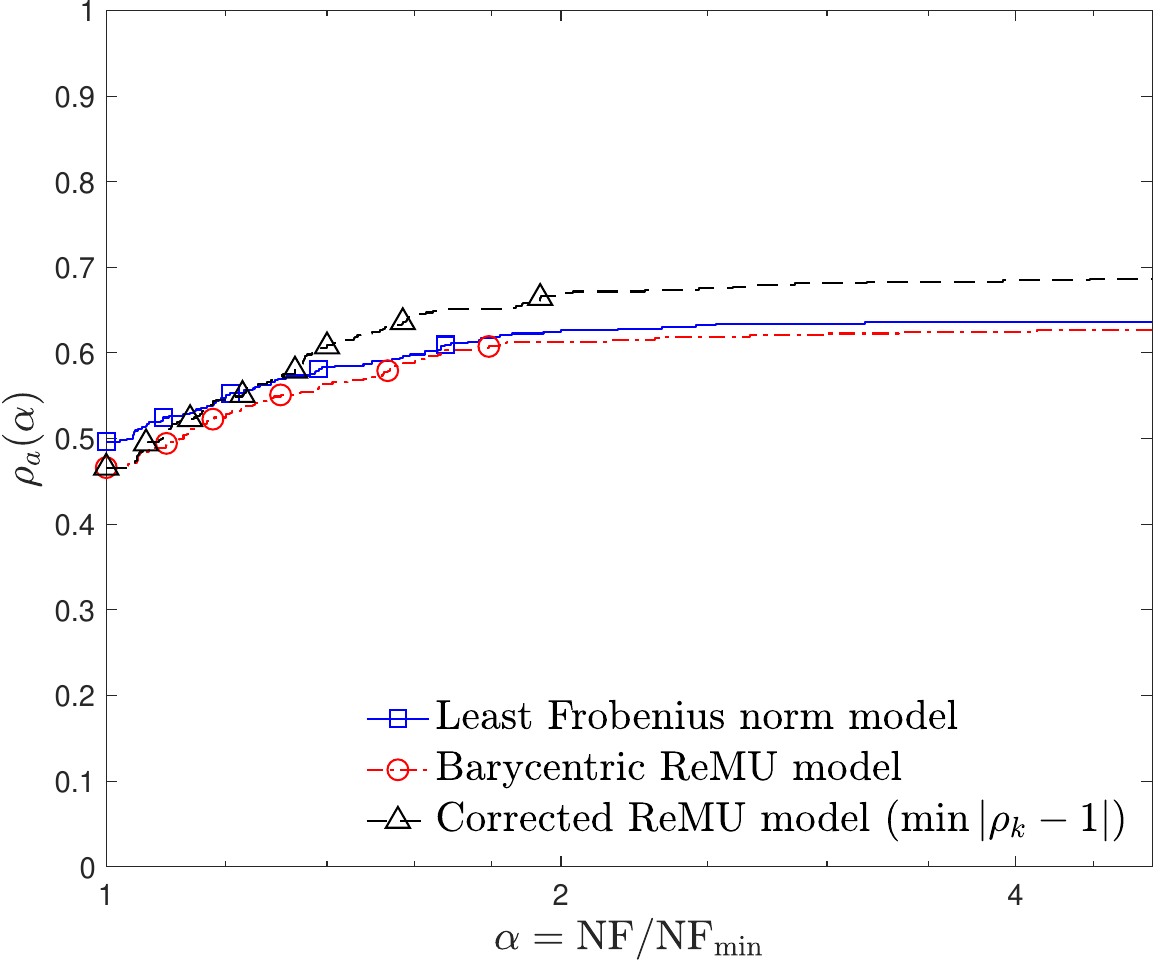} 
     \includegraphics[width=0.45\textwidth]{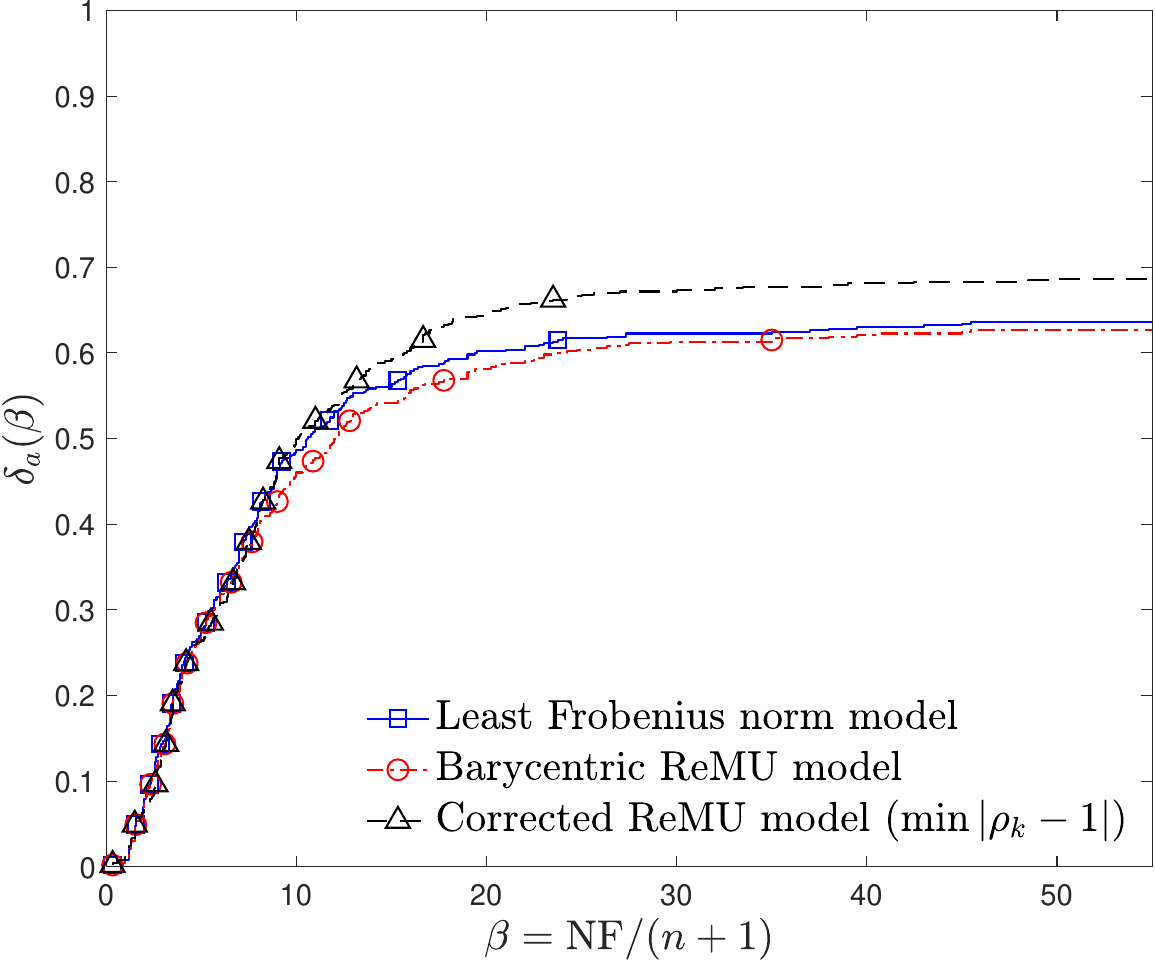} \\
         \includegraphics[width=0.45\textwidth]{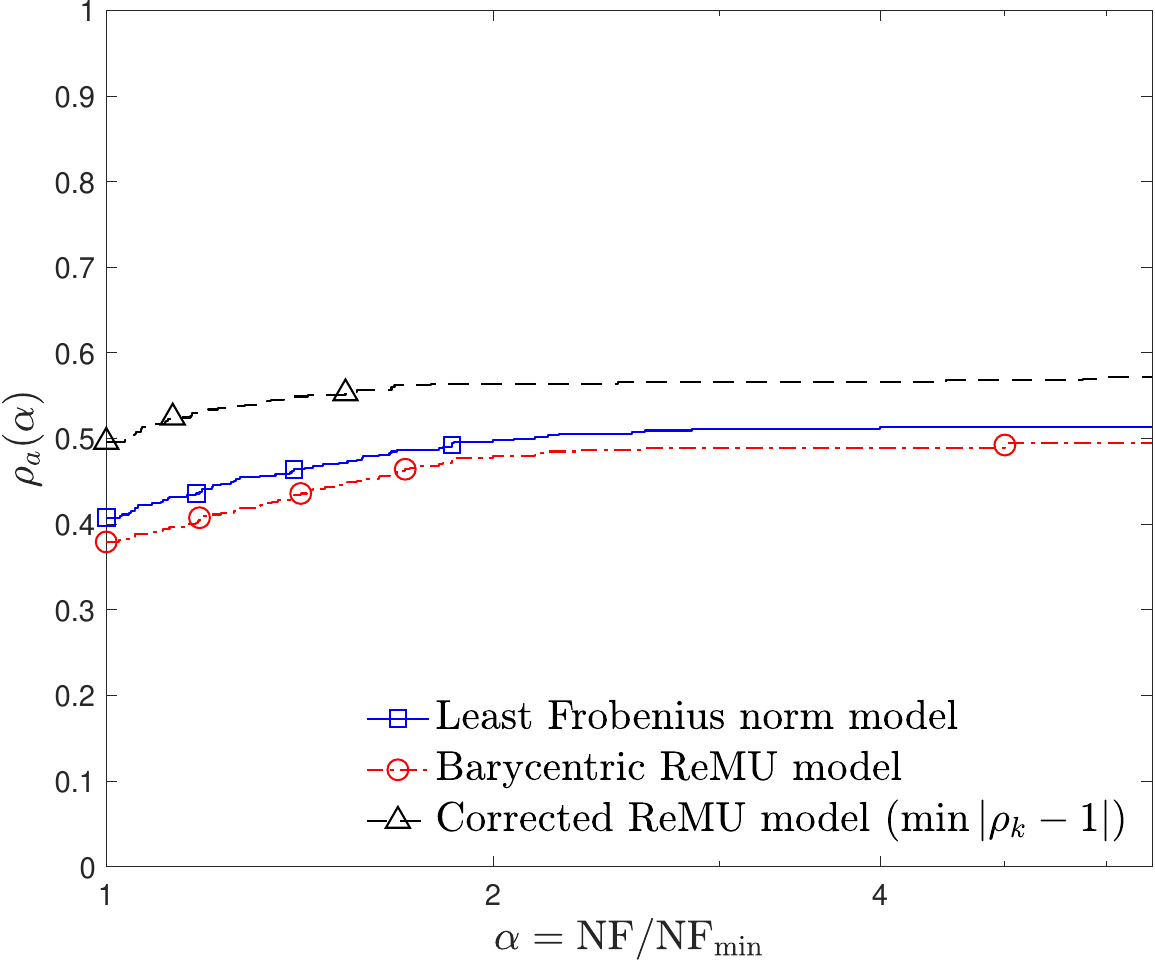}
         \includegraphics[width=0.45\textwidth]{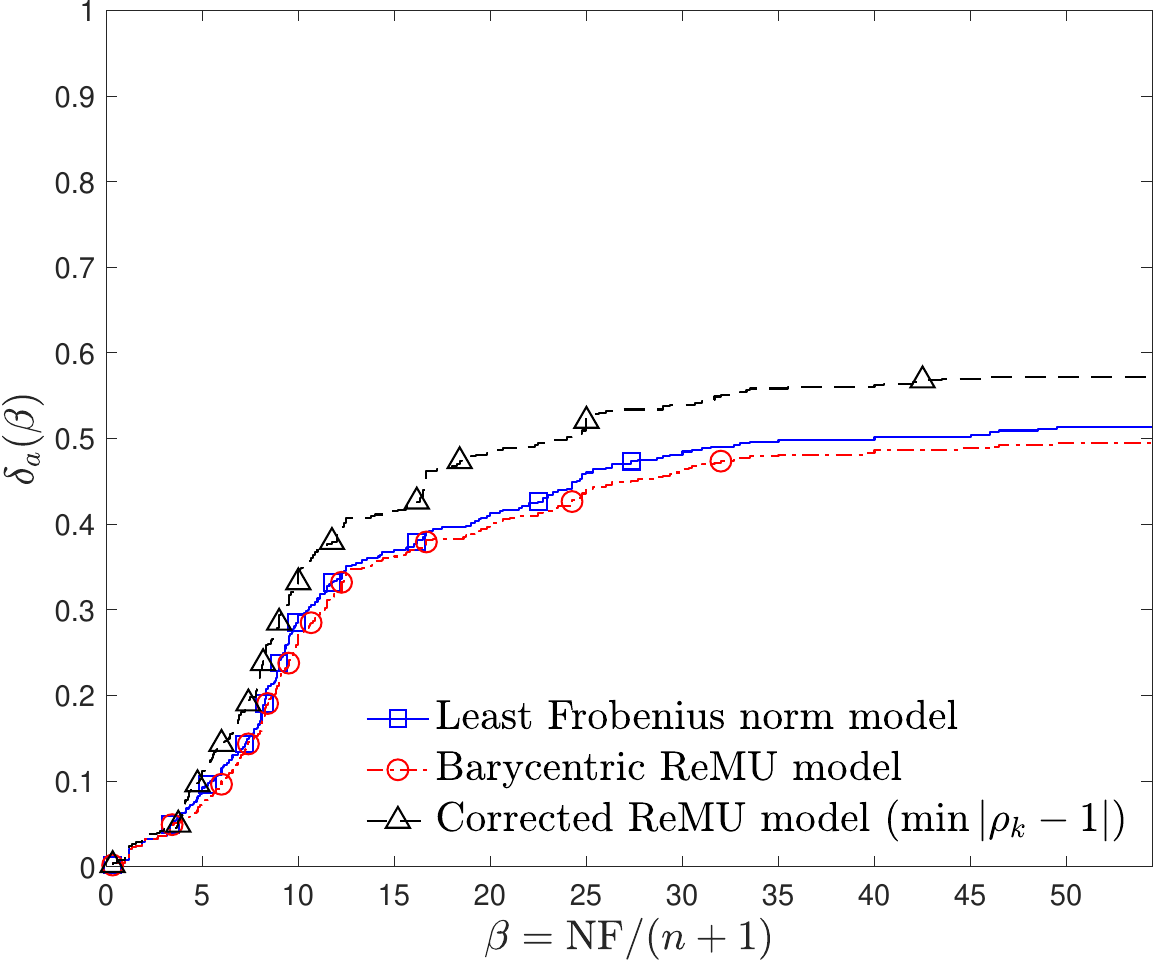}
           \caption{Performance (1st column) and data (2nd column) profiles with accuracy levels \(\tau=10^{-1},10^{-2},10^{-3},10^{-6}\) (from top to bottom) for \(|\cX_k|=2n+1\) interpolation points at each step; for the noisy problems,  \(\sigma=10^{-2}\).\label{perf-data-profile-A}}
\end{figure}

\begin{figure}[htbp]
     \centering
        \includegraphics[width=0.45\textwidth]{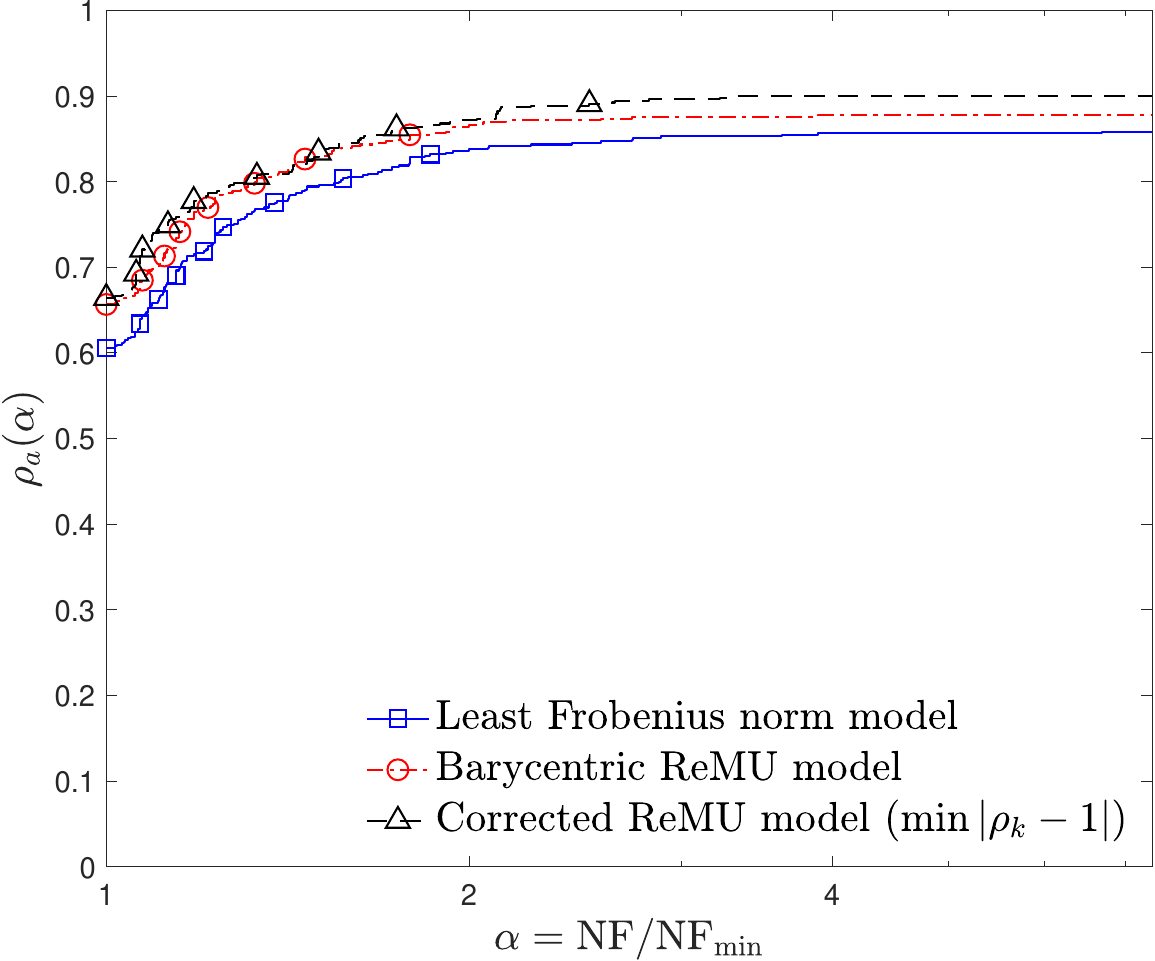} 
     \includegraphics[width=0.45\textwidth]{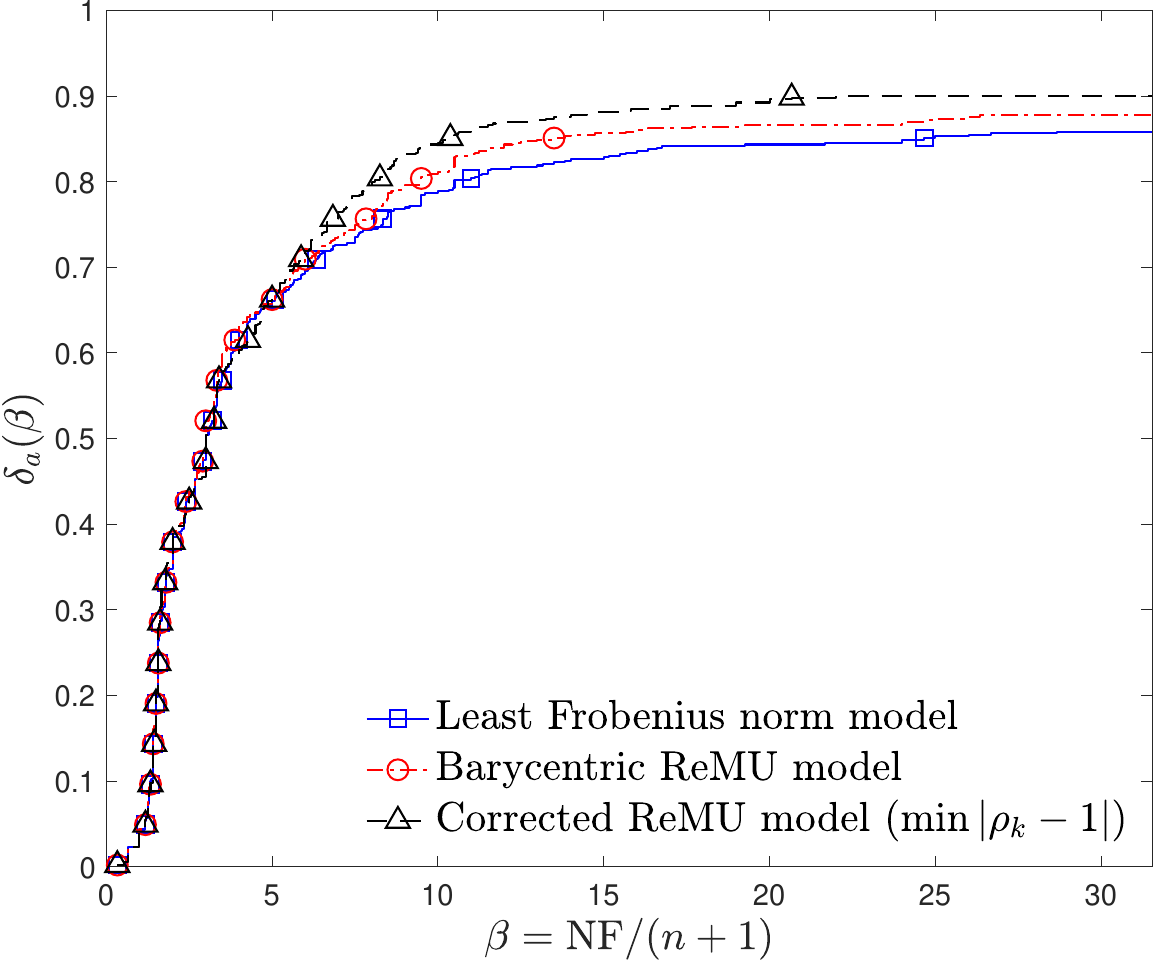} \\
         \includegraphics[width=0.45\textwidth]{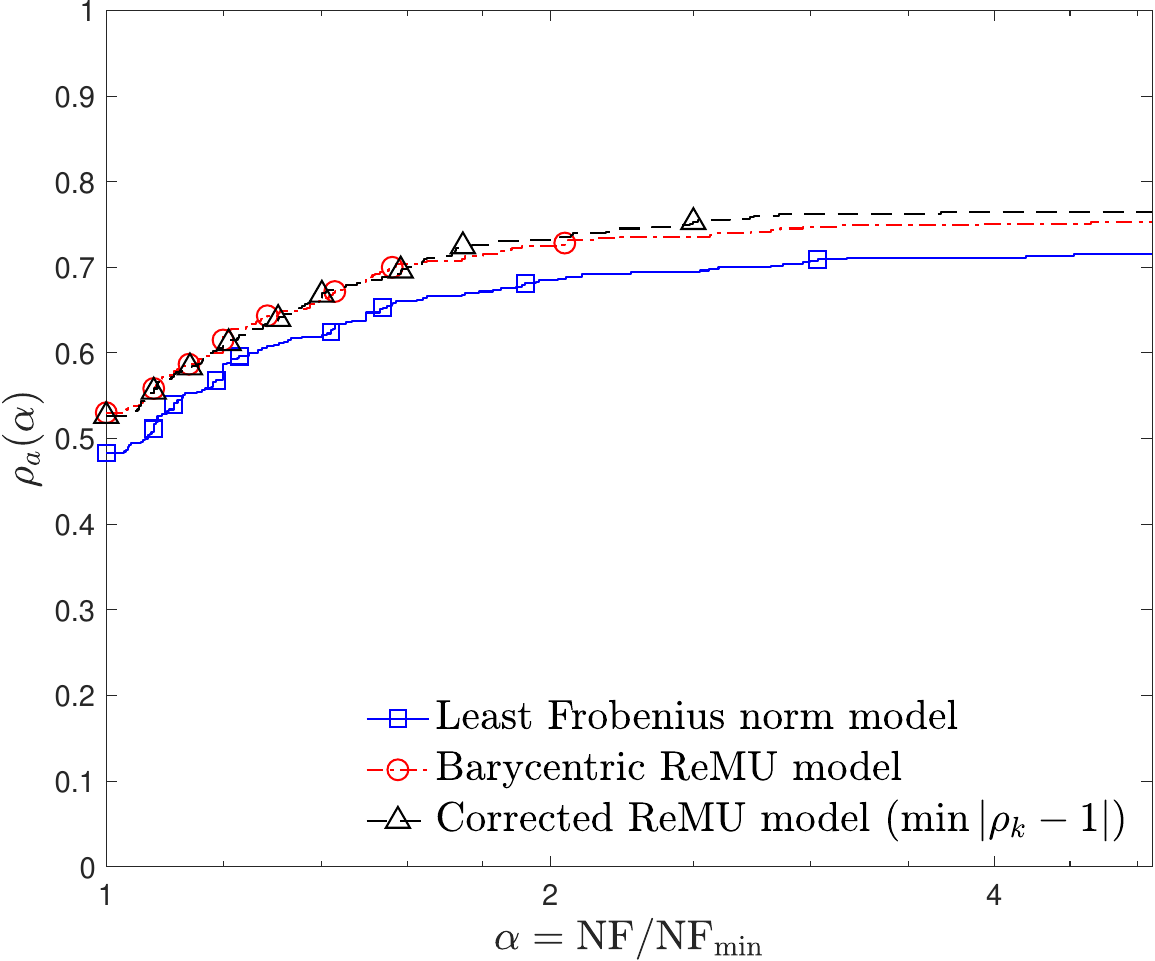} 
     \includegraphics[width=0.45\textwidth]{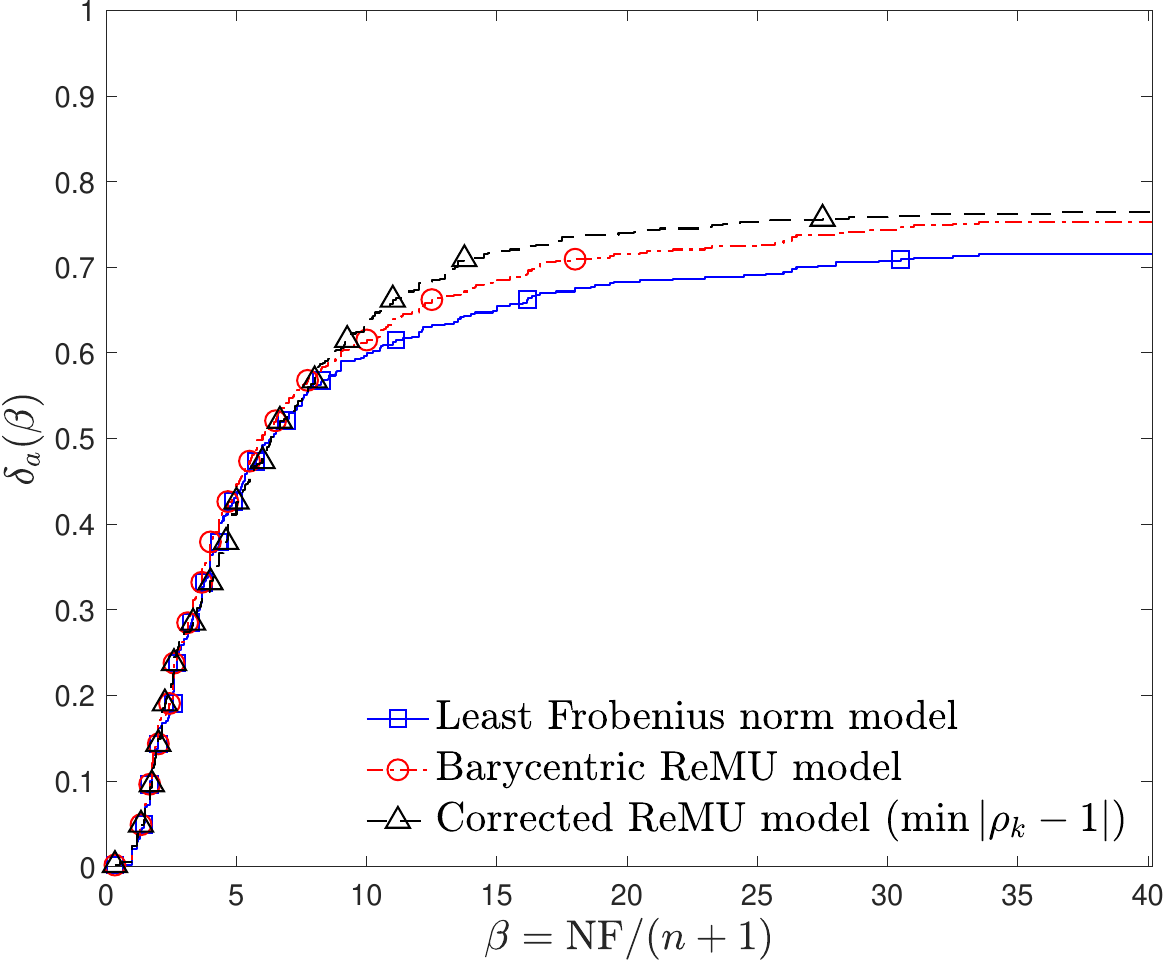} \\
         \includegraphics[width=0.45\textwidth]{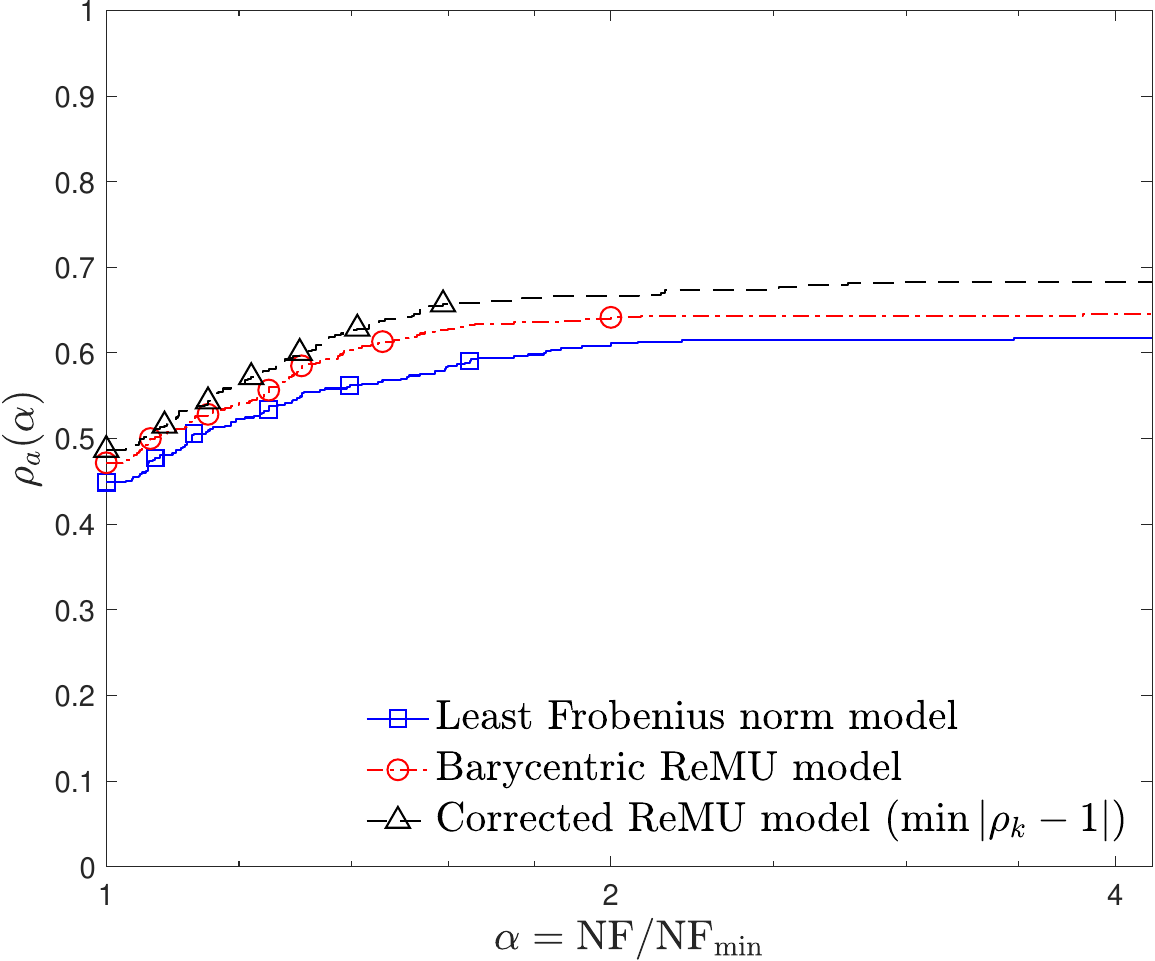} 
     \includegraphics[width=0.45\textwidth]{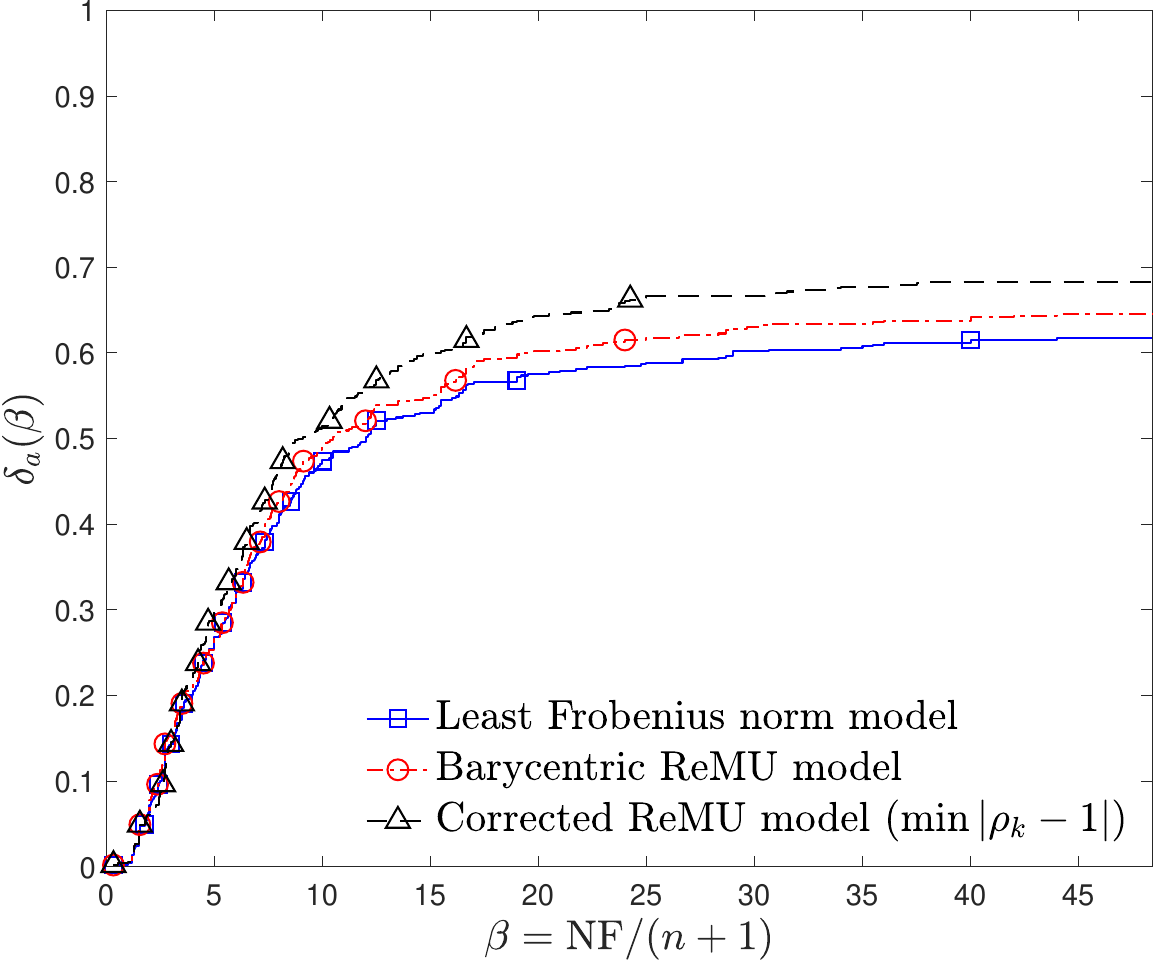} \\
         \includegraphics[width=0.45\textwidth]{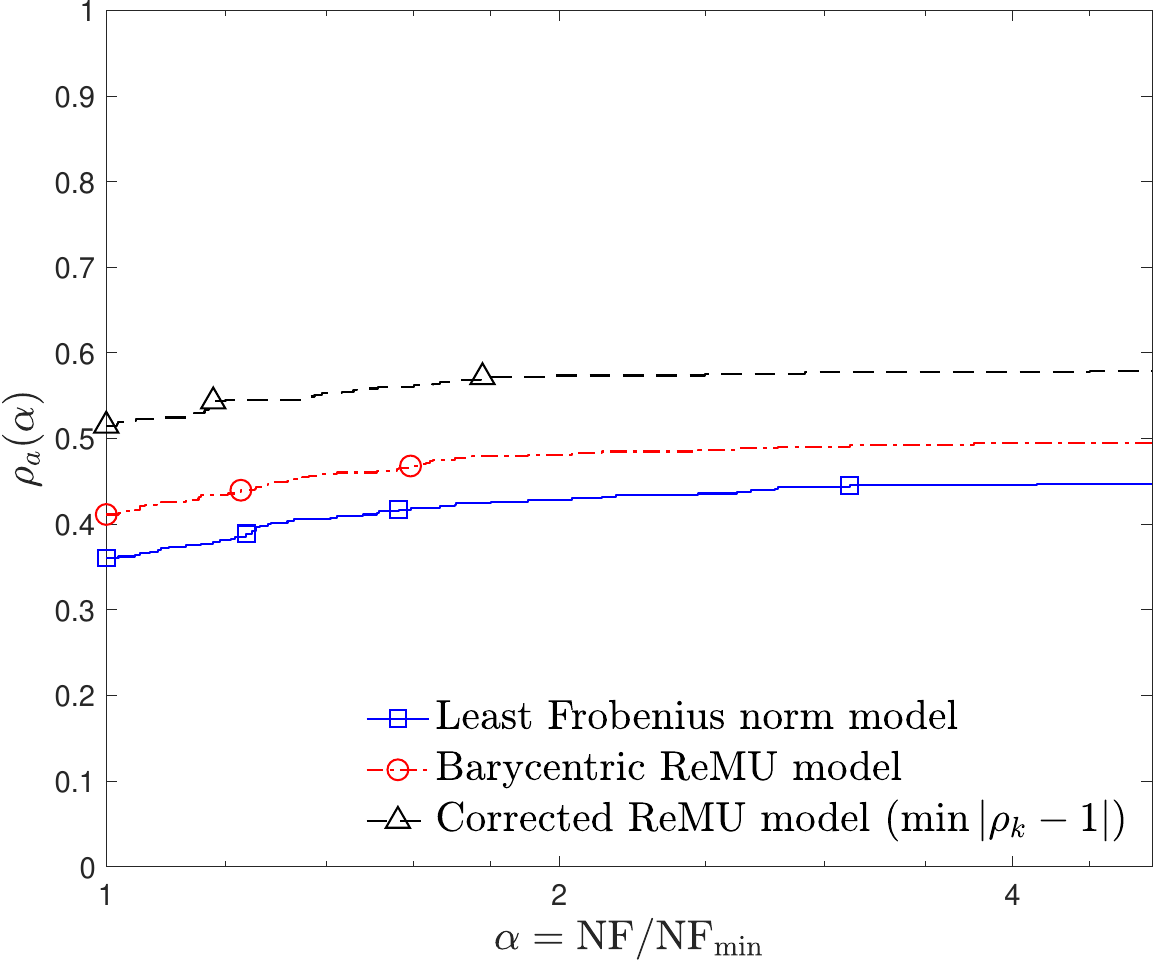}
         \includegraphics[width=0.45\textwidth]{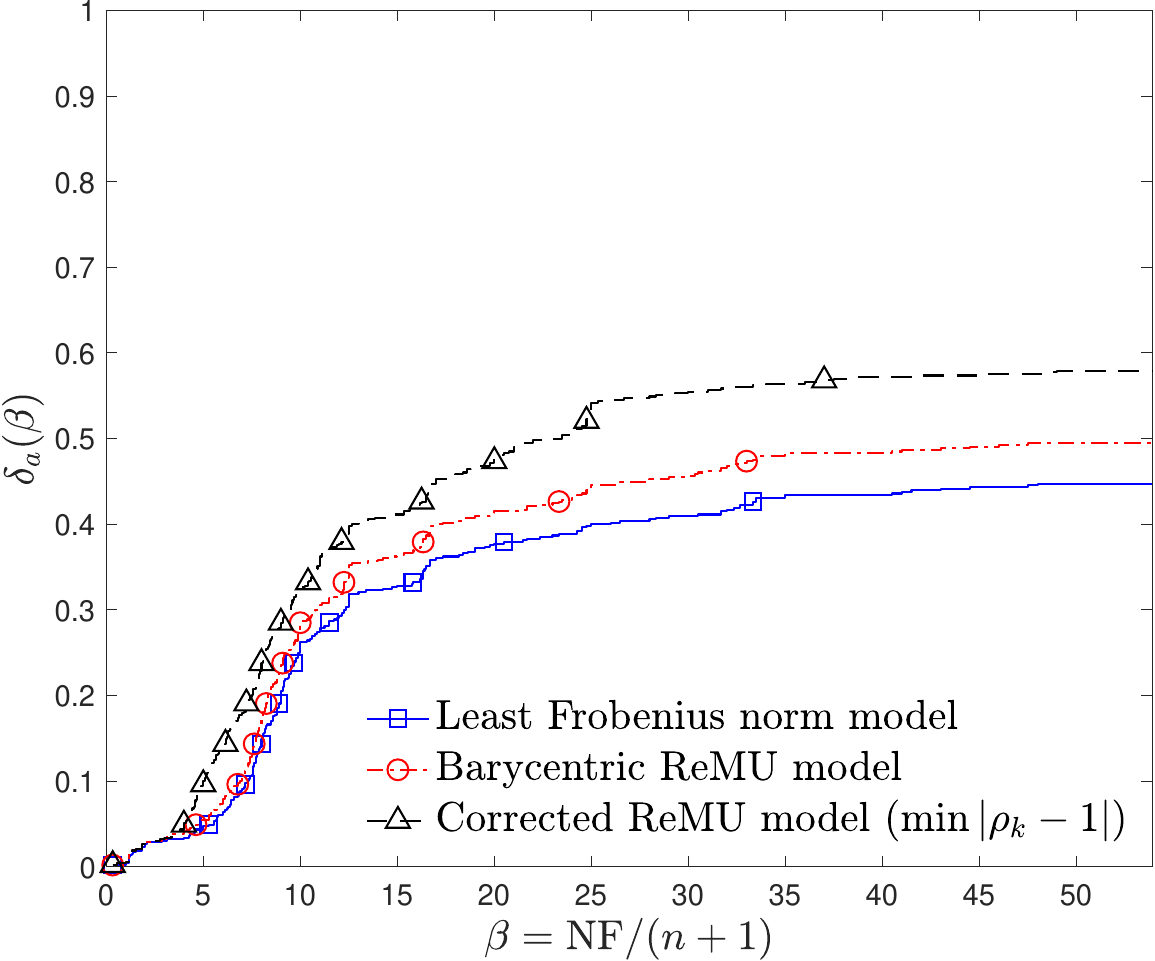} 
           \caption{Performance (1st column) and data (2nd column) profiles with accuracy levels \(\tau=10^{-1},10^{-2},10^{-3},10^{-6}\) (from top to bottom) for \(|\cX_k|=n+3\) interpolation points at each step; for the noisy problems,  \(\sigma=10^{-2}\).\label{perf-data-profile-B}}
\end{figure}

\begin{figure}[htbp]
     \centering
        \includegraphics[width=0.45\textwidth]{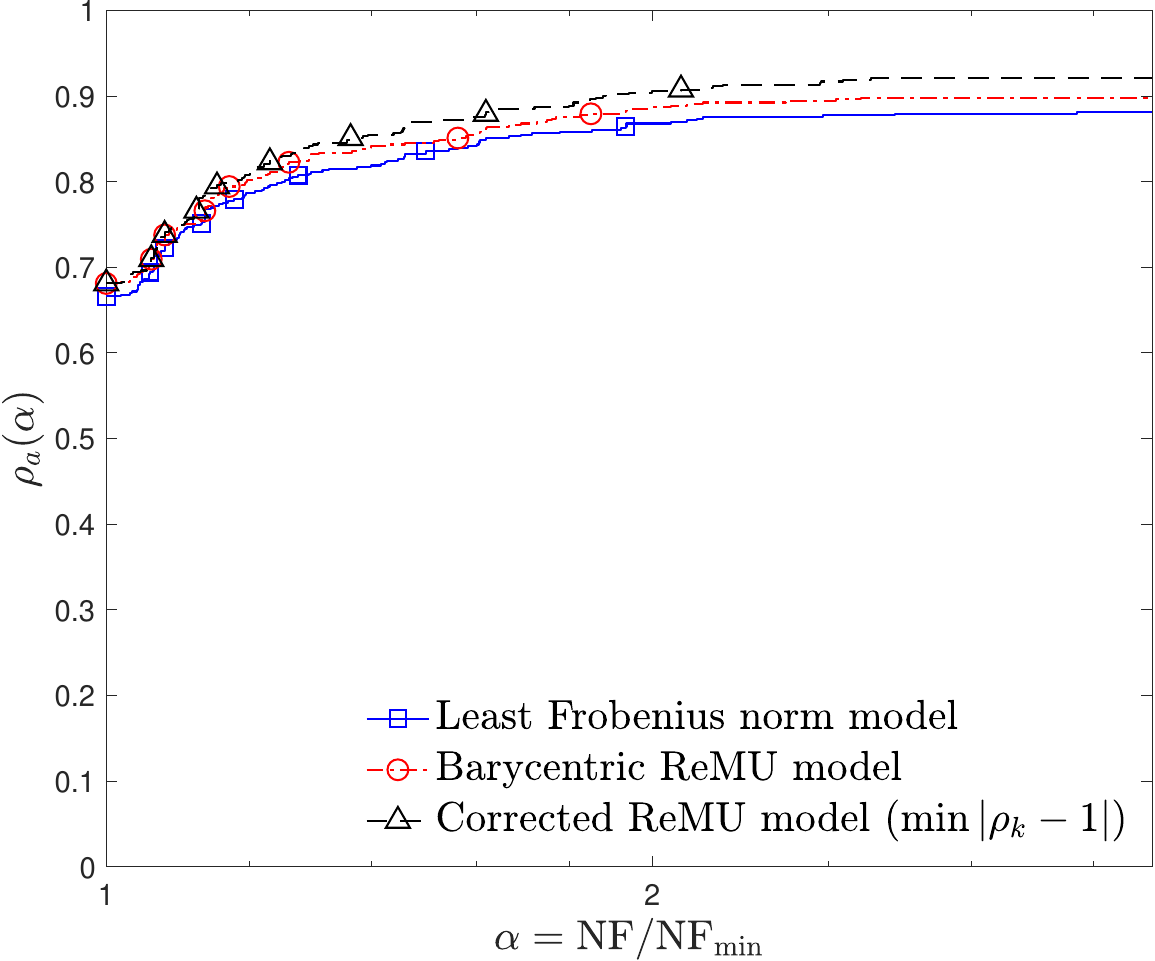} 
     \includegraphics[width=0.45\textwidth]{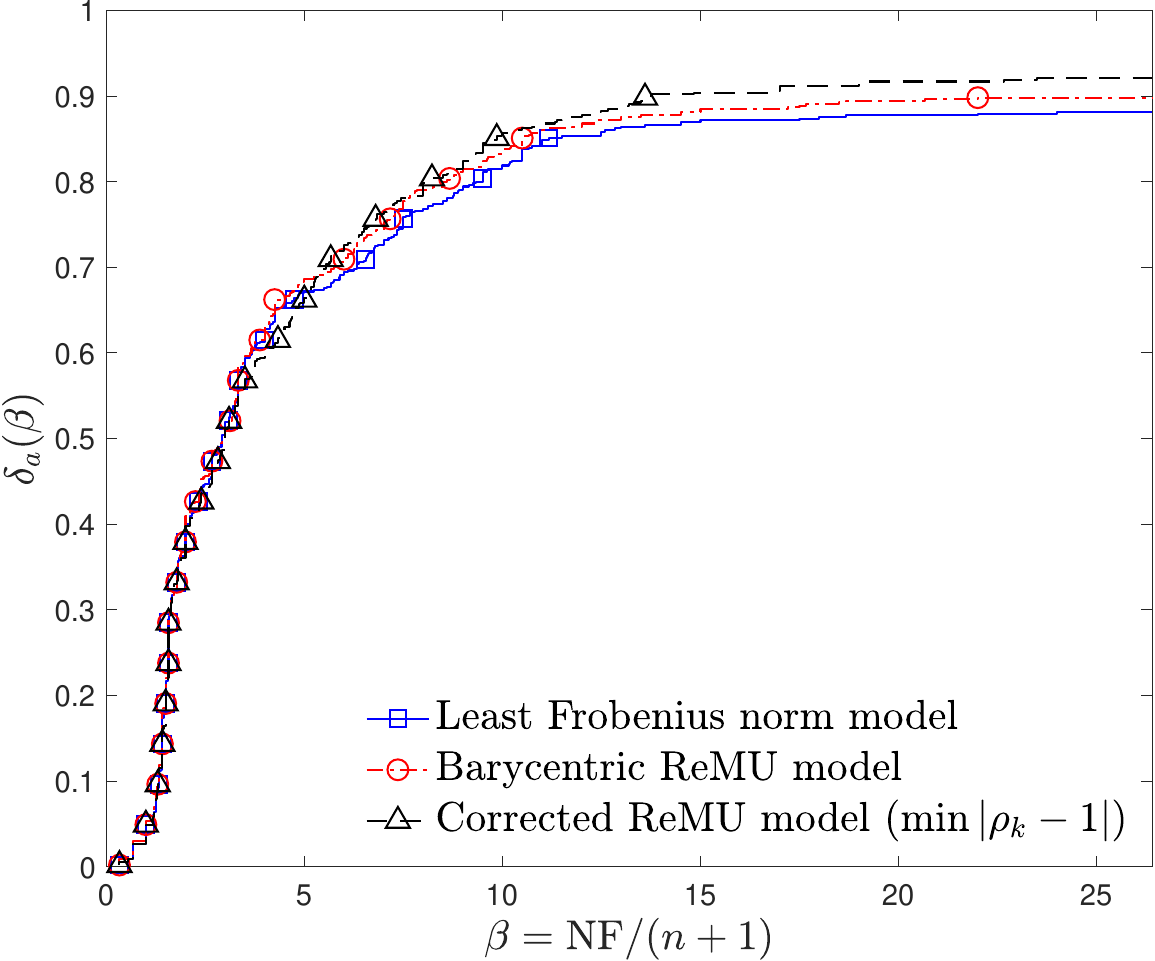} \\
         \includegraphics[width=0.45\textwidth]{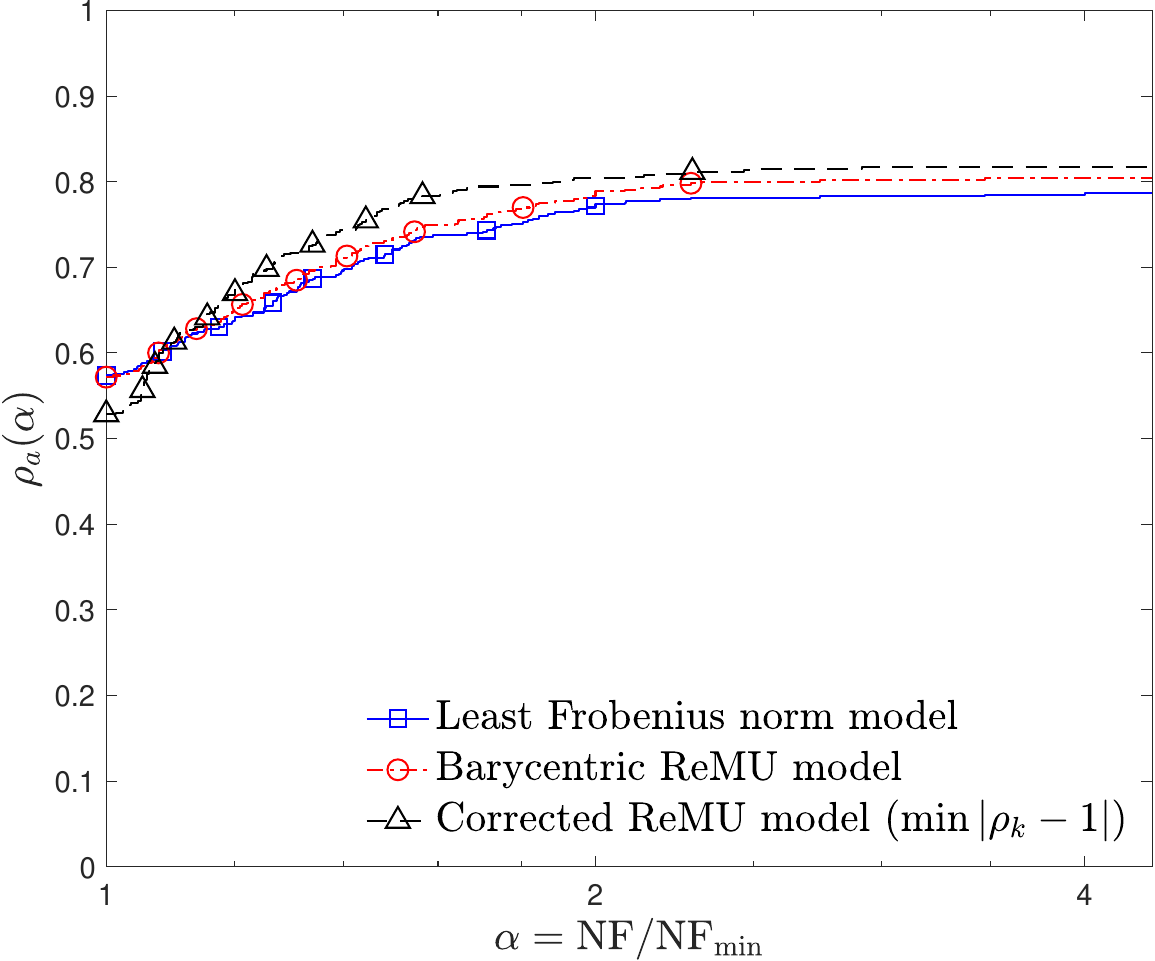} 
     \includegraphics[width=0.45\textwidth]{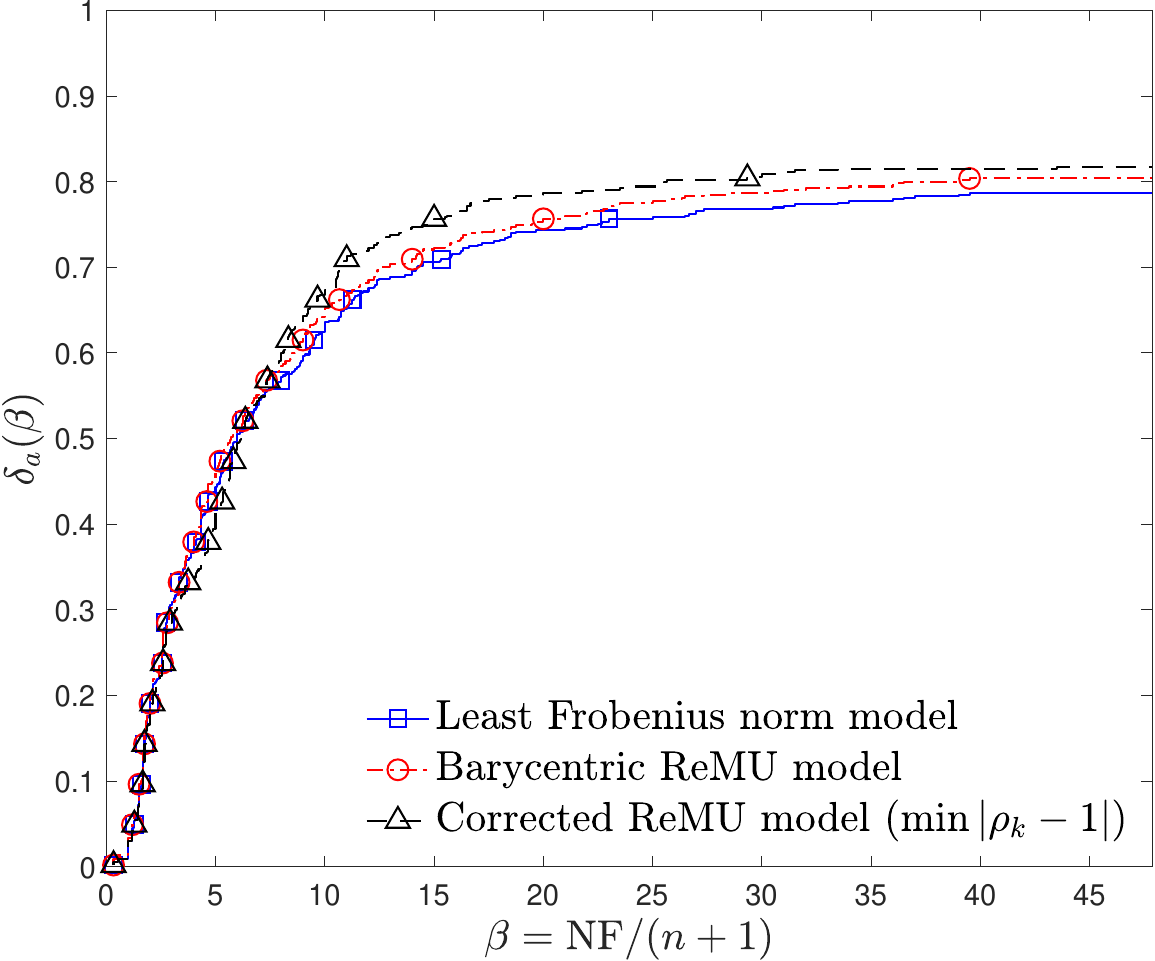} \\
         \includegraphics[width=0.45\textwidth]{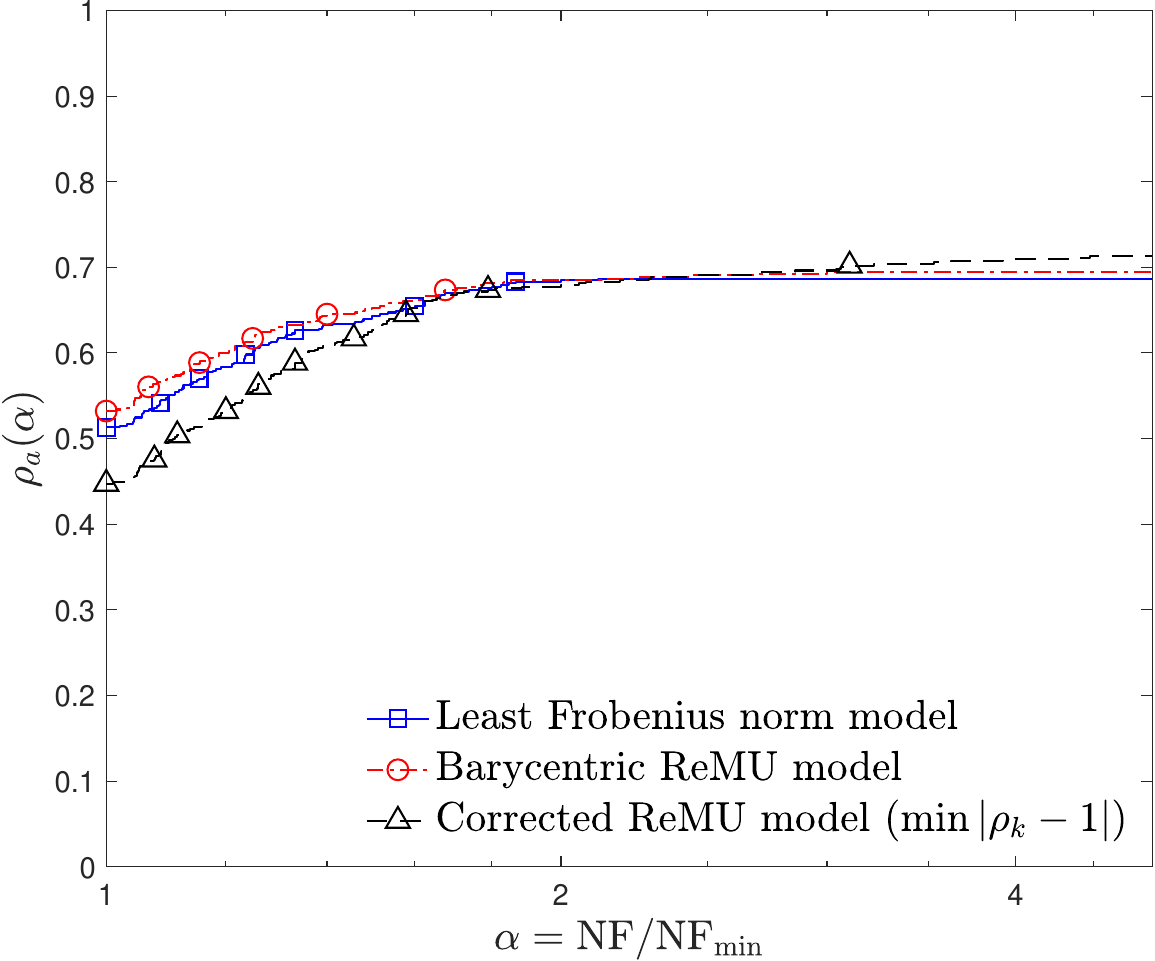} 
     \includegraphics[width=0.45\textwidth]{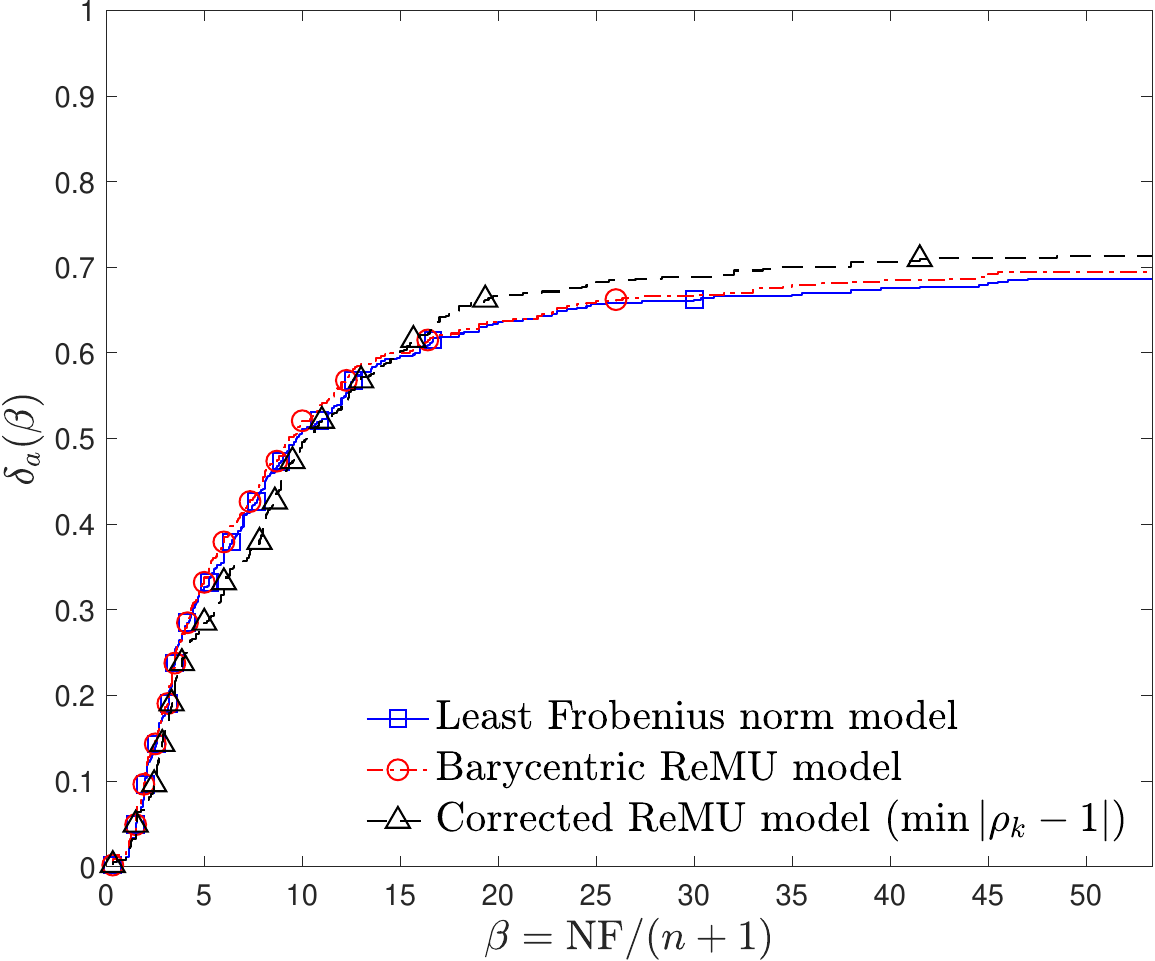} \\
         \includegraphics[width=0.45\textwidth]{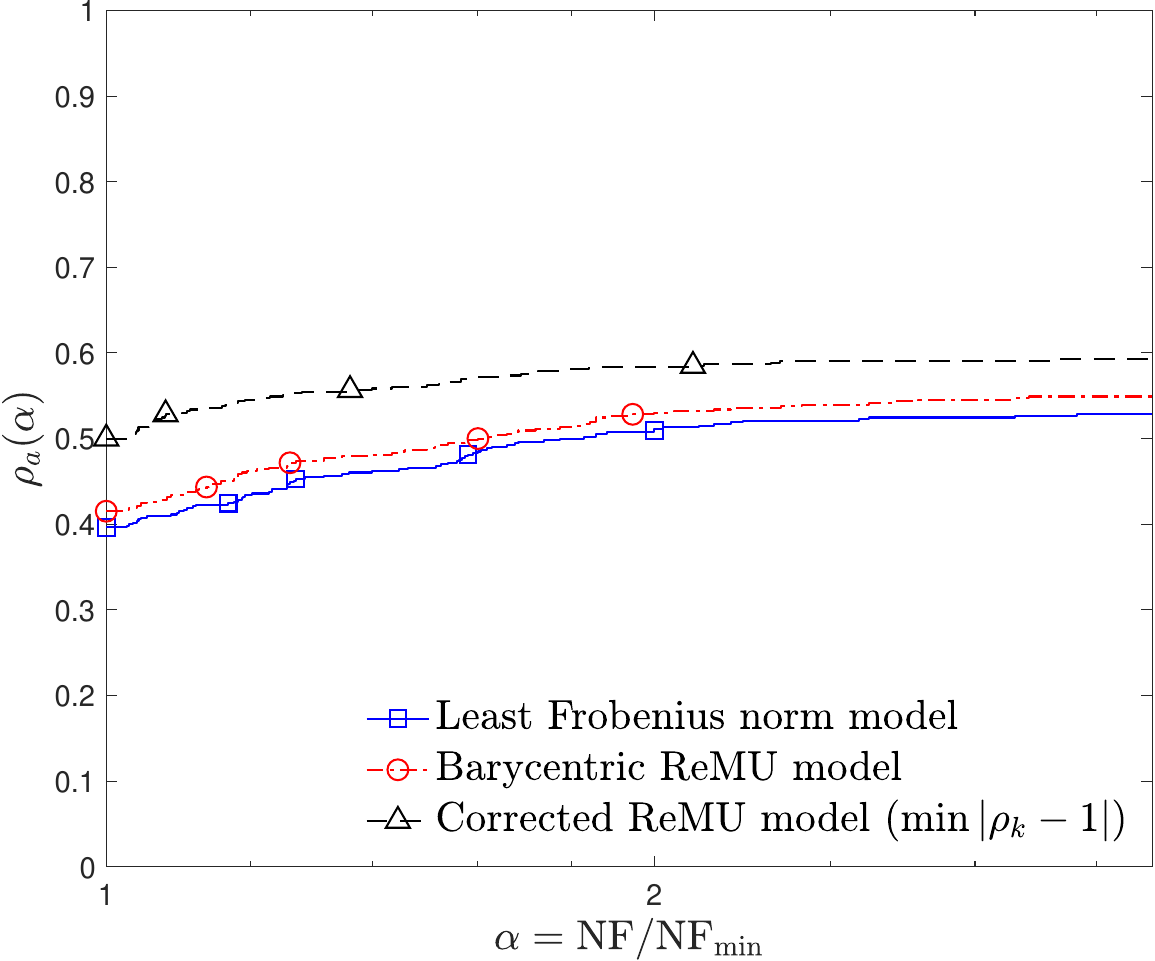}
         \includegraphics[width=0.45\textwidth]{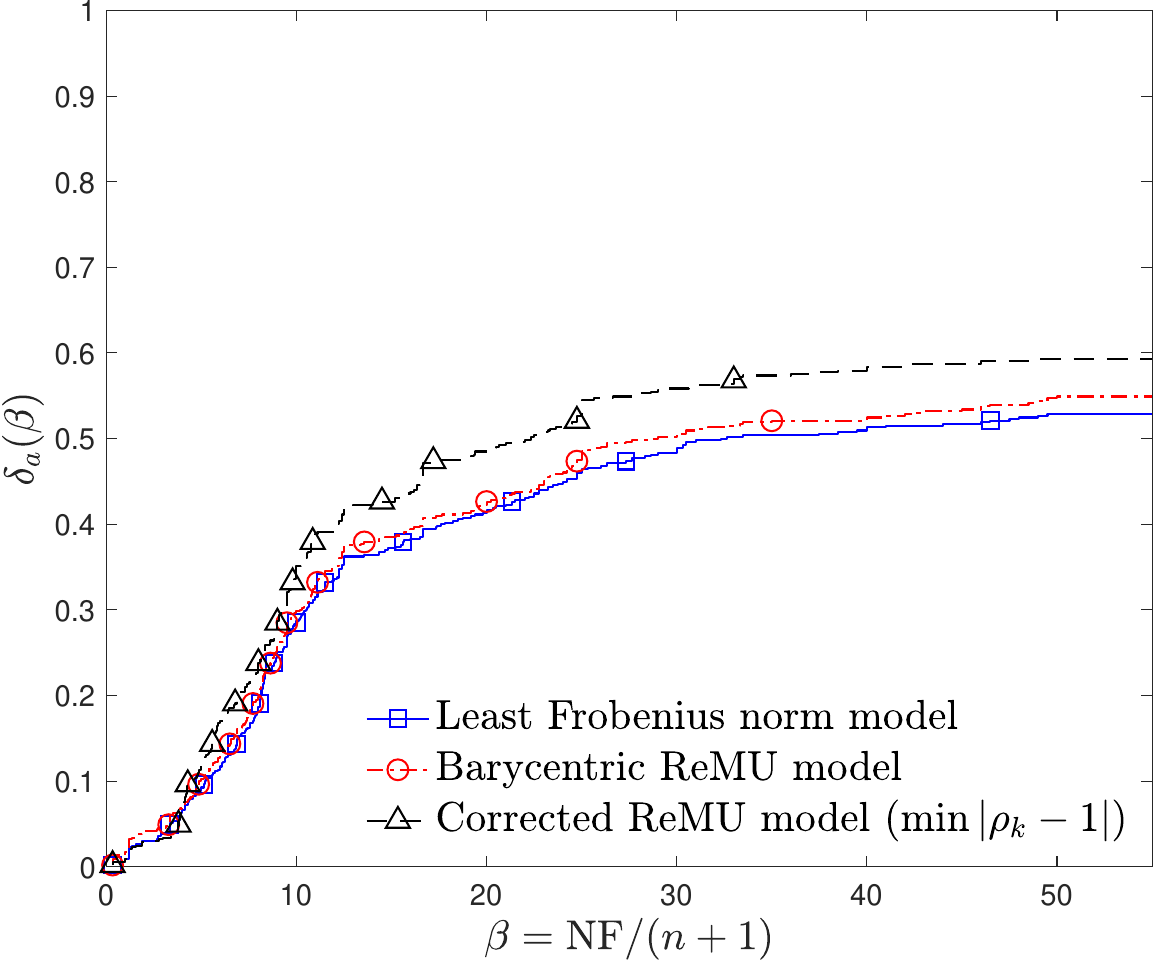} 
           \caption{Performance (1st column) and data (2nd column) profiles with accuracy levels \(\tau=10^{-1},10^{-2},10^{-3},10^{-6}\) (from top to bottom) for \(|\cX_k|=2n+1\) interpolation points at each step; for the noisy problems,  \(\sigma=10^{-4}\).\label{perf-data-profile-C}}
\end{figure}

\begin{figure}[htbp]
     \centering
        \includegraphics[width=0.45\textwidth]{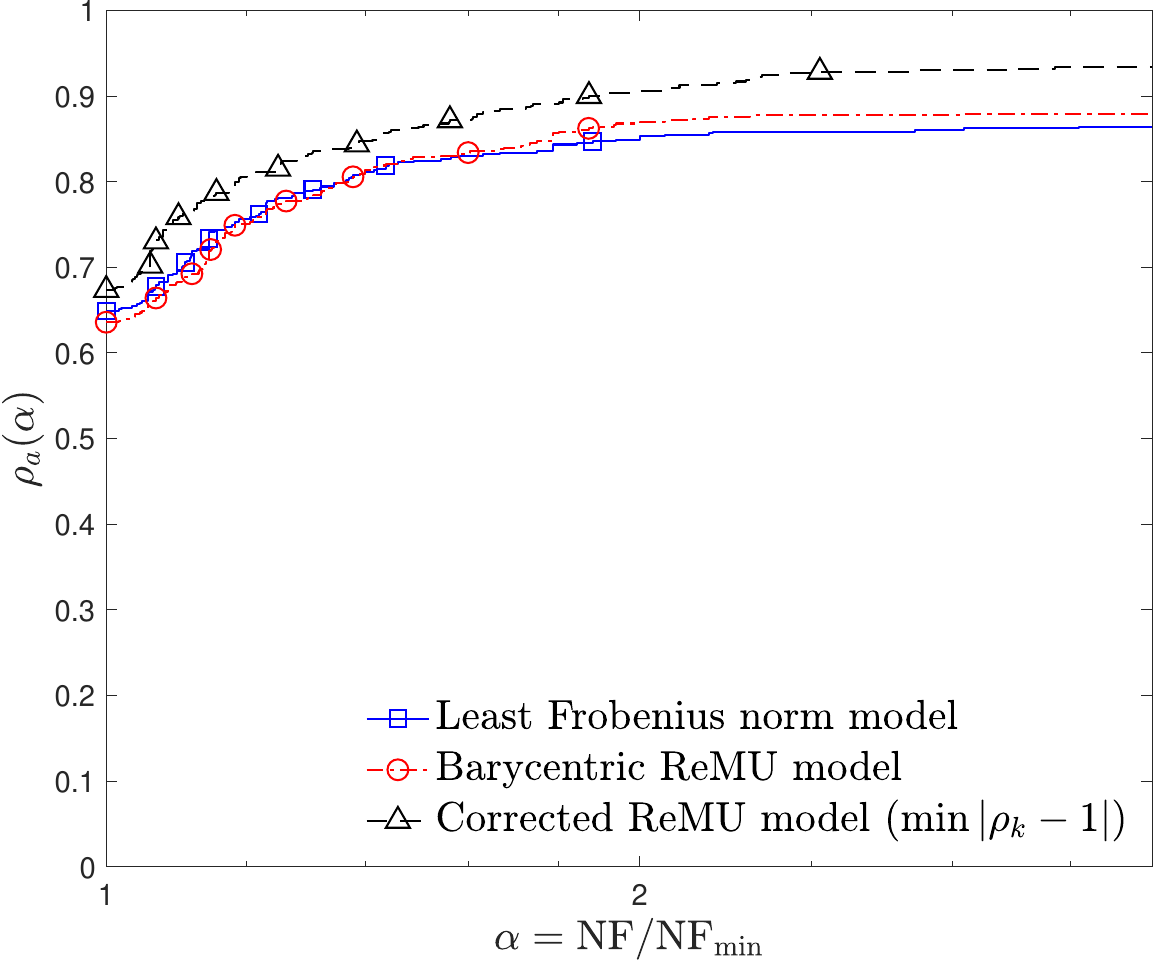} 
     \includegraphics[width=0.45\textwidth]{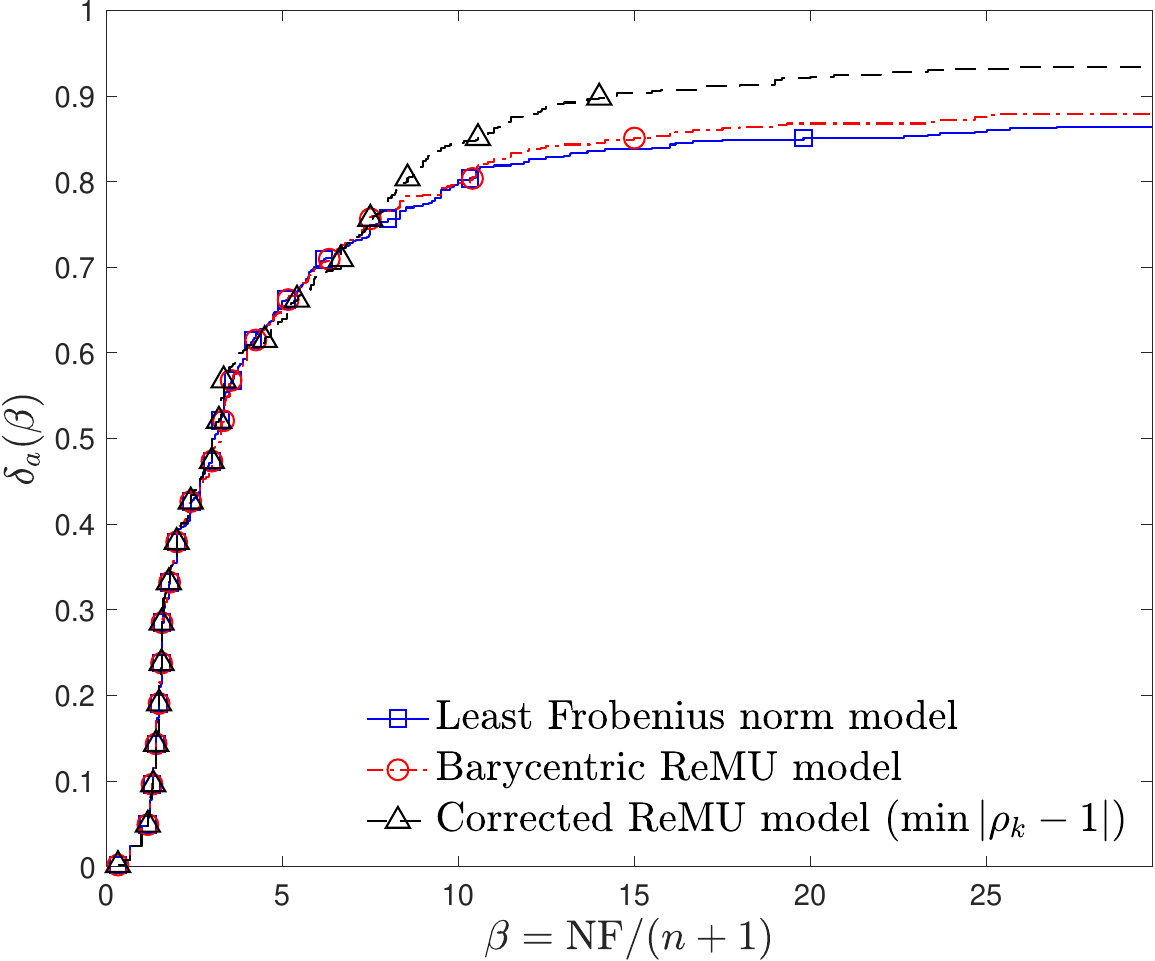} \\
         \includegraphics[width=0.45\textwidth]{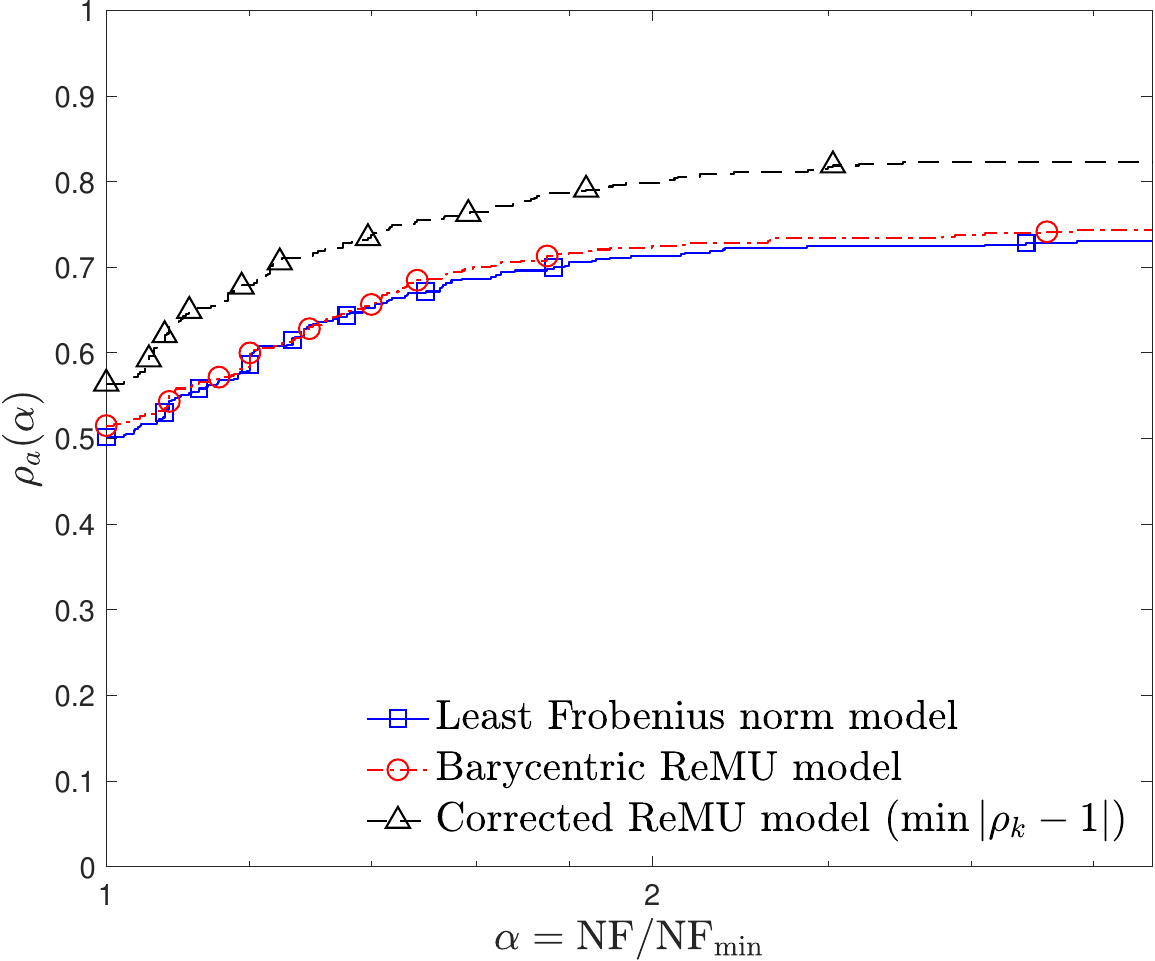} 
     \includegraphics[width=0.45\textwidth]{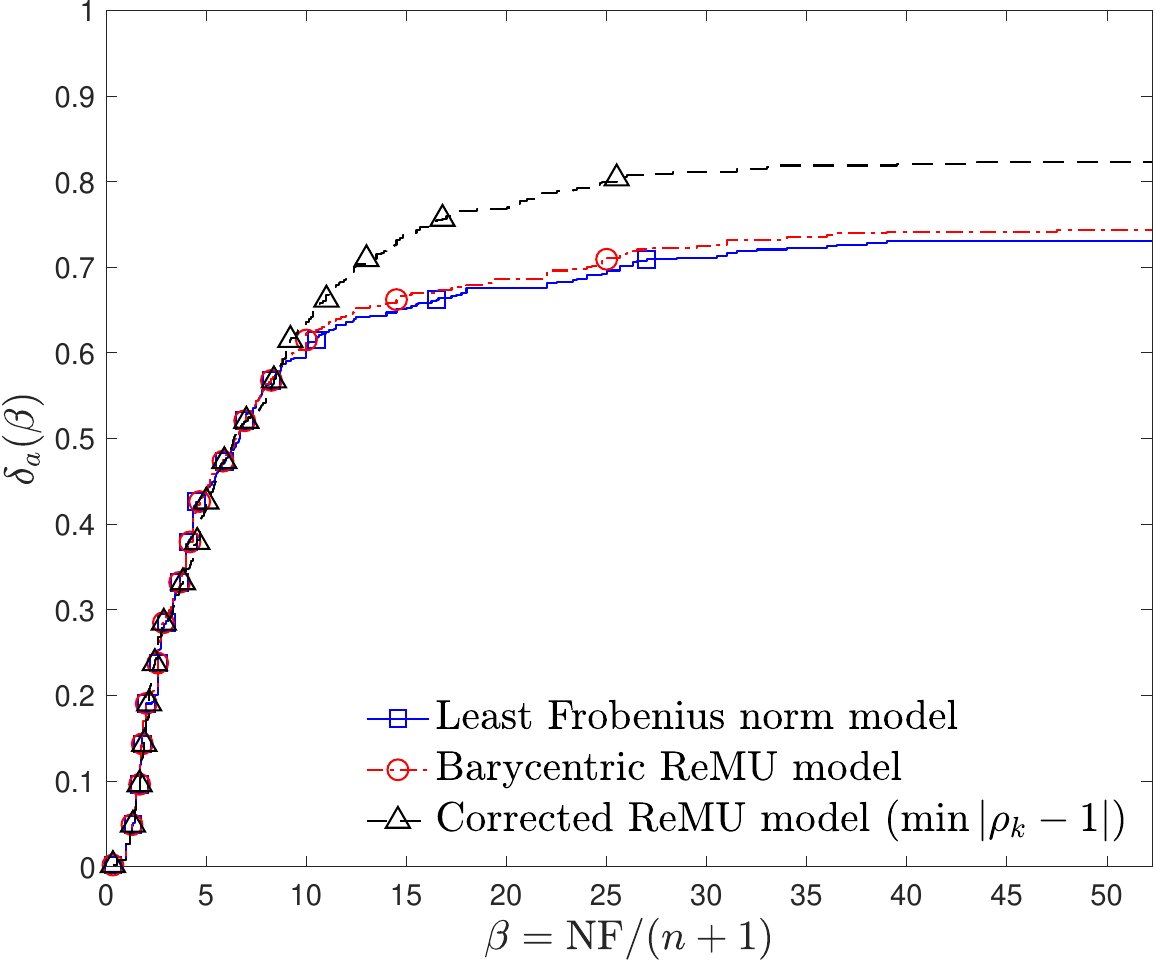} \\
         \includegraphics[width=0.45\textwidth]{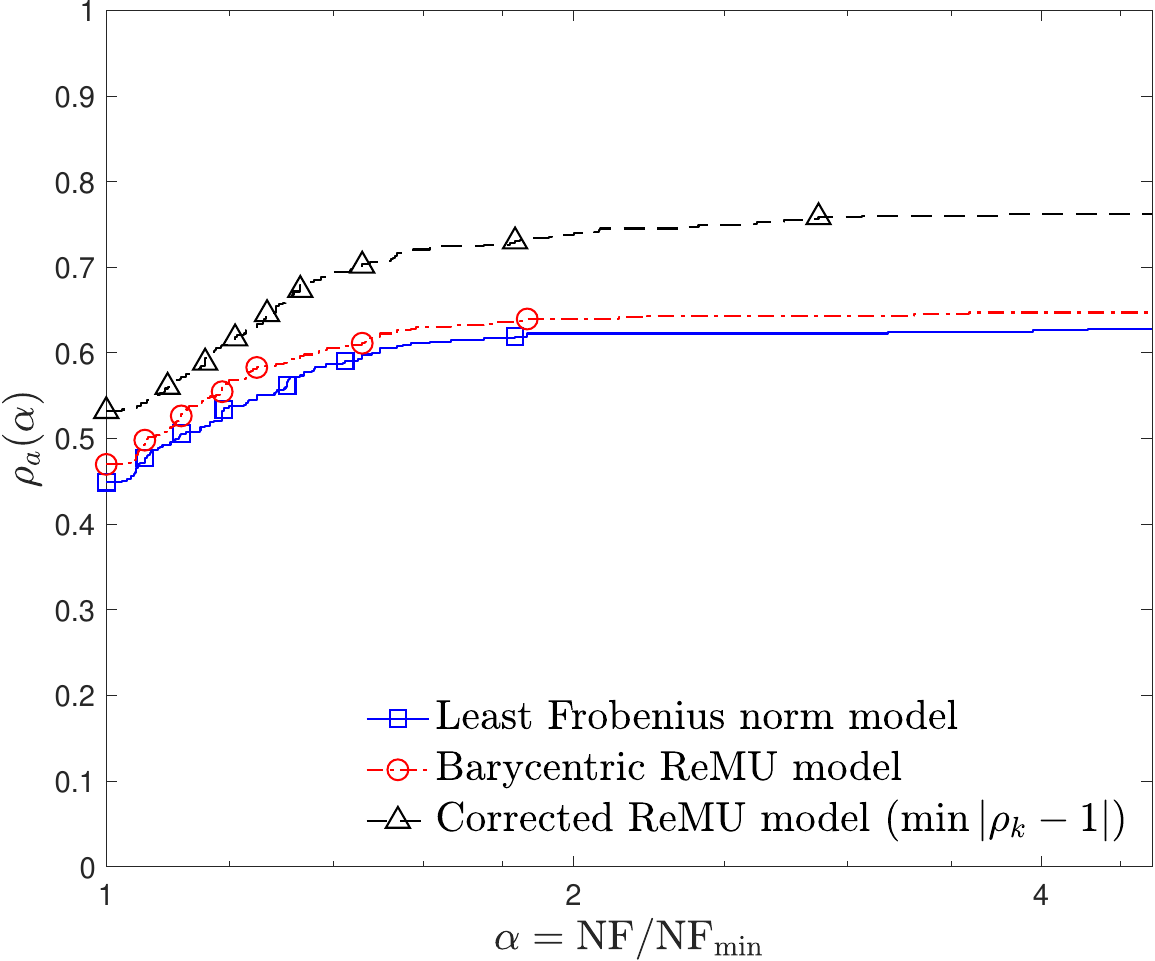} 
     \includegraphics[width=0.45\textwidth]{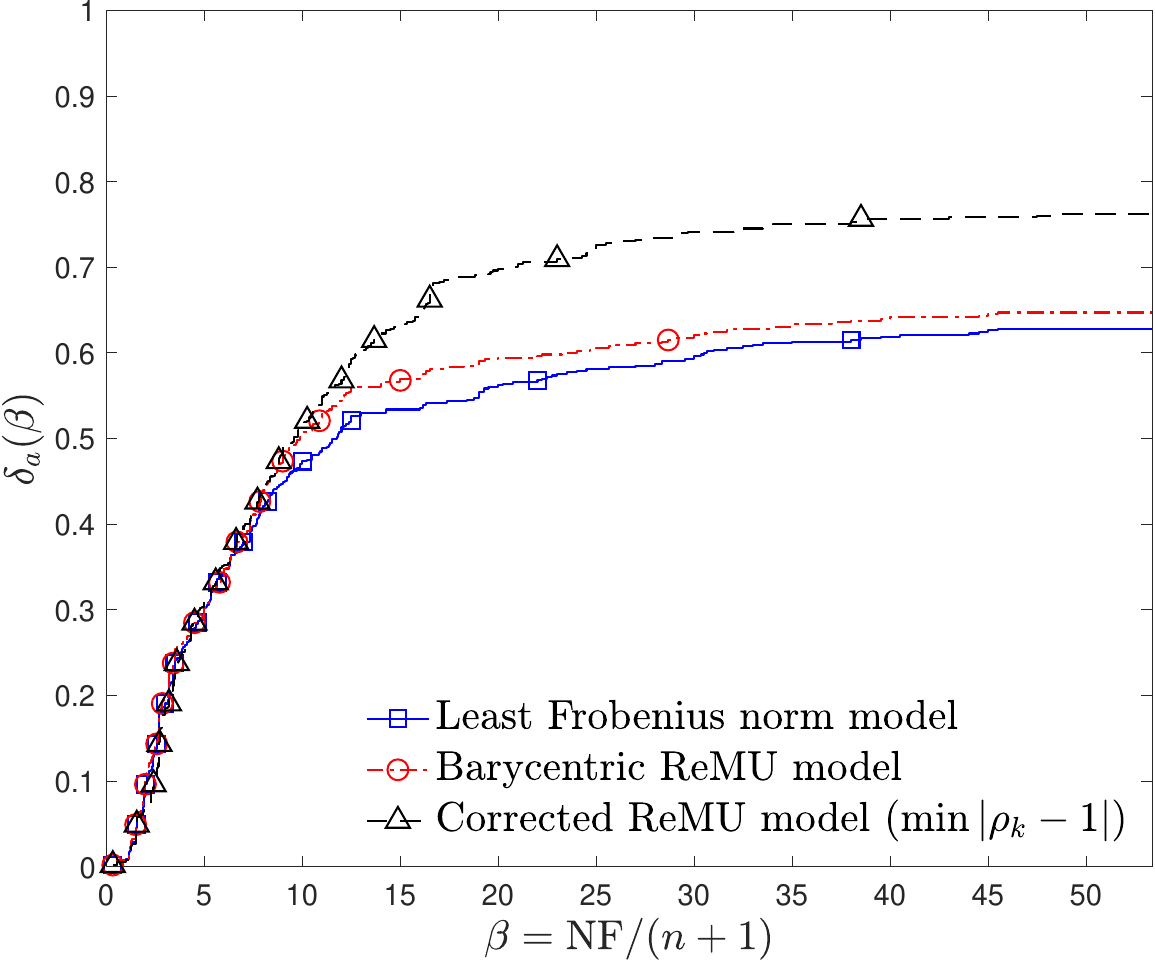} \\
         \includegraphics[width=0.45\textwidth]{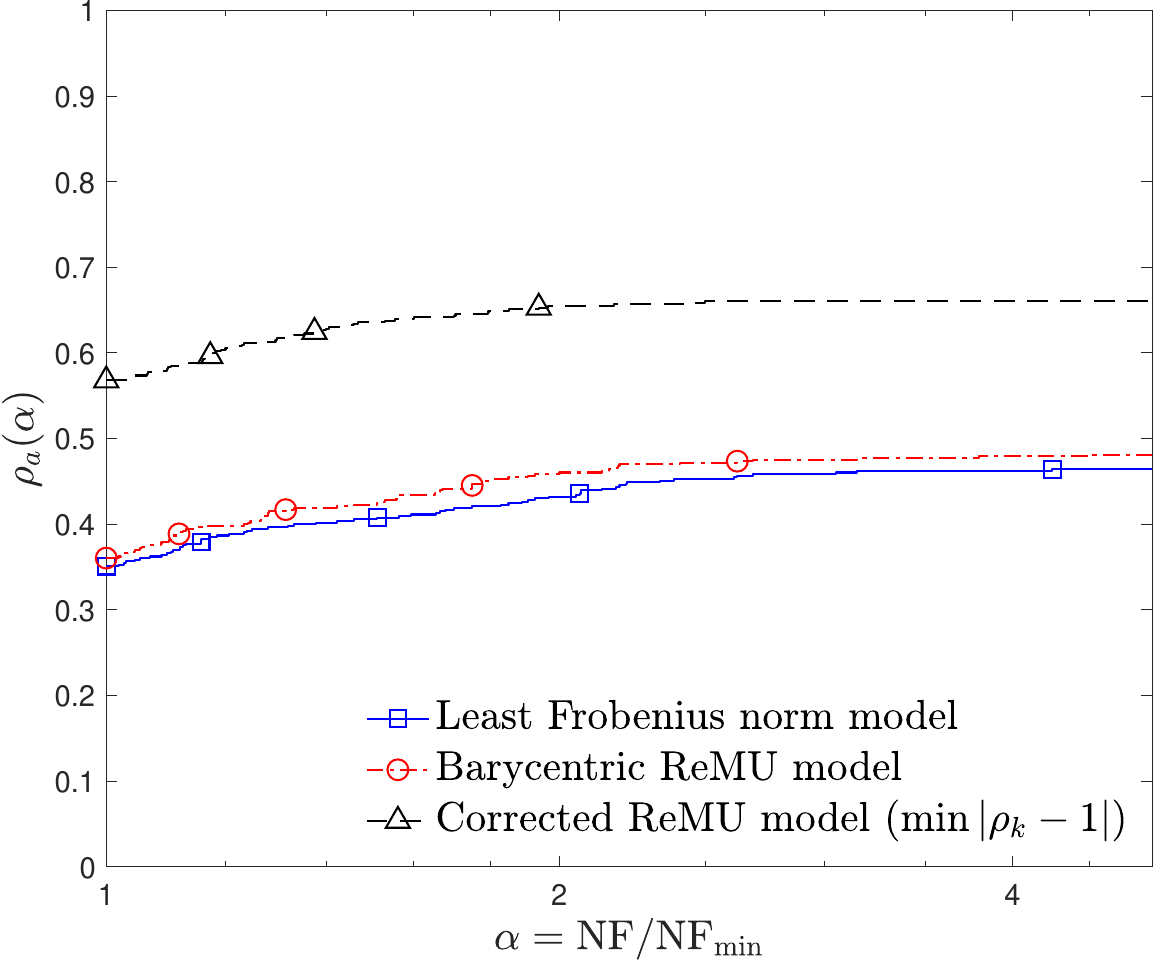}
         \includegraphics[width=0.45\textwidth]{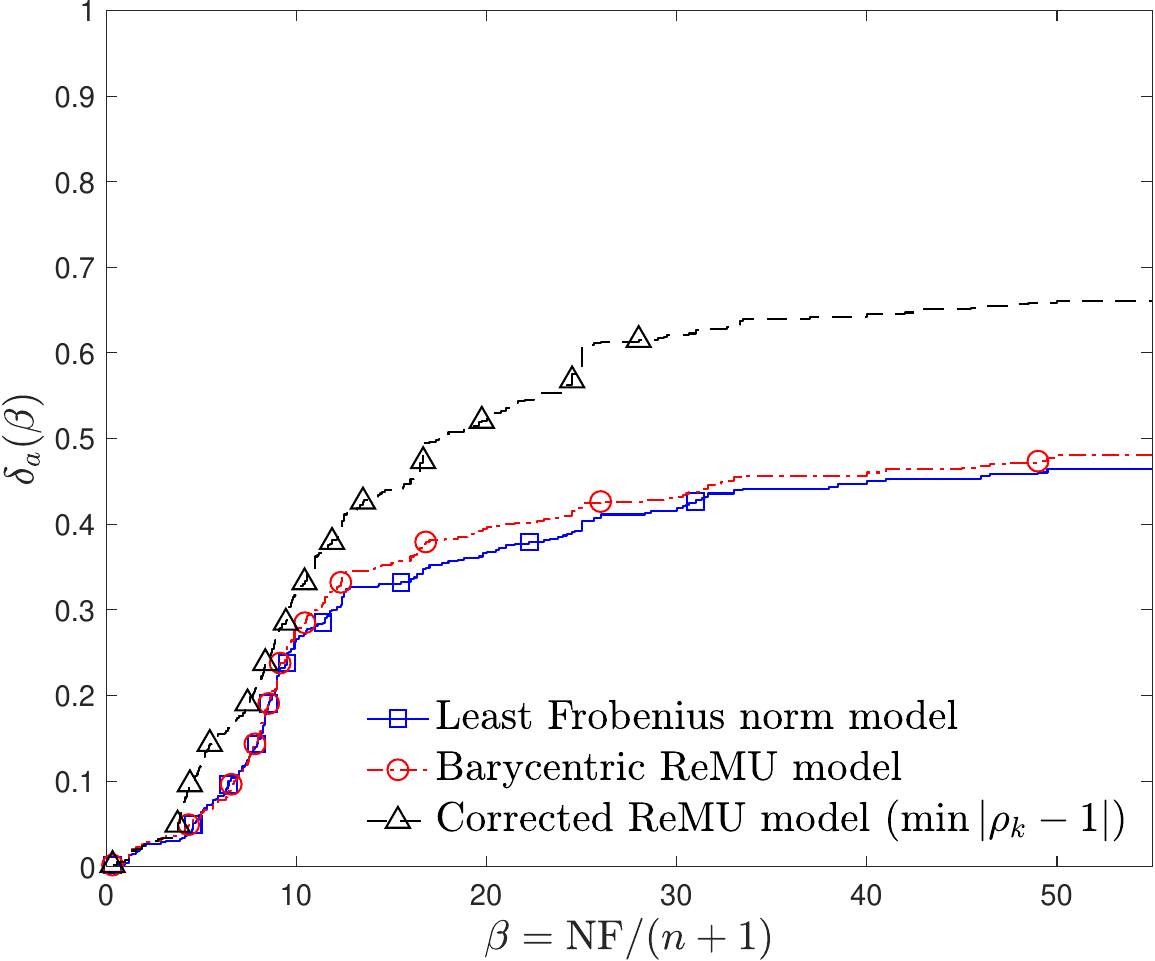} 
           \caption{Performance (1st column) and data (2nd column) profiles with accuracy levels \(\tau=10^{-1},10^{-2},10^{-3},10^{-6}\) (from top to bottom) for \(|\cX_k|=n+3\) interpolation points at each step; for the noisy problems,  \(\sigma=10^{-4}\).\label{perf-data-profile-D}}
\end{figure}

\figurename~\ref{perf-data-profile-A} to \figurename~\ref{perf-data-profile-D} display the performance profiles (1st column) and data profiles (2nd column) for different interpolation point sizes and noise levels. Each figure provides comparisons at accuracy levels \(\tau=10^{-1},10^{-2},10^{-3},10^{-6}\) (from top to bottom).  In particular, \figurename~\ref{perf-data-profile-A} and \figurename~\ref{perf-data-profile-B} use a noise level of \(\sigma=10^{-2}\), with \(2n+1\) and \(n+3\) interpolation points, respectively; \figurename~\ref{perf-data-profile-C} and \figurename~\ref{perf-data-profile-D} use a noise level of \(\sigma=10^{-4}\), with \(2n+1\) and \(n+3\) interpolation points, respectively.

The figures provide visual comparisons across different interpolation setups, noise levels, and accuracy levels. An overall conclusion is that {\ttfamily POUNDerS}' least Frobenius change models and the basic ReMU barycentric model perform comparably. However, in almost all cases, an improvement is seen when using the corrected ReMU models. This is a strong indication that such an approach effectively improves a solver's efficiency and robustness by combining and selecting between models; we expect that this advantage would further improve by enlarging the set of models considered in  \eqref{CReMU}. Since all of the compared approaches are based on the {\ttfamily POUNDerS} framework and the only difference is the model used at each iteration, this especially highlights the advantage of using the corrected ReMU model in the model-based trust-region methods.

We also examined cases when the corresponding KKT matrix was ill conditioned. This can happen because of numerical errors or because the matrix is potentially singular due to the geometry of the interpolation points. In particular, we report when the norm of the error in the KKT equations is larger than \(10^{-8}\)) during the minimization of 530 test functions, each with 100 function evaluations. For the methods tested in \Cref{Performance and date profiles for test set}, the warning rates for different ReMU models were as follows: \([1, 0, 0]\) at \(12\%\), \([0, 1, 0]\) at \(9\%\), \([0, 0, 1]\) at \(4\%\), \([1/3, 1/3, 1/3]\) at \(7\%\), \([1/2, 1/2, 0]\) at \(9\%\), \([0, 1/2, 1/2]\) at \(9\%\), and \([1/2, 0, 1/2]\) at \(6\%\). For the approach based on the framework of {\ttfamily POUNDerS}, the corresponding warning rates were: \([1, 0, 0]\) at \(10\%\), \([0, 1, 0]\) at \(7\%\), \([0, 0, 1]\) at \(5\%\), \([1/3, 1/3, 1/3]\) at \(4\%\), \([1/2, 1/2, 0]\) at \(8\%\), \([0, 1/2, 1/2]\) at \(4\%\), and \([1/2, 0, 1/2]\) at \(3\%\). Because these rates are relatively consistent across both frameworks, we conclude that geometry concerns associated with the ReMU models were minimal. These results indicate that in most cases the model satisfies the KKT conditions of the model subproblem \eqref{eq:weightproblem} to a high degree of accuracy.

\section{Conclusions}
\label{Conclusions}

This paper proposes an extension of the underdetermined quadratic model class: regional minimal updating (ReMU) quadratic interpolation models. These models offer flexibility through the weight coefficients appearing in the objective function of the interpolation subproblem. We define the KKT matrix distance, KKT matrix error, and the barycenter of the weight coefficient region.  We establish the barycenter of weight coefficient region of the ReMU models in the limit of a vanishing trust-region radius to further motivate our findings. Numerical performance comparisons through performance and data profiles elucidate some of the performance variability with different weight coefficients. We also propose a model-based derivative-free framework using ReMU models with corrected weight coefficients, and demonstrate that this strategy improves numerical performance, even when operating within another algorithm's framework. In the future, other comparisons and improvements of the weight coefficients of the ReMU models from other perspectives can be considered with the aim of discovering more properties of the underdetermined interpolation model in derivative-free optimization and further robustifying the selection of interpolation-based models.

\vspace{1cm}

\noindent {\bf Acknowledgments}  
This work was partially supported by Laboratory Directed Research and Development (LDRD) funding from Lawrence Berkeley National Laboratory and by the U.S.\ Department of Energy, Office of Science, Office of Advanced Scientific Computing Research applied mathematics program (SEAZOTIE) under Contract Number DE-AC02-05CH11231.

\bibliographystyle{abbrvnat}
\bibliography{smw-bigrefs}

\end{document}